\documentclass[11pt,a4paper,twosides,reqno]{amsart}
\usepackage{amsmath, amsthm, amssymb, enumerate, array, enumitem,mathtools,epsfig,graphics,graphicx,amsfonts,tabularx}
\usepackage{xcolor} \usepackage{hyperref} \usepackage{tikz-cd}
\usepackage{cancel} 
\usepackage{mathabx}
\usepackage[scr=boondox,cal=esstix]{mathalpha}
\usepackage[margin=3cm]{geometry}
\usepackage{comment}

\usepackage{enumitem}
\setlist{leftmargin=20pt}
\theoremstyle{plain}
\newtheorem{theorem}{Theorem}[section]
\newtheorem{lemma}[theorem]{Lemma}
\newtheorem{example}[theorem]{Example}
\newtheorem{corollary}[theorem]{Corollary}
\newtheorem{claim}[theorem]{Claim}
\newtheorem{addendum}[theorem]{Addendum}
\newtheorem{proposition}[theorem]{Proposition}

\newtheorem{question}[theorem]{Question}
\newtheorem{problem}[theorem]{Problem}
\newtheorem{fact}[theorem]{Fact}
\newtheorem{observation}[theorem]{Observation}

\theoremstyle{definition}
\newtheorem{definition}[theorem]{Definition}

\newtheorem{convention}[theorem]{Convention}
\newtheorem{notation}[theorem]{Notation}
\newtheorem{notations}[theorem]{Notations}
\newtheorem{remark}[theorem]{Remark}
\newtheorem{remarks}[theorem]{Remarks}

\newcommand{\Emat}[2]{E_{#1,#2}} 

   \def\DD{{\mathbb D}}

   \def\PP{{\mathbb P}}
 \def\RR{{\mathbb R}} \def\SS{{\mathbb S}} 
   
 \def\ZZ{{\mathbb Z}}
\DeclareMathOperator\M{Mod}
\DeclareMathOperator\Sp{Sp}

\def\cF{{\mathcal F}}

\title[BBMY2025]{Construction of Anosov Flows  \\ 
on Fibered Hyperbolic 3-Manifolds}
\author{François Béguin, Christian Bonatti, Biao Ma and Bin Yu}
\date{\today}

\begin{document}
\sloppy

\begin{abstract}
We prove that fibered hyperbolic $3$-manifolds carrying transitive Anosov flows are abundant. More precisely, for every $g\geq 2$, there is a finite index subgroup~$\Gamma$ of \( \operatorname{Mod}(S_g)/\operatorname{Tor}(S_g) \simeq \mathrm{Sp}(2g,\ZZ) \) so that every element of \(\Gamma\) has a representative \hbox{\(\varphi \in \operatorname{Mod}(S_g)\)} such that the mapping torus \hbox{\( M_\varphi := S_g \times [0,1]/(x,1) \sim (\varphi(x),0) \)} carries a transitive Anosov flow. The manifold $M_\varphi$ is hyperbolic for almost every element of~$\Gamma$. This shows in particular that, in the set of all fibered hyperbolic manifolds, the subset made of the manifolds carrying Anosov flows has positive density up to trivial linear monodromy. Moreover, the subgroup $\Gamma$ is defined by an explicit set of generators, and our construction yields many examples of simple fibered hyperbolic manifolds carrying Anosov flows.
\end{abstract}
\maketitle

\section{Introduction} 
\label{s.Int}
\subsection{Context.}\label{ss.context} 
The interplay between dynamics, topology, and geometry in dimension~$3$ is very well illustrated by the study of \textit{Anosov flows}. These flows, characterized by uniform hyperbolic behavior, play a important role in smooth dynamics and low-dimensional topology. A central problem is to understand which closed 3-manifolds admit Anosov flows. While such flows exist on many solvmanifolds and Seifert manifolds, their presence in the geometrically richest setting — \textit{closed hyperbolic 3-manifolds} — remains incompletely understood. 

The first construction of an Anosov flow on a closed hyperbolic 3-manifold is due to Goodman \cite{Goodman1983}. Subsequently, further examples were constructed by \cite{Fenley}. More recently, Bowden-Mann \cite{BM2022} and Béguin-Yu \cite{BY2024} have constructed arbitrarily large numbers of inequivalent Anosov flows on closed hyperbolic manifolds. 

Agol's virtual fibering theorem \cite{Agol2013} shows that every closed hyperbolic 3-manifold \(M\) admits a finite cover \(\widehat{M}\) that fibers over the circle. Consequently, if a closed hyperbolic \(3\)-manifold \(M\) admits an Anosov flow \(X\), then, by the virtual fibering theorem, one can lift the flow \(X\) to get an Anosov flow on a fibered hyperbolic manifold \(\widehat{M}\). However, the proof of the virtual fibering theorem provides no explicit description of the cover \(\widehat{M} \to M\). These considerations motivate the following two natural and compelling questions:

\begin{question}\label{q.fundq1}
Which closed fibered hyperbolic 3-manifolds carry Anosov flows? Do most (or almost all) closed fibered hyperbolic 3-manifolds carry an Anosov flow?
 \end{question}

\begin{problem}\label{q.fundq2}  
Can one build simple, explicit examples of closed fibered hyperbolic 3-manifolds that carry Anosov flows? 
\end{problem}

Here, by ``simple" fibered hyperbolic manifolds, we mean manifolds whose fiber has small genus (if possible $2$) and whose monodromy is given by a product of only a few Dehn twists along simple closed curves.

Recently, Matthew Hedden, Katherine Raoux and Jeremy Van Horn-Morris~\cite{HeddenRaouxVanHornMorris} announced that some fibered closed hyperbolic 3-manifolds -- including the manifolds obtained by zero surgery on hyperbolic $L$-space knots in $\SS^3$ -- do not admit Anosov flows with co-orientable stable and unstable foliations. This result can be regarded as significant progress towards Question~\ref{q.fundq1}.

In addition, the construction of Anosov flows on hyperbolic manifolds could lead progress on questions and conjectures concerning the existence of taut foliations. Let $X$ be an Anosov flow on a closed orientable $3$-manifold $M$. The weak stable foliation $\mathcal{F}^s$ of $X$ is a \emph{taut foliation}, since every leaf of $\mathcal{F}^s$ is non-compact. Moreover, when $\mathcal{F}^s$ is co-orientable, its Euler class satisfies $e(\mathcal{F}^s)=0$, as the Anosov vector field $X$ is tangent to $\mathcal{F}^s$. Therefore, the two questions above concerning Anosov flows are directly connected to a fundamental problem in foliation theory:

\begin{question}[Question $9.2$ of \cite{Yazdi2020}]
\label{q.fund-0-top} 
Which $3$-manifolds with a positive first Betti number admit a taut foliation with trivial Euler class?
\end{question}

Question~\ref{q.fund-0-top} is directly related to Thurston's \emph{Euler class-one conjecture}. Let $M$ be a closed orientable $3$-manifold equipped with a co-orientable taut foliation $\cF$. Thurston's foundational work \cite{Thurston1986} established that the dual Thurston norm \(x^\ast\) of the Euler class \(e(\cF)\) satisfies:
\[ x^\ast(e(\cF)) \in B^\ast(M) \quad \text{with} \quad x^\ast(e(\cF)) \leq 1, \]
where \(B^\ast(M)\) is the Thurston unit dual ball of \(M\). Thurston's \emph{Euler class-one conjecture} states that, for any integral class $a \in H^2(M;\ZZ)$ with $x^\ast(a) = 1$, there should exists a taut foliation $\cF$ on $M$ such that $e(\cF) = a$. After partial positive results by Gabai (\cite{Gabai1983}), the conjecture was disproved by Yazdi (\cite{Yazdi2020}; see also~\cite{Liu2024}). This motivates the following refined version of Question~\ref{q.fund-0-top}:

\begin{question}[Question $9.1$ of \cite{Yazdi2020} ]\label{q.fund-top} 
Which lattice points in $B^\ast(M)$ correspond to co-orientable taut foliations?
\end{question}

Progress on Questions~\ref{q.fundq1} or Problem~\ref{q.fundq2} directly yields progress on Question~\ref{q.fund-0-top}, hence on Question \ref{q.fund-top}.

\subsection{Main results of the paper.}\label{ss.con-pap}  

The purpose of the present paper is to provide partial answers to Question~\ref{q.fundq1} and Problem~\ref{q.fundq2}. For \( g \geq 2 \), we denote by \( S_g \) a genus \( g \) closed orientable surface, by \( \operatorname{Mod}(S_g) \) the mapping class group of \( S_g \), and by \( \operatorname{Tor}(S_g) \) the Torelli subgroup of \( \operatorname{Mod}(S_g) \), i.e. the subgroup of \( \operatorname{Mod}(S_g) \) made of the elements that act by identity on \( H_1(S_g, \mathbb{Z}) \). Recall that any choice of a basis of \( H_1(S_g, \mathbb{Z}) \) yields an identification of \( \operatorname{Mod}(S_g)/\operatorname{Tor}(S_g) \) with the symplectic group \( \operatorname{Sp}(2g, \mathbb{Z}) \). We will prove the following:

\begin{theorem}\label{t.findex}
There exists a finite index subgroup \(\Gamma\) of \( \operatorname{Mod}(S_g)/\operatorname{Tor}(S_g) \) such that every element of \(\Gamma\) has a representative \(\varphi \in \operatorname{Mod}(S_g)\) for which the mapping torus \hbox{\( M_\varphi := S_g \times [0,1]/(x,1) \sim (\varphi(x),0) \)} carries a transitive Anosov flow.
\end{theorem}

\begin{remark}\label{remark:t.findex}
\begin{enumerate}
 \item The finite index subgroup \( \Gamma \) provided by Theorem \ref{t.findex} is known explicitly (by a set of generators). It is not a principal congruence  subgroup of \( \operatorname{Sp}(2g, \mathbb{Z}) \). Nevertheless, we were able to prove that \( \Gamma \) contains the principal congruence subgroup of level $2^{2g-2}$ of \( \operatorname{Sp}(2g, \mathbb{Z}) \) (see Section~\ref{s.F-index-subgp}). This implies that the index of $[\mathrm{Sp}(2g,\ZZ):\Gamma]$ is bounded from above by $2^{8g^3}$ (which is for sure very far from being an optimal bound). It is then natural to wonder whether $\Gamma$ can be taken to be the full group $\operatorname{Sp}(2g, \mathbb{Z})$. 
 
 \item Note that it is not possible that for every element \(g\) of \(\Gamma\), and every \(\varphi_g \in \operatorname{Mod}(S_g)\) representing \(g\), the corresponding mapping torus \(M_{\varphi_g}\) admits a transitive Anosov flow. For instance, let  \(g\) to be the unit element of  \(\Gamma\) and \(\varphi_g= \text{Id}_{S_g}\), then \(M_{\varphi_g}\) is homeomorphic to \(S_g \times \SS^1\), which does not carry any Anosov flow by \cite{Ghys1984}.
\end{enumerate}
\end{remark}

Actually, Theorem \ref{t.findex} will appear as a consequence of a more precise and explicit statement. Consider the system of simple essential closed curves on \( S_g \) and the orientation of $S_g$ represented on Figure~\ref{f.abcd}:
\begin{figure}[htb]
    \centering
    \includegraphics[scale=0.3]{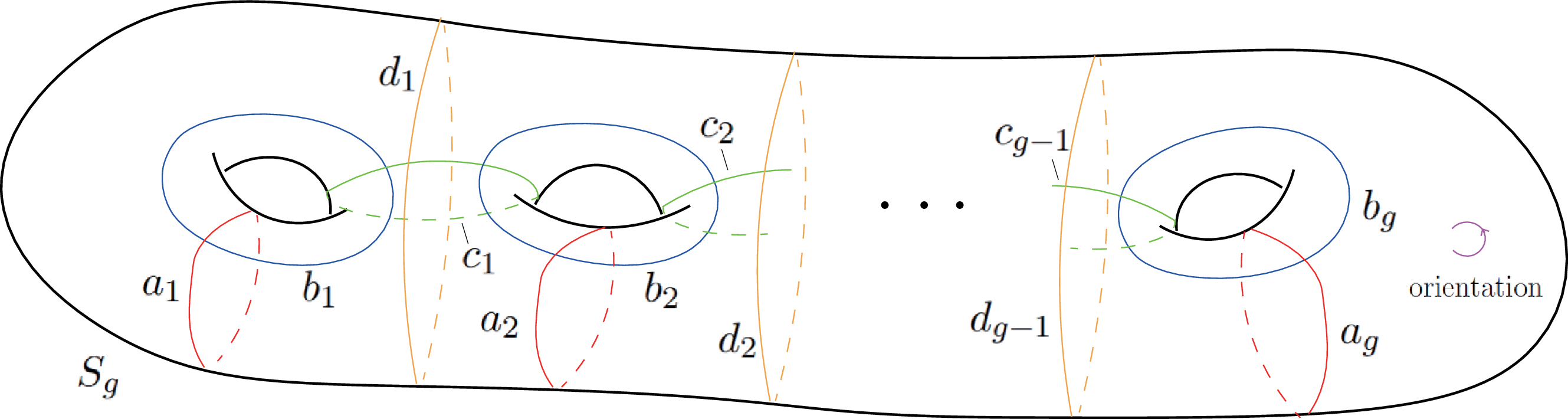}
    \caption{The simple closed curves  $ a_i,b_j, c_k, d_l$ on $ S_g $.}
    \label{f.abcd}
\end{figure}
Denote by \( \tau_{a_i} \), \( \tau_{b_j} \), \( \tau_{c_k} \), \( \tau_{d_l} \) the left Dehn twists along the curves \( a_i, b_j, c_k, d_l \) for \( 1 \leq i, j \leq g \) and \( 1 \leq k, l \leq g-1 \). 

\begin{remarks}
\begin{enumerate}
\item According to Lickorish-Wallace theorem, the Dehn twists $\{\tau_{a_i}\}_{1\leq i\leq g}$, $\{\tau_{b_j}\}_{1\leq j\leq g}$ and $\{\tau_{c_k} \}_{1\leq k\leq g-1}$ generate the mapping class group of the surface $S_g$. 
\item Since Dehn twists along disjoint curves commute, the Dehn twists $ \{\tau_{a_i}\}_{1\leq i\leq g}$, $\{\tau_{b_j}\}_{1\leq j\leq g}$, $\{\tau_{c_k}\}_{1\leq k\leq g-1}$, $\{\tau_{d_l}\}_{1\leq l\leq g-1}$ commute except for the pairs $(\tau_{a_{m}},\tau_{b_m})$, $(\tau_{b_n},\tau_{c_n})$, $(\tau_{b_{n+1}},\tau_{c_n})$, and $(\tau_{c_{n}},\tau_{d_n})$, with $1\leq m\leq g$ and $1\leq n\leq g-1$.
\end{enumerate}
\end{remarks}

We consider the following infinite family of products of Dehn twists
$$\mathcal{T}:=\left\{\tau_{d_1}^{\small{-2}}\dots\tau_{d_{g-1}}^{\small{-2}}
\tau_{b_1}^{q_1}\dots\tau_{b_{g}}^{q_g}
\tau_{c_1}^{r_1}\dots\tau_{c_{g-1}}^{r_{g-1}}
\tau_{a_1}^{p_1}\dots\tau_{a_{g}}^{p_g}\right\}_{\substack{r_1,\dots,r_{g-1}\,=\,0\; \mathrm{or}\;-2 \\
q_1,\dots,q_g\,\in\,\ZZ \hfill \\  p_1,\dots,p_q\,\in\,\ZZ\hfill }}$$
We insist on the fact that, in every element of the family $\mathcal{T}$,
the powers of the left Dehn twists along the curves $a_1,\dots,a_g,b_1,\dots,b_g$ are arbitrary integers, but the powers of the left Dehn twists along the curves $c_1,\dots,c_{g-1}$ can only be equal to $0$ or $-2$, and the powers of the left Dehn twists along the curves $d_1,\dots,d_{g-1}$ are equal to $-2$. Observe that the family $\mathcal{T}$ contains 
$$\mathcal{T}_0:=\{\widehat\tau_d\}\cup \{\widehat\tau_d\tau_{a_i}^{\pm 1}\}_{1 \leq i \leq g}\cup \{\widehat\tau_d\tau_{b_j}^{\pm 1}\}_{1 \leq j \leq g}\cup \{\widehat\tau_d\tau_{c_k}^{-2 }\}_{1 \leq k \leq g-1},$$
where 
$$\widehat\tau_d = \tau_{d_1}^{-2} \tau_{d_2}^{-2} \cdots \tau_{d_{g-1}}^{-2} .$$

\begin{theorem}
\label{t.generators}
If $\varphi$ is a non-trivial product of elements of $\mathcal{T}$ then the mapping torus 
\[ M_\varphi = S_g \times [0,1]_{/((x,1) \sim (\varphi(x),0))} \] 
carries a transitive Anosov flow.
\end{theorem}

In particular, for any non-trivial element $\varphi$ of the subsemigroup of $\mathrm{Mod}(S_g)$ generated by $\{\widehat\tau_d\}$, $\{\widehat\tau_d\tau_{a_i}^{\pm 1}\}_{1 \leq i \leq g}$, $\{\widehat\tau_d\tau_{b_j}^{\pm 1}\}_{1 \leq j \leq g}$ and $\{\widehat\tau_d\tau_{c_k}^{-2 }\}_{1 \leq k \leq g-1}$, the mapping torus of $\varphi$ carries a transitive Anosov flow.

The curves \(d_1, \ldots, d_{g-1}\) being homologous to zero, the product of Dehn twists $\widehat\tau_d=\tau_{d_1}^{-2} \tau_{d_2}^{-2} \cdots \tau_{d_{g-1}}^{-2}$ acts trivially on \(H_1(S_g, \mathbb{Z})\), i.e., belongs to the Torelli subgroup. Hence the following statement is a straightforward consequence of Theorem~\ref{t.generators}.

\begin{corollary}
\label{corollary:generators}
Let \(\Gamma\) be the subsemigroup of \(\text{Mod}(S_g)/ \operatorname{Tor}(S_g) \cong \operatorname{Sp}(2g,\mathbb{Z})\) generated by the equivalence classes of the Dehn twists
\[
\tau_{a_1}^{\pm 1}, \ldots, \tau_{a_g}^{\pm 1}, \tau_{b_1}^{\pm 1}, \ldots, \tau_{b_g}^{\pm 1}, \tau_{c_1}^{-2}, \ldots, \tau_{c_{g-1}}^{-2}.
\]
Then every element of $\Gamma$ has a representative \(\varphi \in \text{Mod}(S_g)\) so that the mapping torus of $\varphi$ carries an Anosov flow. 
\end{corollary}

The subsemigroup \(\Gamma\) in the statement of Corollary~\ref{corollary:generators}  is actually a subgroup (not just a subsemigroup) (Proposition~\ref{proposition:subgroup}), and this subgroup has finite index in \(\text{Sp}(2g, \mathbb{Z})\) (Proposition~\ref{proposition:finiteindex}). Theorem \ref{t.findex} will therefore appear as a consequence of Theorem~\ref{t.generators}, Proposition~\ref{proposition:subgroup} and Proposition~\ref{proposition:finiteindex}.

\subsection{Examples and discussion}

Theorem \ref{t.generators} allows us to construct simple examples of fibered hyperbolic 3-manifolds carrying Anosov flows. 

\begin{example}\label{example:homologycri}
Using the notations of Subsection~\ref{ss.con-pap}, consider the homeomorphism 
\[
\varphi = \tau_{d_1}^{-2} \tau_{c_1}^{-2} \tau_{a_1}\tau_{d_1}^{-2} \tau_{b_2}\tau_{b_1}
\]
of the genus 2 surface $S_2$. The matrix of the action of \(\varphi\) on \(H_1(S_2, \mathbb{Z})\) in the canonical basis $(a_1,a_2,b_1,b_2)$ is
$$\left(\begin{array}{cccc}
1 & 0 & 3 & -2\\
0 & 1 & -2 & 2\\
-1 & 0 & -2 & 2\\
0 & -1 & 2 & -1 
\end{array}\right).$$
The characteristic polynomial of this matrix is \(\chi_{\varphi}(x) = x^4 +x^3 -2x^2 +x + 1\). This polynomial is symplectically irreducible over \(\mathbb{Z}\), not cyclotomic, and not a polynomial in \(x^2\). According to \cite[Theorem 14.5]{FarbMargalit12}, this implies that the mapping class of \(\varphi\) is of pseudo-Anosov type. Therefore the mapping torus of \(\varphi\) is a hyperbolic manifold. Since $\varphi$ is a product of elements of the family $\mathcal{T}$ defined in Subsection~\ref{ss.con-pap}, Theorem~\ref{t.generators} asserts that this manifold carries a transitive Anosov flow. 
\end{example}

The method presented in Example \ref{example:homologycri} is quite general, and many other examples can be produced in a similar manner. One can also apply a result of Papadopoulos (\cite{Papa82}) to get more examples as explained below. 

\begin{observation}
Let $\varphi$ and $\psi$ be two products of elements of $\mathcal{T}$. Assume that the mapping class of $\varphi$ is of pseudo-Anosov type, and that $\psi(f^u)$ is not isotopic to $f^s$ where $f^s$ and $f^u$ are the stable and unstable foliations of $\psi$. Then, by \cite[Théorème 1]{Papa82}, for all sufficiently large $k$, the composition \(\varphi^k \circ \psi\) is pseudo-Anosov, hence the mapping torus of \(\varphi^k \circ \psi\) is a fibered hyperbolic manifold . As $\varphi^k \circ \psi$ is a product of elements of $\mathcal{T}$, Theorem \ref{t.generators} ensures that this fibered hyperbolic manifold carries a transitive Anosov flow.
\end{observation}

As we noted in Remark \ref{remark:t.findex}, for $\varphi_g = \text{Id}_{S_g}$, $M_{\varphi_g}$ does not carry any Anosov flow. To produce a monodromy in the Torelli group, we can take $\varphi_g = \widehat\tau_d$, then by Theorem \ref{t.generators}, $M_{\phi_g}$ admits an Anosov flow and $\phi_g \in \operatorname{Tor}(S_g)$. Such an Anosov flow is obtained by performing Dehn-Fried surgeries on the geodesic flow on $S_g$ along the lifts of $d_i$ (cf. Section \ref{s.cover}). However, $ \widehat\tau_d$ is not pseudo-Anosov. This raises the following interesting question.

\begin{question}
    Does there exist a pseudo-Anosov mapping class $\varphi \in \operatorname{Tor}(S_g) $ so that $M_{\varphi}$ admits a transitive Anosov flow?
\end{question}

Let us now explain the general meaning and some implications of Theorem \ref{t.findex}. This theorem implies that, for any fixed genus $g \geq 2$, among all fibered hyperbolic 3-manifolds with genus $g$ fibers, those that carry transitive Anosov flows are ``abundant", when abundance is measured in terms of the symplectic representation of their monodromy. Indeed, by a result of Rivin (\cite{Rivin08}), for a generic coset in \( \operatorname{Mod}(S_g) / \operatorname{Tor}(S_g) \), every element \( \varphi \) of this coset is of pseudo-Anosov type (hence \( M_\varphi \) is hyperbolic). Here \emph{generic} means almost every product of the elements of a fixed finite symmetric generating set of \(  \operatorname{Sp}(2g, \mathbb{Z})\) (see \cite[Theorem 8.2]{Rivin08}). In this sense, Theorem \ref{t.findex} implies that, for a fixed genus $g$, the set of fibered hyperbolic 3-manifolds with genus $g$ fibers carrying Anosov flows has positive density within the set of all fibered 3-manifolds with genus $g$ fibers considered ``up to trivial linear monodromy" (that is considered up to a change of monodromy having a trivial action in homology).

Combined with the virtual fibering theorem and the results announced in \cite{HeddenRaouxVanHornMorris}, these facts make the following open problem posed by Potrie even more appealing.

\begin{question}\label{que:virtualAnosov}\cite{Potrie}
    Does every closed hyperbolic 3-manifold have a finite cover carrying an Anosov flow? 
\end{question}

\subsection{Organization of the paper}

Dehn-Fried surgeries are classical and important construction tools of Anosov flows in dimension~3. They will play a central role in our paper. In section \ref{section:Dehn-Fried-fibered}, we carefully recall the definition of these surgeries and study which Dehn-Fried surgeries preserve the property of the manifold being fibered. 

In Section \ref{s.g-2-fiber}, we construct a fibered (non-hyperbolic) 3-manifold with genus 2 fibers and a very simple monodromy that carries an Anosov flow with many horizontal periodic orbits, \emph{i.e.} periodic orbits included in a fiber up to isotopy. This is the main technical part of this paper. The construction consists of considering the geodesic flow on the unit tangent bundle of a genus 2 closed surface, and performing Dehn-Fried surgeries on two specific orbits of this flow. The main difficulty is to prove that the resulting manifold is fibered. This is done by cutting the closed genus 2 surface into two halves --- each of them being a once-puncture torus ---, carefully defining an explicit family of partial sections of the unit tangent bundle over each of the two halves, and proving that, after the Dehn-Fried surgeries, each section on the left-half can be glued to the corresponding section on the right half, yielding a family of closed genus 2 surfaces in the surgeried manifold that will be the fibers of a fibration over the circle. 

In section~\ref{s.cover}, we generalize the result of Section~\ref{s.g-2-fiber} to an arbitrary genus $g\geq 2$, by using a finite covering argument. 

The proof of Theorem~\ref{t.generators} is completed in Section~\ref{s.proof}. This is done by considering the specific Anosov flows on the fibered manifold constructed in Section~\ref{s.g-2-fiber} (for genus 2 fibers) and Section~\ref{s.cover} (for arbitrary genus fibers), taking a finite cover, and performing Dehn-Fried surgeries on horizontal orbits. Such Dehn-Fried surgeries preserve the property of the underlying manifolds being fibered, but will change the monodromy of the fibrations, allowing to obtain fibered manifolds with complicated monodromies carrying Anosov flows. 

In the last section~\ref{s.F-index-subgp} of the paper, we explain why Theorem~\ref{t.findex} follows from Theorem~\ref{t.generators}. In other words, we prove that, for every $g$, the sub-semi-group of $\mathrm{Sp}(H_1(S_g,\ZZ))\simeq \mathrm{Sp}(2g,\ZZ)$ generated by the classes in $\mathrm{Sp}(2g,\ZZ)$ of the Dehn twists allowed by Theorem~\ref{t.generators} is actually a subgroup of finite index in $\mathrm{Sp}(2g,\ZZ)$. This is done by proving that our sub-semi-group contains the time $2^{g-1}$ of the one-parameter subgroups generated by the root vectors associated to the classical root system of $\mathrm{Sp}(2g)$, allowing to apply results of Bass, Milnor, Serre and Tits which imply that our sub-semi-group contains the level $2^{2g-2}$ principal congruence of $\mathrm{Sp}(2g,\ZZ)$.

\section{Dehn-Fried surgeries in fibered manifolds}
\label{section:Dehn-Fried-fibered}
Dehn-Fried surgeries, introduced by David Fried in~\cite{Fried},
play a central role in 3-dimensional (pseudo-)Anosov flow theory, and more specifically in our proof of Theorem~\ref{t.generators}. The purpose of the present section is to find  sufficient conditions for a Dehn-Fried surgery to preserve the existence of a fibration over the circle. In other words, we will address the following problem: 

\begin{problem}
\label{problem:Dehn-Fried-fibered}
Let $M$ be a closed $3$-manifold which admits a fibration $p:M\to\SS^1$ and carries a transitive Anosov vector field $X$ so that one periodic orbit $\gamma$ of $X$ is contained in a fiber of the fibration $p$. Let $M'$ and $X'$ be the manifold and the Anosov flow obtained by performing an index $k$ Dehn-Fried surgery on the orbit~$\gamma$. Find sufficient conditions ensuring that the manifold $M'$ also fibers over the circle. If it be so, describe the monodromy of the fibration of $M'$. 
\end{problem}

As it can be expected, we will see that $M'$ still fibers over the circle in the case where the local stable manifold of $\gamma$ is ``parallel" to the fibers. More surprisingly, we will prove that the same is true in some specific cases where the local stable manifold of $\gamma$ is ``twisted" with respect to the fibers. 

Our first task, in Subsection~\ref{subsection:blow-up}, will be to recall the definition of the blow-up of a manifold along a periodic orbit is and, more importantly, to fix some orientation conventions. In Subsection~\ref{subsection:Dehn-Fried}, we will recall the precise definition of an index $k$ Dehn-Fried surgery. 
In Subsection~\ref{subsection:twist}, we will introduce an integer measuring the ``torsion" of the local stable manifold of a hyperbolic periodic orbit of a flow with respect to the fibers of the fibration. 
In Subsection~\ref{subsection:Dehn-Fried-versus-Dehn} and~\ref{subsection:Dehn-fibered}, we will explain the effect of a Dehn-Fried surgery on the topology of the underlying manifold, first in terms of (classical topological) Dehn surgery, then in terms of a Dehn twist in a surface. We will finally be ready to answer problem~\ref{problem:Dehn-Fried-fibered} in Subsections~\ref{subsection:Dehn-Fried-fibered-I},~\ref{subsection:Dehn-Fried-fibered-II} and~\ref{subsection:multi-Dehn-Fried}. 

\subsection*{Some notations.} All along the paper, $\SS^1$ stands for $\RR/2\pi\ZZ$, and $\DD$ stands for the open unit disc in $\RR^2$ with respect to the standard Euclidean metric. Given a real vector space $V$, we denote by $\PP^+(V)$ the positive projectivization of $V$, \emph{i.e.} the quotient of $V-\{0\}$ by positive homotheties, \emph{i.e.} the set of oriented directions in $V$. More generally, if $E$ is a vector bundle, we denote by $\PP^+(E)$ the bundle whose fibers are the (positive) projectivization of the fibers of $V$.

\subsection{Blowing-up a manifold along a periodic orbit}
\label{subsection:blow-up}

Let $M$ be an oriented closed $3$-manifold, $X$ be a vector field on $M$ and $\gamma$ be a periodic orbit of $X$. Blowing-up the periodic orbit $\gamma$ roughly consists of taking polar coordinates in the direction transverse to $\gamma$. More precisely, consider a tubular neighbourhood $V$ of $\gamma$ and a diffeomorphism  $\zeta=(\theta,x,y) : V\to \SS^1\times\DD$ such that $\zeta(\gamma)=\SS^1\times \{0_{\RR^2}\}$. The \emph{blow-up of $M$ along $\gamma$} is the $3$-manifold with boundary (well-defined up to diffeomorphism)
$$\widehat M=\widehat M_\gamma := \left((\SS^1\times [0,1)\times\SS^1)\sqcup(M-\gamma)\right)_{/(\theta,\rho,\varphi)\sim \zeta^{-1}(\theta,\rho\cos(\varphi),\rho\sin(\varphi))\;\mathrm{for}\;\rho>0}.$$
By construction, there is a canonical identification of $M-\gamma$ with $\mathrm{int}(\widehat M)$, and $\partial\widehat M$ is a two-dimensional torus. The construction also provides some coordinates 
$$(\theta,\rho,\varphi):\widehat V\to \SS^1\times [0,1)\times \SS^1$$
defined on a collar neighbourhood $\widehat V$ of $\partial M$ in $M$. Actually, $(\rho,\varphi)$ are nothing but the polar coordinates defined by $(x,y)=(\rho\cos\varphi,\rho\sin\varphi)$. The torus $\partial\widehat M$ corresponds to the set $(\rho=0)$. There is natural projection $\widehat M\to M$, which is ``the identity" on $\mathrm{int}(\widehat M)$ and maps the circle $\{\theta_0\}\times\{0\}\times\SS^1$ in $\partial\widehat M$ to the point $\zeta^{-1}(\theta_0,0,0)$ in $\gamma$. The restriction of this projection defines a $\SS^1$-bundle $\partial\widehat M\to\gamma$.

The torus $\partial\widehat M$ naturally identifies with the projectivized normal bundle of the orbit $\gamma$
$$N_{\gamma}:=\partial\widehat M\simeq \PP^+\left(TM/\RR.X\right)_{|\gamma},$$
providing a more intrinsic definition of $\partial\widehat M$. 
The identification is obtained by mapping the point $(\theta,\varphi)\in\partial M$ to the point $m$ of coordinates $(\theta,0,0)\in\gamma$ and the class of the vector $v=\cos(\varphi)\partial_x+\sin(\varphi)\partial_y$ in $\PP^+(T_mM/\RR.X(m))$. Note that the identification is compatible with the $\SS^1$-bundle structures of $\partial\widehat M$ and $\PP^+\left(TM/\RR.X\right)_{|\gamma}$. See Figure \ref{f.blowup-gamma}.

Now we set some conventions of orientations. First observe that, since $M-\gamma$ canonically identifies with $\mathrm{int}(\widehat M)$, the orientation of $M$ induces a canonical orientation on $\widehat M$. 

\begin{convention}
\label{convention:orientation-torus}
We endow the torus $\partial\widehat M\simeq N_{\gamma}$ with the orientation as a boundary of the oriented $3$-manifold $\widehat M$.
\end{convention}

We observe that the orientation of $M$ and the orientation of $\gamma$ induce a canonical orientation on the fibers of $TM/\RR.X$ (here we use that $M$ is odd-dimensional), which, in turn, induces an orientation on the fibers of $\PP^+(TM/\RR.X)$ (just in the same way as the standard orientation of the plane induces the trigonometric orientation on the circle). So the fibers of $N_{\gamma}$ receive a canonical orientation.

\begin{definition}
\label{definition:canonical-meridian}
We define the \emph{canonical (oriented) meridian} of the torus $N_{\gamma}=\PP^+(TM/\RR.X)\simeq \partial\widehat M$, denoted by $\mu$, as the class in $H_1(N_{\gamma},\ZZ)$ of the fibers of $N_{\gamma}\to\gamma$ equipped with their orientation canonically induced by those of $M$.
\end{definition}

The coordinates $(\theta,\varphi)$ provide more tractable characterizations of these orientations. Indeed, one easily checks that: 

\begin{fact}
\label{fact:orientations}
If the local coordinates $(\theta,x,y):M\to\SS^1\times M$ near $\gamma$ are such that $(\partial_\theta,\partial_x,\partial_y)$ is a direct frame of $M$, and we define $(\rho,\varphi)$ by setting $(x,y)=(\rho\cos(\varphi),\rho\sin(\varphi))$, then:
\begin{itemize}
    \item the orientation of the torus $N_{\gamma}=\partial\widehat M$ is chosen such that $(\partial_\theta,\partial_\varphi)$ is a direct frame;
    \item the canonical meridian of $N_{\gamma}=\partial\widehat M$ is oriented by the vector field $\partial_\theta$.
\end{itemize}
\end{fact}

Finally, we observe that the orientation of the curve $\gamma$ induces a canonical orientation on the sections of $\partial\widehat M\simeq N_{\gamma}$:

\begin{definition}
\label{definition:dynamical-orientation}
We define the \emph{dynamical orientation} of (the image of) a section of the $\SS^1$-bundle $\partial\widehat M\simeq N_{\gamma}\to \gamma$ as the orientation obtained by pushing forward the dynamical orientation of the orbit $\gamma$.
\end{definition}

\begin{observation}
\label{remark:intersection-number}
Note that, with our conventions, the intersection number of any dynamically oriented section $\sigma$ with the canonical meridian $\mu$ is $\mathrm{Int}(\sigma,\mu)=+1$.
\end{observation}

\begin{figure}
  \centering
    \includegraphics[scale=0.32]{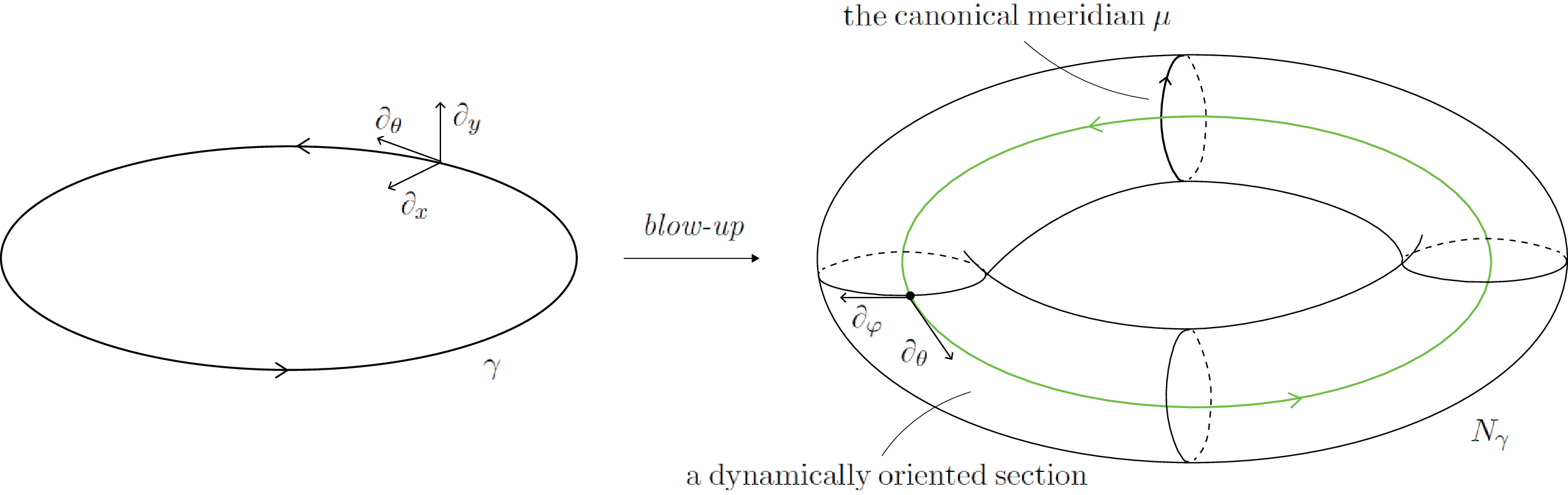}
\caption{The torus $N_{\gamma}\simeq\partial\widehat M$, the fibers of its projection onto $\gamma$, the canonical meridian $\mu$, a dynamically oriented section and the frame $(\partial_\theta,\partial_\varphi)$.}
\label{f.blowup-gamma}
\end{figure}

\begin{remark}
\label{remark:orbit-vs-closed-curve}
We restricted ourselves to the case where $\gamma$ is a periodic orbit of a vector field $X$ on $M$ since this will be the natural setting for Dehn-Fried surgeries. Nevertheless, all the content of Subsection~\ref{subsection:blow-up} generalizes to the more general setting where $\gamma$ is any oriented simple closed curve in $M$. One just needs to replace the direction $\RR.X$ by $\RR.\dot{\gamma}$ in the definition of the projectivized normal bundle $N_\gamma$. 
\end{remark}

\subsection{Index $k$ Dehn-Fried surgeries}
\label{subsection:Dehn-Fried}

We shall now recall what is an index $k$ Dehn-Fried surgery  on a periodic orbit of a transitive Anosov flow. Details and proofs of the assertion made below can be found \emph{e.g.} in Shannon's PhD thesis~\cite{shannon2020dehn}.  

Let $M$ be a closed oriented $3$-manifold, $X$ be a transitive Anosov vector field on $M$, and $\gamma$ be a periodic orbit of $X^t$ with positive multipliers. Let $\widehat M$ be the blow-up of $M$ along~$\gamma$. Recall that the boundary torus of $\partial \widehat M$ identifies with the projectivized normal bundle $N_{\gamma}=\PP^+(TM/\RR.X)_{|\gamma}$. Also recall that the fibers of the projection $N_{\gamma}\to\gamma$ inherit of a canonical orientation, and that the class in $H_1(N_{\gamma},\ZZ)$ of the oriented fibers is called the \emph{canonical meridian} of $N_{\gamma}$ and denoted by $\mu$ (see Definition~\ref{definition:canonical-meridian}). 

The flow $X^t$ being $C^1$, it lifts to a continuous flow $\widehat X^t$ tangent to a $C^0$ vector field: the restriction of $\widehat X^t$ to $N_{\gamma}$ is canonically defined as the action of the derivative of $X^t$ on the projectivized normal bundle of $\gamma$. The hyperbolicity of the orbit $\gamma$ yields a complete description of this action, and therefore of the dynamics of $\widehat X^t$ on $N_{\gamma}$. The weak stable (resp. unstable) plane of along $\gamma$ defines two opposite $\widehat X^t$-invariant sections\footnote{Recall that the multipliers of $\gamma$ were assumed to be positive. We use this hypothesis here to ensure that we indeed get genuine sections rather than "two-points-valued sections".} of $\gamma$. This yields four parallel periodic orbits of $\widehat X^t$. Of course, the flow orientations of these periodic orbits coincides with their dynamical orientation as sections of $N_{\gamma}$. These four periodic orbits divide the torus $N_{\gamma}$ into four annuli, each of which having a boundary component corresponding to the stable direction and a boundary component corresponding to the unstable direction. The other orbits of $\widehat X^t$ are contained in one of the above mentioned annuli, spiraling from the boundary component corresponding to the stable direction to the boundary component corresponding to the unstable direction. See Figure~\ref{f.Xt-widehat}.

\begin{definition}
\label{definition:canonical-longitude}
The class in $H_1(N_{\gamma},\ZZ)$ of the oriented periodic orbits of $\widehat X^t$ (\emph{i.e.} of the dynamically oriented sections of $N_{\gamma}$ corresponding to the stable/unstable directions) will be called the \emph{canonical longitude} of $N_{\gamma}$. We will denote it by $\lambda^{s}$. 
\end{definition}

\begin{figure}
\centering
    \includegraphics[scale=0.32]{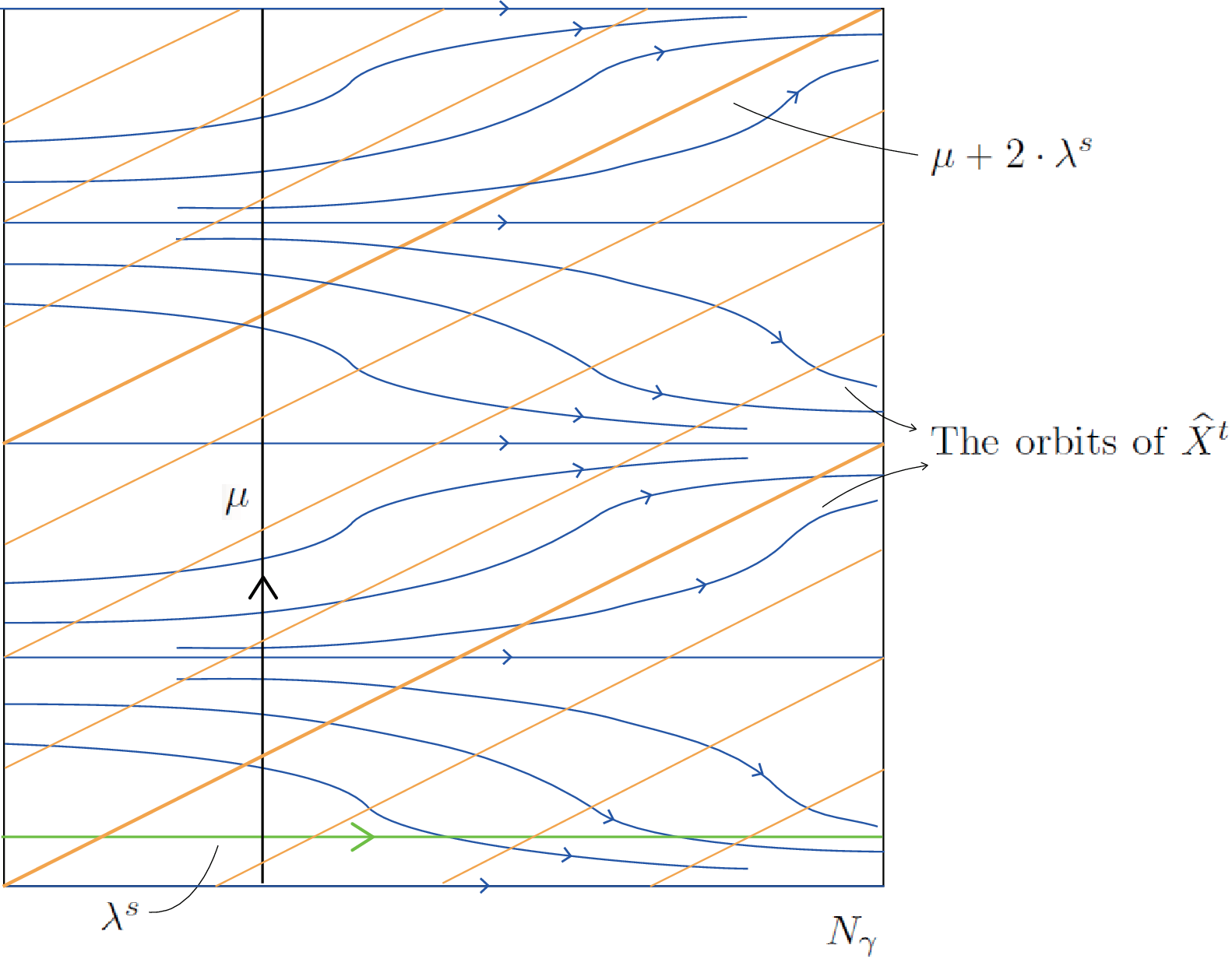}
\caption{The orbits of $\widehat X^t$ on the torus $N_{\gamma}$ (in blue). The canonical meridian $\mu$ (in black) and the canonical longitude $\lambda^{s}$ of $N_{\gamma}$ (in green), and a fibration whose fibers are in the homology class $\mu+2\cdot\lambda^{s}$ and transverse to the orbits of~$\widehat X$ (in tangerine).} 
\label{f.Xt-widehat}
\end{figure}

In view of the dynamics of $\widehat X_\gamma^t$ on $N$, for every $k\in\ZZ$, one can find a new smooth fibration on $p':N_{\gamma}\to\gamma$, so that the homology class of the fibers is $\pm(\mu+k\cdot\lambda^s)$ (see Figure~\ref{f.Xt-widehat} for the case $k=2$) and so that $\widehat X_\gamma^t$ is transverse to the fibers of $p'$. We consider the closed $3$-manifold $M':=\widehat M/p'$ (here the quotient means that we crash each fiber of $p'$ to a point). The oriented one-dimensional foliation of $\widehat M$ by the orbit of $\widehat X^t$ projects to a continuous oriented one-dimensional foliation of $M'$ tangent to a continuous half-line field (the continuity follows from the transversality of the fibers of $p'$ to the foliation on $\widehat M$). One may choose a continuous non-singular vector field $X'$ tangent the leaves of this foliation. By construction, the projection $\gamma'$ of the torus $N_{\gamma}$ in $M'$ is a periodic orbit of the vector field~$X'$, and there is a canonical orbit equivalence 
$$i:\left(M'-\gamma',X'_{|M'-\gamma'}\right)\to \left(M-\gamma,X_{|M-\gamma}\right),$$ 
(which can be thought as ``the identity" since it is obtained by composing the natural identifications of $M-\gamma$ and $M'-\gamma'$ with $\mathrm{int}(\widehat M)$). The vector field $X'$ enjoys many properties of an Anosov vector field, but fails to be $C^1$. Nevertheless, Shannon has proved that $X'$ is orbit equivalent to a genuine transitive Anosov vector field, and that $(M',X')$ only depend on $M$, $X$, $\gamma$ and $k$ up to topological equivalence, see \cite{shannon2020dehn}. 

We will apply the construction process above in the slightly more general situation where the orbits of $\widehat X_\gamma^t$ are not genuinely transverse, but only topologically transverse to the fibers of the fibration $p':N_{\gamma}\to\gamma$. In this situation, the foliation by the orbits of $\widehat X^t$ still projects to a continuous oriented foliation on $M'$. The leaves of this foliation are not necessarily tangent to a continuous half-line field, but still are the orbits of a topological Anosov flow, and Shannon's theorem still applies. 

\begin{definition}
\label{definition:Dehn-Fried}
The surgery transforming $(M,X^t)$ to $(M',X')$ is the index $k$ Dehn-Fried surgery on the orbit $\gamma$.  
\end{definition}

\subsection{Twist number of the local stable manifold of a periodic orbit with respect to a surface}
\label{subsection:twist}

Let $M$ be an oriented $3$-manifold and $\gamma$ be a periodic orbit with positive multipliers of a flow $X^t$ on $M$. We denote by $N_\gamma$ the projectivized normal bundle of~$\gamma$. 

Recall that the orientation of the manifold $M$ induces an orientation of the torus $N_\gamma$ (see Convention~\ref{convention:orientation-torus}). Also recall that the image of every section of the bundle $N_\gamma\to\gamma$ is equipped with a canonical orientation, called the \emph{dynamical orientation}, obtained by pushing forward the dynamical orientation of the orbit $\gamma$ (see Subsection~\ref{subsection:blow-up}). As a consequence, the homological intersection of (the images of) any two sections of $N_\gamma$ is well-defined. 

Every orientable surface $S$ containing the orbit $\gamma$ induces a pair of opposite sections $\pm\sigma_S$ of the circle bundle $N_\gamma\to\gamma$. Observe that, if $S$ and $S'$ are two orientable surfaces containing $\gamma$, then the intersection number $\mathrm{Int}(\sigma_{S'},\sigma_{S})$ is unchanged if one replaces $\sigma_{S'}$ and/or $\sigma_{S}$ by its opposite (since any section $\sigma$ is homologous to $-\sigma$). This gives a well-defined meaning to the following definition:

\begin{definition}
\label{definition:twist-1}
Given two orientable surfaces $S$ and $S'$ containing the orbit $\gamma$, the \emph{twist number of $S'$ with respect to $S$} is defined as
$$\mathrm{Twist}_\gamma(S',S):=\mathrm{Int}(\sigma_{S'},\sigma_S)$$
where $\sigma_S'$ (resp. $\sigma_S$) is any of the two opposite sections of the projectivized normal bundle $N_\gamma=\PP^+(TM/\RR.X)_{|\gamma}$ defined by the tangent planes of $S'$ (resp. $S$). 

In particular, if $S$ is an orientable surface containing a hyperbolic periodic orbit  $\gamma$ with positive multipliers, the \emph{twist number of the local stable manifold of $\gamma$ with respect to~$S$} is $$\mathrm{Twist}_\gamma\left(W^s_{\mathrm{loc}}(\gamma),S\right):=\mathrm{Int}\left(\lambda^s,\sigma_S\right)$$
where $\lambda^s$ is the canonical longitude of $N$ (Definition~\ref{definition:canonical-longitude}). If $S$ is a fiber of a fibration $p:M\to\SS^1$, then $\mathrm{Twist}_\gamma\left(W^s_{\mathrm{loc}}(\gamma),S\right)$ is called the twist number of $W^s_{\mathrm{loc}}(\gamma)$ with respect to $p$, and $W^s_{\mathrm{loc}}(\gamma)$ is said to be \emph{horizontal} (with respect to $p$) if $\mathrm{Twist}_\gamma\left(W^s_{\mathrm{loc}}(\gamma),S\right)=0$.
\end{definition}

\begin{figure}[htb]
    \centering
    \includegraphics[scale=0.35]{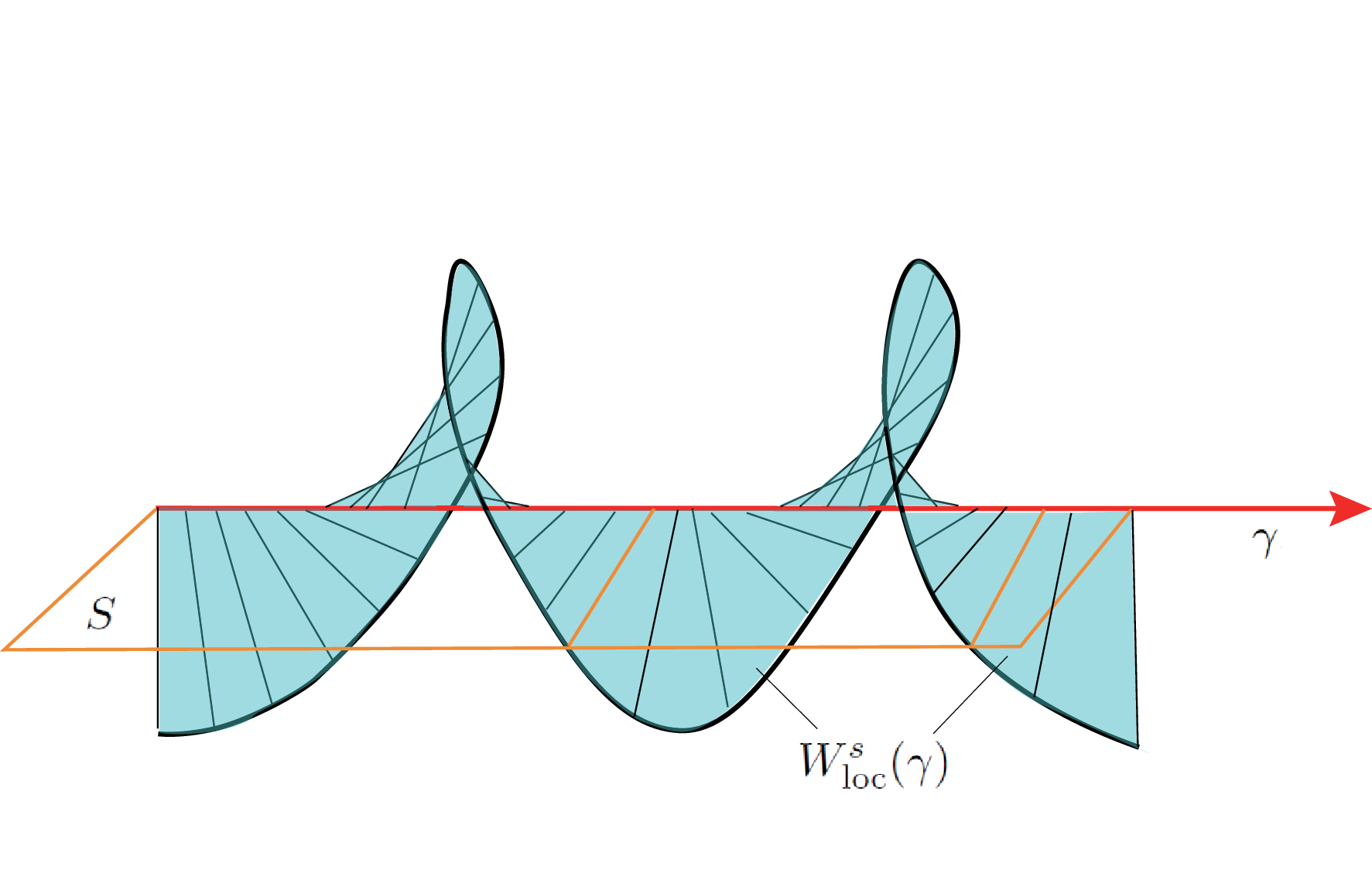}
    \caption{The case \(\operatorname{Twist}(W_{\text{loc}}^s(\gamma),S) = -2\) (assuming that the orientation corresponds to the standard orientation of $\RR^3$, and that the left of the figure is identified with the right)}
    \label{f.twistnumber2}
\end{figure}

We will use the following fact:

\begin{fact}
\label{f.twist-sections}
Let $S$ and $S'$ be two orientable surfaces containing a hyperbolic orbit  $\gamma$ with positive multipliers. Denote by $\lambda_S$ and $\lambda_{S'}$ the classes in $H_1(N_\gamma,\ZZ)$ of the dynamically oriented sections of $N_\gamma$ associated to $S$ and $S'$ respectively. Denote by $\mu$ the canonical meridian of $N_\gamma$. Then 
$$\lambda_{S'}=\lambda_S-\mathrm{Twist}_\gamma(S',S)\cdot \mu.$$
\end{fact}

\begin{proof}
Since $\lambda_S$ and $\lambda_{S'}$ are the homological classes of two sections, there must exist an integer $j\in\ZZ$ so that $$\lambda_{S'}=\lambda_S+j\cdot\mu.$$ 
Taking the intersection with $\lambda_{S}$, we get 
$$\mathrm{Twist}_\gamma(S',S)=\mathrm{Int}(\sigma_{S'},\sigma_S)=\mathrm{Int}(\lambda_S+j\cdot\mu,\lambda_S)=j\cdot\mathrm{Int}(\mu,\lambda_S)=-j,$$
where the last equality follows from Observation~\ref{remark:intersection-number}. Hence, we get
$$\lambda_{S'}=\lambda_S-\mathrm{Twist}_\gamma(S',S)\cdot\mu.$$ 
as desired.
\end{proof}

\subsection{Dehn-Fried surgeries as Dehn surgeries}
\label{subsection:Dehn-Fried-versus-Dehn}

A very classical tool in $3$-manifold topology is the Dehn surgery along a closed simple curve.  

Let $M$ be a closed $3$-manifold and $\gamma\colon \SS^1\to M$ be simple closed curve, identified with its image, and $V\simeq \SS^1\times\DD$ be a tubular neighborhood of $\gamma$ in $M$. Let $M_V$ be the closed manifold with boundary obtained by removing the interior of $V$. The boundary of $M_\gamma$ is the torus $\partial M_V=\partial V$.  A \emph{(non-oriented) meridian} is any simple closed curve on $\partial M_V$, non zero homotopic on the torus $\partial M_V=\partial V$ and bounding a disk in the solid torus $V$. 

A Dehn-filling of $M_V$ is a manifold $M_{\gamma,\varphi}$ obtained from $M_V$ by gluing a solid torus $\SS^1\times\DD$ on the torus $\partial M_V$ through a diffeomorphism $\varphi\colon \SS^1\times \partial \DD\to \partial M_V$. A classical result from the $3$-dimensional topology is that $M_{\gamma,\varphi}$ only depends (up to a diffeomorphism equal to the identity map outside an arbitrarily small neighborhood of $\gamma$) on the free homotopy class (as a non-oriented curve) of $m:=\varphi(\partial \DD)$ in the torus $\partial M_V$.  In other words, $M_{\gamma,\varphi}$ only depends on the homotopy class of the new non-oriented meridian $m$. One says that $M_{\gamma,\varphi}$ is obtained from $M$ by a \emph{Dehn surgery along $\gamma$ with new meridian $m$}. 

Assume now that $M$ is oriented, $X$ is a transitive Anosov vector field on $M$ and $\gamma$ is a periodic orbit of $X$ with positive multipliers. The manifold $\widehat M$ obtained by blowing up $\gamma$ is diffeomorphic to $M_V$ by a diffeomorphism equal to the identity map out of an arbitrarily small neighborhood of $\gamma$ (the solid torus  $V$ can be chosen arbitrarily small). The transverse orientation of $X$ induced by the orientation of $M$ has allowed us to define a canonical oriented meridian $\mu$ (Definition~\ref{definition:canonical-meridian}), and the invariant (stable/unstable) manifolds of $\gamma$ defined canonical longitudes $\lambda^s$ (Definition~\ref{definition:canonical-longitude}) on $N_\gamma\simeq \partial\widehat M$. When we identify $\widehat M$ with $M_V$, the canonical oriented meridian $\mu$ is identify to the (homotopy class) of the curves on $\partial M_V=\partial V$ bounding discs in $V$, equipped with the orientation as boundaries of those discs, transversally oriented by $\gamma$.

The blow down of $\hat M$ along the new fibration $p'$ considered in Subsection~\ref{subsection:Dehn-Fried} is equivalent to performing a Dehn filling whose new meridians are the fibers of $p'$. More precisely,

\begin{fact}
\label{f.Dehn-Fried-vs-Dehn}
The $3$-manifold $M'$ (carrying a transitive Anosov flow) obtained from $M$ by performing an index $k$ Dehn-Fried surgery along $\gamma$ is homeomorphic to the manifold $M_{\gamma,\varphi}$ obtained by the Dehn surgery along $\gamma$ with new meridian $\pm(\mu+k\lambda^s)$. Furthermore, the homeomorphism can be chosen to coincide out of an arbitrarily small neighborhood of $\gamma$ with the canonical identification with $M$. 
\end{fact} 

We will say that the index $k$ Dehn-Fried surgery along $\gamma$ and the Dehn surgery along $\gamma$ with new meridian $\pm(\mu+k\lambda^s)$ are \emph{topologically equivalent}.

\subsection{Dehn twist in fibers as Dehn surgeries}
\label{subsection:Dehn-fibered}

Assume now that $M$ is a closed oriented $3$-manifold fibering over the circle. More precisely, assume 
$$M:=(S\times [0,2\pi])_{/(\psi(x),0)\sim (x,2\pi)}, $$
where $S$ is an oriented closed surface and $\psi$ is an orientation preserving diffeomorphism whose isotopy class is called the monodromy of the fibration.  For $\nu\in \SS^1$, we denote by $S_\nu=S\times \{\nu\}$ the fiber over $\nu$.  

Let $\gamma\colon \SS^1\to S_\nu$ be a simple closed curve.  
Let $\tau_\gamma\colon S_\nu\to S_\nu$ be a diffeomorphism of $S_\nu$ which is the identity map out of a small neighborhood of $\gamma$ and  is the left Dehn-twist along $\gamma$ in a annular neighborhood $A_\gamma$ of $\gamma$ in $S_\mu$. In what follows, we will consider $\tau_\gamma$ as a diffeomorphism of $S$ (as well as of $S_\nu=S\times \{\nu\}$). 

For $k\in\ZZ$,  we denote by $M_{\tau_\gamma, k}$ the closed  $3$-manifold obtained by the following procedure, called \emph{Dehn-twist of order $k$ in $S_\mu$ along $\gamma$}: 
\begin{itemize}
    \item for some small $\varepsilon>0$, one cuts $M$ along $S_\nu$ so that a neighborhood  $S\times [\nu-\varepsilon,\nu+\varepsilon]$ of $S_\nu$ is disconnected and becomes  $$S\times [\nu-\varepsilon, \nu^-:=\nu] \coprod S\times [\nu^+:=\nu,\nu+\varepsilon];$$
    \item one glues back $S\times \{\nu^-\}$ with $S\times \{\nu^+\}$ by 
    $(\tau_\gamma^k(x),\nu^+)\simeq (x,\nu^-).$
\end{itemize}

Then $M_{\tau_\gamma,k}$ is a closed $3$-manifold, fibering over the circle, with fiber $S$ and monodromy $\psi\circ \tau_\gamma^k$. 

The manifold $M_{\tau_\gamma,k}$ can be seen as the result of a Dehn surgery on $M$ along $\gamma$ as follows. When one proceeds to the gluing of $S\times \{\nu^-\}$ with $S\times \{\nu^+\}$, one starts by gluing by $(S-A_\gamma)\times \{\nu^-\}$ with $(S-A_\gamma)\times \{\nu^+\}$ by the identity map.  One gets a topological manifold with boundary $M_{A_\gamma}$, whose boundary is a torus $T_\gamma$ which consists of two copies of the annulus $A_\gamma$ glued along their boundary by the identity map. This manifold $M_{A_\gamma}$ is homeomorphic to $M\setminus\mathrm{int}(V)$ where $V$ is a closed tubular neighborhood of $\gamma$. To obtain $M_{\tau_\gamma,k}$, it remains to glue $A_\gamma\times\{\nu^-\}$ with $A_\gamma\times \{\nu^+\}$ by the diffeomorphism $\tau_\gamma$. This is equivalent to a Dehn filling, where the new meridian is the union of a essential arc $\alpha\subset A_\gamma\times \{\nu^-\}$ and its image $\tau_\gamma(\alpha)\subset A_\gamma\times\{\nu^+\}$. For any orientation of $\gamma$, this new meridian is freely homotopic to $\pm(\mu+k\lambda_{S})$ where:
\begin{itemize}
\item $\mu$ is the old meridian of $T_\gamma$ oriented by $\gamma$ (which is the canonical oriented meridian if we identify $M_{A_\gamma}$ with the blow-up of $M$ along $\gamma$, see Remark~\ref{remark:orbit-vs-closed-curve});
\item $\lambda_{S}$ is the longitude of $T_\gamma$ which is the boundary components of the annuli $A_\gamma\subset S$ oriented as $\gamma$.
\end{itemize}
Note that changing the orientation of $\gamma$ changes both the orientation of the meridian $\mu$ and the orientation of the longiutude $\lambda_S$, hence does not change $\pm(\mu+k\lambda_{S})$. 

Summarizing, one gets that 

\begin{fact}
\label{f.Dehn-twist-vs-Dehn-surgery}
The manifold obtained by a Dehn surgery along the curve $\gamma$  whose new meridian is $\pm(\mu+k\lambda_{S})$ homeomorphic to the manifold obtained by a Dehn-twist of order $k$ in $S_\nu$ along $\gamma$ (and the homeomorphism can be chosen to coincide with the canonical identification of these manifolds with $M$ out of an arbitrarily small tubular neighborhood of $\gamma$). In particular, the manifold obtained by this Dehn surgery is still a fibered manifold, with new monodromy $\psi\circ \tau_\gamma^k$. 
\end{fact}

We will say that the Dehn surgery along $\gamma$ with new meridian $\pm(\mu+k\lambda_{S})$ and the Dehn-twist of order $k$ in $S_\nu$ along $\gamma$ are \emph{topologically equivalent}.

\begin{remarks}\label{r.classical} 
The equivalence between Dehn surgeries and Dehn twists is the classical argument used in the proof of Lickoricz-Wallace theorem, see \emph{e.g}~\cite[Chapter 9]{Rolfsen}.
\end{remarks}

We can  perform simultaneous Dehn-twists on a finite family of simple closed curves:

\begin{proposition}\label{p.multitwist-monodromy}
Let  $M=(S\times [0,2\pi])_{/(\psi(x),0)\sim (x,2\pi)}$ be a closed oriented fibered manifold over the circle $S^1$, with fiber $S$ monodromy $\psi$. Consider a finite sequence $0\leq \nu_1\leq \nu_2\leq \cdots \leq\nu_j<1$, a sequence of simple closed curves $\gamma_i,\dots,\gamma_j$  and a sequence of integers $k_1,\dots,k_j\in\ZZ$, so that $\gamma_i\subset S\times \{\nu_i\}$ for every $i$ and so that the $\gamma_i$'s lying in the same fiber are pairwise disjoint.

Consider the manifold $M_{\{\gamma_i,k_i\}}$ obtained by performing a Dehn-twist of order $k_i$ on $\gamma_i$ for every $i\in\{1,\dots,j\}$.  Then $M_{\{\gamma_i,k_i\}}$ is a fibered manifold with fiber $S$ and whose  monodromy is  
$$\psi\circ \tau_{\gamma_j}^{k_j}\circ\tau_{\gamma_{j-1}}^{k_{j-1}}\circ\cdots\circ \tau_{\gamma_1}^{k_1}.$$
\end{proposition}

\begin{remark}
The curves $\gamma_1,\dots,\gamma_j$ being pairwise disjoint the manifold $M_{\{\gamma_i,k_i\}}$ does not depend on the order in which the Dehn-twists surgeries are performed. 
\end{remark}

\begin{proof} 
The proof is an induction where we perform successively the Dehn-twists in the dicreasing order (first on $\gamma_j$, then on $\gamma_{j-1}$ and so on...). The surgery on $\gamma_j$ preserves the fibration and preserves the trivialisation of the fibration over the interval $[0,\nu_j)$, but the new monodromy is $\psi\circ \tau_{\gamma_j}^{k_j}.$ 

Assume that we prove that the Dehn twist on $\gamma_j,\dots,\gamma_{i+1}$ induces a fibration with monodromy $\psi\circ \tau_{\gamma_j}^{k_j}\circ\tau_{\gamma_{j-1}}^{k_{j-1}}\circ\cdots\circ \tau_{\gamma_{i+1}}^{k_{i+1}}$. The key point is that this sequence of Dehn twist did not affect the trivialisation of the fibration over $[0,\nu_{i+1}]$.  As $\nu_i\in[0,\nu_{i+1}]$ one gets that the monodromy associated to the Dehn-twist of order $k_i$ along $\gamma_i$ is the composition of the monodromy, before the Dehn twist, with the Dehn-twist of order $k_i$ along $\gamma_i$ that is 
$\left(\psi\circ \tau_{\gamma_j}^{k_j}\circ\tau_{\gamma_{j-1}}^{k_{j-1}}\circ\cdots\circ \tau_{\gamma_{i+1}}^{k_1}\right)\circ  \tau_{\gamma_i}^{k_i}$. This concludes the proof by induction. 
\end{proof}

\begin{remark}
As already noticed, the order of the surgeries does not affect the topology of the resulting manifold. We used a specific order in the preceding proof, only because it makes the computation of the monodromy much easier. 
\end{remark}

\subsection{Dehn-Fried surgeries preserving surface fibrations I: the horizontal non twisted case}
\label{subsection:Dehn-Fried-fibered-I}

As a straightforward consequence of Subsections~\ref{subsection:Dehn-Fried-versus-Dehn} and ~\ref{subsection:Dehn-fibered}   one gets: 

\begin{proposition}
\label{proposition:preserves-fibration-I}
Let  $M:=(S\times [0,2\pi])_{/(\psi(x),0)\sim (x,2\pi)}$
be the mapping torus of an orientation preserving diffeomorphism $\psi$ of a closed oriented surface~$S$. Suppose that $M$ carries a transitive Anosov vector field $X$, so that a  horizontal curve $\gamma\subset S_\nu= S\times\{\nu\}$, with $\nu\in [0,2\pi)$, is a periodic orbit of $X$ with positive multipliers. Suppose that the twist number of the local stable manifold of $\gamma$ with respect to $S_\nu$ is equal to zero. 

For any $k\in\ZZ$,  the index $k$ Dehn-fried surgery along the periodic orbit $\gamma$ is topologically equivalent to the Dehn-twist of order $k$ along $\gamma$ in the fiber $S_\nu$. As a consequence,  the manifold obtained from $M$ by performing a index $k$ Dehn-Fried surgery  along $\gamma$ (hence carrying a transitive Anosov flow) is a fibered manifold  with fiber $S$ and  monodromy $\psi\circ \tau_\gamma^k$, where $\tau_\gamma$ is the Dehn twist in $S$ along $\gamma$. 
\end{proposition}

\begin{proof}
In Subsections~\ref{subsection:Dehn-Fried-versus-Dehn} and ~\ref{subsection:Dehn-fibered}, we have seen that the index $k$ Dehn-Fried surgery along $\gamma$ and the Dehn-twist of order $k$ along $\gamma$ in $S_\nu$ both are topologically equivalent to some Dehn surgeries. We are left to check that the new meridians are the same for the two surgeries. In the case of the index $k$ Dehn-Fried surgery, the new meridian is $\pm(\mu+k\lambda^s)$ where $\mu$ and $\lambda^s$ are respectively the canonical oriented meridian and the canonical longitude (given by the stable manifold of $\gamma$, see Definition~\ref{definition:canonical-longitude}).  In the case of the Dehn-twist of order $k$ in $S_\nu$, the new meridian is $\pm(\mu+k\lambda_S)$ where $\lambda_S$ is the longitude defined by the surface $S_\nu$ oriented as $\gamma$. Finally, saying that the twist number $\mathrm{Twist}(W^s_{loc}(\gamma),S_\nu)$ is equal to zero is, by definition, the same as saying that the (free homotopy classes of the) longitude $\lambda^s$ and $\lambda_S$ coincide. Hence, the index $k$ Dehn-Fried surgery along $\gamma$ and the Dehn-twist of order $k$ along $\gamma$ in $S_\nu$ are topologically equivalent to the same Dehn surgery.
\end{proof}

\subsection{Dehn-Fried surgeries preserving surface fibrations II: the horizontal twisted case. }
\label{subsection:Dehn-Fried-fibered-II}

Now, we consider the case where the twist number of the local stable manifold does not vanish.

\begin{proposition}
\label{proposition:preserves-fibration-II}
Let  $M:=(S\times [0,2\pi])_{/(\psi(x),0)\sim (x,2\pi)}$
be the mapping torus of an orientation preserving diffeomorphism $\psi$ of a closed oriented surface~$S$. Suppose that $M$ carries a transitive Anosov vector field $X$, so that a  horizontal curve $\gamma\subset S_\nu= S\times\{\nu\})$, with $\nu\in S^1$, is a periodic orbit of $X$ with positive multipliers. Suppose that the twist number $\mathrm{Twist}_\gamma(W^s_{\mathrm{loc}}(\gamma,X), S_\nu)$ of the local stable manifold of $\gamma$ with respect to $S_\nu$  is not zero. 

Then the index $k$ Dehn-Fried surgery along $\gamma$ is topologically equivalent to Dehn twist of order $\ell$ along $\gamma$ in $S_\nu$ in the following cases  
\begin{itemize}
    \item $\mathrm{Twist}_\gamma(W^s_{\mathrm{loc}}(\gamma,X), S_\nu)=1$ and  $k=2=-\ell$
    \item $\mathrm{Twist}_\gamma(W^s_{\mathrm{loc}}(\gamma,X), S_\nu)=-1$ and  $k=-2=-\ell$
    \item $\mathrm{Twist}_\gamma(W^s_{\mathrm{loc}}(\gamma,X), S_\nu)=2$ and  $k=1=-\ell$
    \item $\mathrm{Twist}_\gamma(W^s_{\mathrm{loc}}(\gamma,X), S_\nu)=-2$ and $k=-1=-\ell$ 
\end{itemize}
In all these cases, the monodromy of the new fibration is $\psi\circ \tau_{\gamma}^{-k}$. 
\end{proposition}

\begin{proof} 
According to Sections~\ref{subsection:Dehn-Fried-versus-Dehn} and ~\ref{subsection:Dehn-fibered} one just needs to check that the new meridians, for the Dehn-Fried surgery and for the Dehn-twist are the same. Remember that these meridians are not oriented.  These meridians are equal
 if and only if
$$\mu+k\lambda^s=\pm(\mu+\ell\lambda_{S})\in H_1(\partial V, \ZZ),$$
where $V$ is a tubular neighborhood of $\gamma$. 
Since the twist number $\mathrm{Twist}_\gamma(W^s_{\mathrm{loc}}(\gamma,X), S_\nu))$ is not zero, the longitudes $\lambda^s$ and $\lambda_{S}$ have a non-zero intersection number.  One deduces that the sign in the equality above cannot be $+$.  Thus the meridian coincide if and only if 
$$2\mu= -k\lambda^s-\ell\lambda_{S} \in H_1(\partial V, \ZZ).$$
Both $\lambda^s$ and $\lambda_{S}$ have intersection number equal to $1$ with $\mu$.  Thus the equation is equivalent to $\ell=-k$ and 
$$2\mu= -k(\lambda^s-\lambda_{S}).$$
Now $\lambda_{S}=\lambda^s+\left(\mathrm{Twist}_\gamma(W^s_{\mathrm{loc}}(\gamma,X), S_\nu)\right)\cdot \mu$ (see Fact~\ref{f.twist-sections}). 
Thus the equation is equivalent to $k=-\ell$ and 
$$2\mu=k \left(\mathrm{Twist}_\gamma(W^s_{\mathrm{loc}}(\gamma,X), S_\nu)\right)\cdot \mu.$$
One deduces the four possible announced cases. The fact that $k=-\ell$ implies that the new monodromy is  $\psi\circ \tau_{\gamma}^{-k}$, ending the proof.
\end{proof}

\subsection{ Sequence of Dehn-Fried surgeries preserving surface fibrations}
\label{subsection:multi-Dehn-Fried}

Summarizing the results in Subsection~\ref{subsection:Dehn-fibered}, ~\ref{subsection:Dehn-Fried-fibered-I} and~\ref{subsection:Dehn-Fried-fibered-II}, one gets the following 

\begin{proposition}\label{p.multi-Dehn-Fried}  
Let $M:=(S\times [0,2\pi])_{/(\psi(x),0)\sim (x,2\pi)}$ be a closed oriented $3$-manifold fibered over $\SS^1$ and with fiber $S$ and monodromy $\psi$. Assume that $M$ carries a transitive Anosov vector-field $X$ so that there are $0\leq \nu_1\leq\nu_2\dots\leq \nu_j<1$ and periodic orbits $\gamma_1,\dots,\gamma_j$ with positive multipliers so that $\gamma_i$ is contained in the fiber $S_{\nu_i}=S\times \{\nu_i\}$ and the $\gamma_i$'s contained in the same fiber are pairwise disjoint. Assume moreover that   the twist number of the local stable manifold of $\gamma_i$ with respect to $S\times \{\nu_i\}$  either vanishes or belongs to $\{-1,1, 2,-2 \}$. 
Let $k_1\dots k_j$ be a sequence of integers, with the condition that $k_i$ is any integer if $\mathrm{Twist}_{\gamma_i}(W^s_{\mathrm{loc}}(\gamma_i,X), S_{\nu_i})=0$ and  
$$
k_i=2,-2,1,-1 \mbox{\; if\;  } \mathrm{Twist}_{\gamma_i}(W^s_{\mathrm{loc}}(\gamma_i,X), S_{\nu_i})=1,-1,2,-2,  \mbox{  respectively}.
$$
Let $\ell_i=k_i$ if $\mathrm{Twist}_{\gamma_i}(W^s_{\mathrm{loc}}(\gamma_i,X), S_{\nu_i})$ vanishes and $\ell_i=-k_i$ otherwise. 

Then, performing the index $k_i$ Dehn-Fried surgeries along $\gamma_i$ for every $i$ is topologically equivalent to performing the Dehn twist of order $\ell_i$ along the curves $\gamma_i$ in the surfaces $S\times \{\nu_i\}$ for $i=1,\dots,j$. The resulting manifold carries a transitive Anosov flow and fibers of $\SS^1$ with fiber $S$ and monodromy $\psi\circ \tau_{\gamma_j}^{\ell_j}\circ\cdots\circ \tau_{\gamma_1}^{\ell_1} $. 
\end{proposition}

\begin{remark}\label{r.small-neighborhoods} 
As already said in Facts~\ref{f.Dehn-Fried-vs-Dehn} and~\ref{f.Dehn-twist-vs-Dehn-surgery}, both Dehn-Fried surgeries and Dehn twists can be performed in arbitrarily small neighborhoods of the orbits $\gamma_i$ so that the homeomorphism between the  resulting manifolds coincides with the canonical identification out of these neighborhoods. 
\end{remark}

\begin{remark} \label{r.pieces-of-monodromy}
In this paper, we will apply Proposition~\ref{p.multi-Dehn-Fried} to a slightly different presentation of the fibered manifold $M$, which carries a transitive Anosov flow $X$. Instead of considering  $M$ as $(S\times [0,2\pi])_{/(\psi(x),0)\sim (x,2\pi)}$ we consider 
$$
M:=\left(\bigsqcup_{s=1}^{l} S\times [2(s-1)\pi,2s\pi]\right)_{{\Huge{/}}\begin{array}[t]{l}
(x,2s\pi^-)\sim (\psi_s(x),2s\pi^+),\; s=1\dots,l-1\\
(x,2l\pi)\sim (\psi_l(x),0) 
\end{array}}
$$
This means that we consider $M$ as a fibration over the circle $\RR/2l\pi$, with fiber $S$ and monodromy  $\psi=\psi_l\circ\psi_{l-1}\circ\cdots\circ\psi_1$. 

We assume that there are 
$$
0\leq \nu_1^1\leq\dots\leq \nu_{j_1}^1<2\pi\leq  \nu_1^2\leq\dots\leq \nu_{j_2}^2<4\pi\leq\dots 
\dots < 2(l-1)\pi\leq \nu_1^l\leq \dots\leq \nu_{j_l}^l< 2s\pi
$$
and that $X$ has periodic orbits $\gamma_1^1,\dots,\gamma_{j_1}^1,\gamma_1^2,\dots,\gamma_{j_2}^2,\dots,\gamma_1^l,\dots,\gamma_{j_l}^l$ with positive multipliers so that $\gamma_i^s$ is contained in the fiber $S_{\nu_i^s}$, and so that the orbits contained in the same fiber are pairwise disjoint. We assume that the twists number $\mathrm{Twist}_{\gamma_i^s}(W^s_{\mathrm{loc}}(\gamma_i^s,X), S_{\nu_i^s})$ either vanishes or belongs to $\{-2,-1,1,2\}$. We consider some integers $k_i^s$ , with the condition that $k_i^s=2,-2,1,-1$ if $\mathrm{Twist}_{\gamma_i^s}(W^s_{\mathrm{loc}}(\gamma_i^s,X), S_{\nu_i^s})=1,-1,2,-2$, respectively.
One denotes $\ell_i^s=k_i^s$ if $\mathrm{Twist}_{\gamma_i^s}(W^s_{\mathrm{loc}}(\gamma_i^s,X), S_{\nu_i^s})$ vanishes and $\ell_i^s=-k_i^s$ otherwise. 

Thus Proposition~\ref{p.multi-Dehn-Fried} applies in the same way with this presentation. Performing index $k_i^s$ Dehn-Fried surgeries along the $\gamma_i^s$ for every $i,s$ leads to a fibered manifold carrying an Anosov flow whose monodromy, after surgeries, is 
$$\psi_l\circ \left(\tau_{\gamma_{j_l}^l}^{\ell_{j_l}^l}\circ \cdots\circ \tau_{\gamma_{1}^l}^{\ell_{1}^l}\right)\circ\psi_{l-1}\circ\cdots\circ \psi_{2}\circ \left(\tau_{\gamma_{j_2}^2}^{\ell_{j_2}^2}\circ \cdots\circ \tau_{\gamma_{1}^2}^{\ell_{1}^2}\right)\circ\psi_1\circ \left(\tau_{\gamma_{j_1}^1}^{\ell_{j_1}^1}\circ \cdots\circ \tau_{\gamma_{1}^1}^{\ell_{1}^1}\right).$$
\end{remark}

\section{Construction of an Anosov flow with many horizontal periodic orbits: the genus two case}
\label{s.g-2-fiber}
The main step of the proof of Theorem~\ref{t.generators} consists of constructing, for every $g \geq 2$, a toroidal fibered closed 3-manifold with genus $g$ fibers carrying an Anosov flow with ``many" horizontal periodic orbits. In this section, we perform the construction in the genus $g=2$ case. To this end, we adapt the notation of the introduction to this setting. 

Let $S$ be the closed, orientable, connected genus~\(2\) surface, and consider a system of essential oriented simple closed curves \(a_\ell, b_\ell, c, a_r, b_r, d\) diffeomorphic to those drawn in Figure~\ref{f.S}. The orientation of $S$ is taken so that $\mathrm{int}(a_\ell,b_\ell)=-1$.

\begin{figure}[htb]
    \centering
    \includegraphics[scale=0.3]{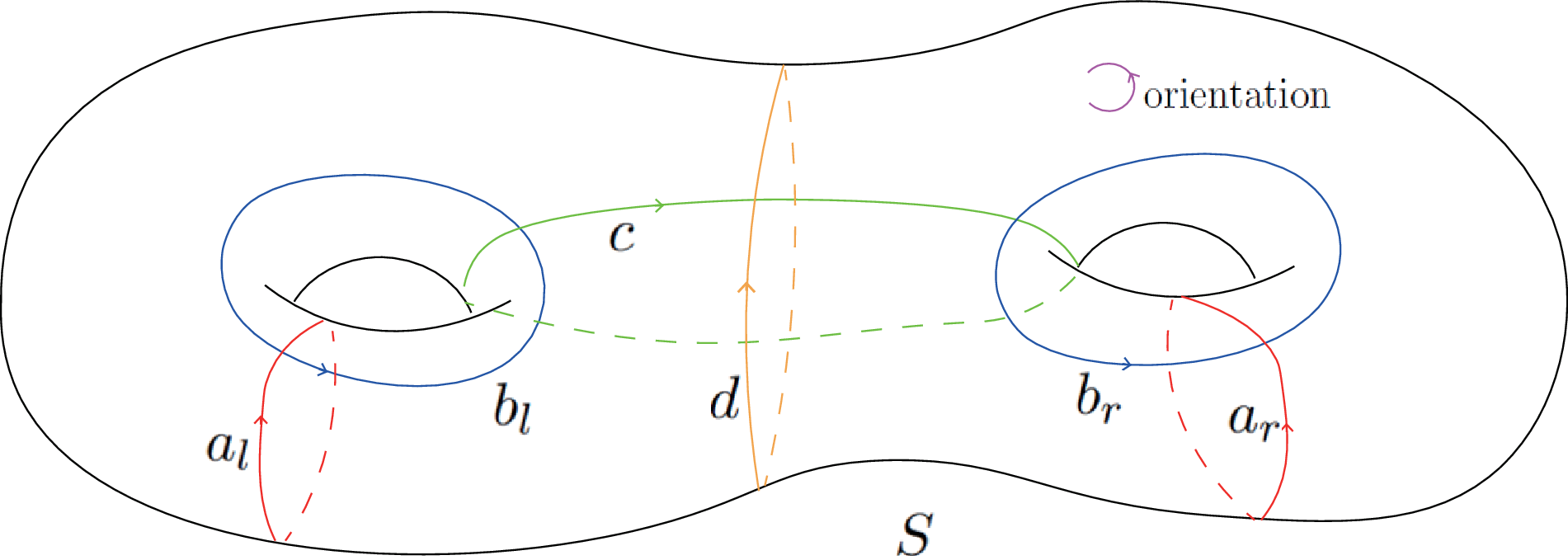}
    \caption{The genus two surface \(S\) and the oriented curves $a_\ell,b_\ell,a_r,b_r,c,d$ (the subscripts $\ell$ and $r$ stand for ``left" and ``right" with respect to $d$).}
    \label{f.S}
\end{figure}

Recall that the twist number of the local stable manifold of a periodic orbit contained in a fiber of a fibration was defined in the previous section (Definition~\ref{definition:twist-1}). The main result of this section is the following proposition. 

\begin{theorem}\label{p.perorb-lk}
Let \(\tau_d^{-2}\) be the square of the left Dehn twist along the curve \(d\). Consider the mapping torus
\[
M = S \times [0,2\pi]_{/ (x,2\pi) \sim (\tau_d^{-2}(x),0)}
\]
and denote by $\mathscr{p}:M\to \SS^1$ the fibration of $M$ induced by the projection of $S \times [0,2\pi]$ to the second coordinate. There exists a transitive Anosov flow \(Y^t\) on $M$ with orientable stable and unstable foliations such that:
\begin{enumerate}
    \item the (projection in $M$ of the) curves \(a_\ell \times \{0\}\), \(a_r \times \{0\}\), \(b_\ell \times \{\frac{3\pi}{2}\}\), \(b_r \times \{\frac{3\pi}{2}\}\) and \(c \times \{0\}\) are periodic orbits with positive multipliers of \(Y^t\), 
    \item the local stable manifolds of the orbits \(a_\ell \times \{0\}\), \(a_r \times \{0\}\), \(b_\ell \times \{\frac{3\pi}{2}\}\), and \(b_r \times \{\frac{3\pi}{2}\}\) are horizontal, \emph{i.e.}, the twist number of the local stable manifolds of these orbits with respect to the fibration $\mathscr{p}$ is equal to zero,
    \item the twist number of the local stable manifold of the orbit \(c \times \{0\}\) with respect to the surface $\mathscr{p}^{-1}(\{0\})$ is equal to \(1\).
\end{enumerate}
\end{theorem}

\begin{remark}
\label{r.incoherent-orientations}
The orientations of the curves \(a_\ell \times \{0\}\), \(a_r \times \{0\}\), \(b_\ell \times \{\frac{3\pi}{2}\}\), \(b_r \times \{\frac{3\pi}{2}\}\) and \(c \times \{0\}\) given by Figure~\ref{f.S} do not all match to their orientations as orbits of the flow \(Y^t\). This is not a problem: we have chosen some orientations of the curves $a_\ell,b_\ell,a_r,b_r,c,d$ because it will be convenient during the proof of Theorem~\ref{p.perorb-lk}, but these orientations will be meaningless once the proof is over since the left Dehn twists along a curve does not depend of orientation of this curve. 
\end{remark}

This long section is devoted to the proof of Theorem~\ref{p.perorb-lk}. We now outline the {\it strategy} of the proof. Let $S_\ell$ and $S_r$ be the two connected components of $S \setminus d$, where \(S_\ell\) contains the curves \(a_\ell\) and \(b_\ell\) and $S_r$ contains the curves \(a_r\) and \(b_r\). Clearly, \(S_\ell\) and \(S_r\) are two once-punctured tori. We will endow \( S \) with a carefully chosen hyperbolic metric \( g \), and consider the associated geodesic flow \( X^t \)  on the unit tangent bundle \( U:=T^1S \). The circle bundle \( \pi : U \to S \) is not globally trivializable. However, the restrictions of $\pi$ to $U_\ell:=T^1 S_\ell$ and $U_r:=T^1S_r$ are trivial. In other words, if we cut \( U \) open along the torus \( T = \pi^{-1}(d) \), we obtain two trivial \( \SS^1 \)-bundles $U_\ell\to S_\ell$ and $U_r\to S_r$. We will carefully choose some trivializations of these bundles, defining trivial fibrations $U_\ell \cong S_\ell \times \SS^1 \to S_{\ell} $ and $U_r \cong S_r \times \SS^1 \to S_{r}$. Since \( U \) is not globally trivializable, the fibration structures on \( U_\ell \) and \( U_r \) cannot be directly glued along \( T = \partial U_\ell = \partial U_r \). Nevertheless, we will prove that these fibrations can be pinched together by performing Dehn surgeries along the orbits \( \widetilde d ^+ \) and \( \widetilde d^- \) of the geodesic flow $X^t$ that projects on the oriented geodesics \(d\) and \(d^{-1}\). In other words, starting with the geodesic flow \( X^t \) on \( U=T^1S \) and performing Dehn-Fried surgeries along the periodic orbits \( \widetilde{d}^+ \) and \( \widetilde{d}^-\), we will obtain an Anosov flow \( Y^t \) on a manifold \(M \) that fibers over the circle. We will carefully determine the monodromy of the fibration on $M$, and get that \( M \) is diffeomorphic to  the mapping torus of the square of the left Dehn twist along the curve \( d \), \emph{i.e.},
$$M = S \times [0,1]_{ / \left( (x,1) \sim \left( \tau_d^{-2}(x), 0 \right) \right)}.$$
The hyperbolic metric \( g \) will be chosen so that \( a_\ell, b_\ell, c, a_r, b_r \) are geodesics. Moreover, the fibration structures on $U_\ell$ and $U_r$ will be chosen such that the periodic orbits of the geodesic flow projecting on \( a_\ell, b_\ell, c, a_r, b_r \) are contained in some specific fibers of the fibrations. The Dehn-Fried surgeries along \( \widetilde{d}^+ \) and \( \widetilde d^- \) will preserve these properties (recall that a Dehn-Fried surgery leaves the flow unchanged everywhere except on a single orbit, see subection~\ref{subsection:blow-up}). As a consequence, we will get that the oriented curves \(a_\ell \times \{0\}\), \(a_r \times \{0\}\), \(b_\ell \times \{\frac{3\pi}{2}\}\), \(b_r \times \{\frac{3\pi}{2}\}\) and \(c \times \{0\}\) on $M$ are periodic orbits of the flow $Y^t$. Due to the fact that stable and unstable foliations of the geodesic \(X^t\) are transverse to the fibers of the unit tangent bundle, the local stable manifolds for $Y^t$ of all the orbits above except for the last one will be horizontal. By a careful analysis, we will prove that the twist number of the local stable manifold of the orbit \(c \times \{0\}\)  with respect to the fiber containing it is equal to \(1\). Formally, Theorem \ref{p.perorb-lk} will be obtained by combining Proposition \ref{p.horizontalorbits}, Proposition \ref{p.intnumber1} and Corollary~ \ref{coro:mono}.

\subsection{The hyperbolic metric \(g\) and the symmetry \(\varsigma\)}
\label{ss.g-sigma} 

First, we need a hyperbolic metric $g$ on $S$ for which the curves $a_\ell, b_\ell, c, d, a_r, b_r$ are geodesics with some specific properties. The lemma below follows from standard hyperbolic geometry, see \emph{e.g.} \cite[Section 1.7]{Buser92}. 

\begin{lemma}\label{l.g-sigma}
There exists a hyperbolic metric \( g \) on \( S \) and an orientation reversing isometric involution \( \varsigma \) of \( (S, g) \) such that:
\begin{enumerate}
    \item The curves \( a_\ell, b_\ell, c, d, a_r, b_r \) are geodesics  for \( g \).
    \item At each intersection point of two of the curves \( a_\ell, b_\ell, c, d, a_r, b_r \), the intersecting curves meet orthogonally.
    \item The geodesic $d$ has a length $2\pi$, and the two intersection points $p$ and $q$ of \( c \) and \( d \) are at a distance of \( \pi \) from each other along the geodesic $d$.
    \item The symmetry \( \varsigma \) preserves \( d \) pointwise,  preserves \(c\) globally, maps \( a_\ell \) to \( a_r \), and maps \( b_\ell \) to \( b_r \).
 \end{enumerate}   
\end{lemma}

We fix a hyperbolic metric $g$ on the surface $S$ as given by Lemma \ref{l.g-sigma}. We denote by
$$
U:=T^1 S = \{ (x, v) \mid x \in S,  v \in T_x S,  g(v,v) = 1 \}
$$
the unit tangent bundle of the surface $S$, \( \pi: U \to S \) the natural projection, and \( X^t \) the geodesic flow on \( U \) associated to $g$. For any oriented geodesic \( e\) in $S$, we will denote by $\tilde e^+$ the oriented orbits of \( X^t \) which projects on $e$ (with the right orientation), and $\tilde e^-$ the oriented orbits of \( X^t \) which projects on $e^{-1}$.

The last item of the above lemma means that the metric $g$ is symmetric with respect to the ``middle" curve $d$. We denote by $S_\ell$ and $S_r$ the two connected components of $S\setminus d$ where  \(S_\ell\) contains the curves \(a_\ell\) and \(b_\ell\) and $S_r$ contains the curves \(a_r\) and \(b_r\). So we have a decomposition 
$$S=S_\ell\sqcup d\sqcup S_r.$$
Note that $S_\ell$ and $S_r$ are two once-punctured tori, and that $\varsigma$ maps $S_\ell$ to $S_r$ and vice versa. The decomposition of $S$ gives rise to a decomposition of the unit tangent bundle
$$U=U_\ell\sqcup T\sqcup U_r$$
where 
$$U_\ell:=T^1 S_\ell = \pi^{-1}(S_\ell),\;\; T := \pi^{-1}(d)\mbox{ and } U_r:=T^1 S_r = \pi^{-1}(S_r).$$
Note that $T$ is a torus, whereas $U_\ell$ and $U_r$ are trivial $\SS^1$-bundles over once-punctured tori.

\subsection{A coordinate system $(\rho, \theta, \varphi)$ on a neighbourhood of the torus $T$}
\label{ss.coordinate-system}

We now introduce a ``nice" coordinate system on a neighbourhood of the  geodesic $d$ in $S$. Since $d$ is a closed geodesic of length $2\pi$ in the hyperbolic surface $S$, there is a neighbourhood of $d$ in $S$ which is a standard hyperbolic cylinder. Hence (see, for instance, \cite{Buser92}):

\begin{fact}
\label{f.coordinates}
There exists a constant $\epsilon\in \left(0,\frac{1}{4}\right)$, a closed annular neighbourhood $A$ of $d$ in~$S$, and a coordinate system $(\rho, \theta): A \to [-\epsilon, \epsilon] \times \SS^1$ such that:
    \begin{enumerate}
        \item The curve $(\rho = 0)$ is precisely the geodesic $d$.
        \item The restriction $g_{|A}$ of $g$ to $A$ can be written as \[g = d\rho^2 + \cosh^2(\rho)  d\theta^2.\]
        \item The annulus $A$ is disjoint from the curves $a_\ell$, $a_r$, $b_\ell$, and $b_r$.
    \end{enumerate}
\end{fact}

The orientations will play an important role in the computations of some intersection and twist numbers, so we need to decide in which directions the coordinates are increasing. Up to replacing $\rho$ by $-\rho$ and/or $\theta$ by $-\theta$, we will assume that:
    \begin{enumerate}
        \item[(1)] The vector field $\partial_\rho$ points towards the subsurface $S_r$ along the geodesic $d$.
        \item[(2)] The vector field $\partial_\theta$ points according to the orientation of the geodesic $d$.
    \end{enumerate}
This convention immediately implies:

\begin{fact}
The left part $S_\ell\cap A$ and the right part $S_r\cap A$ of the annulus $A$ correspond respectively to the regions $\{\rho<0\}$ and $\{\rho>0\}$ in the $(\rho,\theta)$-coordinate system. 
\end{fact}

The expression of \( g \) in the \( (\rho, \theta) \)-coordinates implies that the arcs \( (\theta = \text{constant}) \) are the geodesic arcs of $g_{|A}$ orthogonal to \( d \). Now recall that, by Lemma~\ref{l.g-sigma}, \( c \) is a geodesic orthogonal to \( d \) and that the distance along \( d \) between the two points of \( c \cap d \) is equal to~\( \pi \). As a consequence, adding a constant to the coordinate $\theta$ if necessary, we get:

\begin{fact}
         The intersection of the geodesic \( c \) with the annulus \( A \) corresponds to the segments $\{\theta = 0\}$ and $\{\theta = \pi\}$ in the $(\rho,\theta)$-coordinate system.
\end{fact}

Hence, in particular, the two points \( p, q \) of intersections of \( c \) and \( d \) are the points of coordinates \( (0, 0) \) and \( (0, \pi) \). Up to exchanging the names of $p$ and $q$, we will assume that \( p \) is the point with coordinates \( (0, 0) \) and \( q \) is the point with coordinates \( (0, \pi) \). 

According to the expression of the metric \( g_{|A} \), the modulus of the $\rho$-coordinate of a point in $A$ is the distance to the geodesic \( d \). As a consequence:

\begin{fact}
\label{f.expression-symmetry}
 In the \( (\rho, \theta) \)-coordinate system, the restriction to \( A \) of the symmetry \( \varsigma \) writes
    \[
    \varsigma(\rho, \theta) = (-\rho, \theta).
    \]
\end{fact}

The coordinates \( (\rho, \theta) \) allow us to define an orientation on the surface $S$: we endow \( S \) with the orientation for which (see Figure~\ref{f.S-orientation}) $\left( \partial_\rho, \partial_\theta \right) \mbox{ is a positively oriented frame.}$

 \begin{figure}[htb]
    \centering
    \includegraphics[scale=0.35]{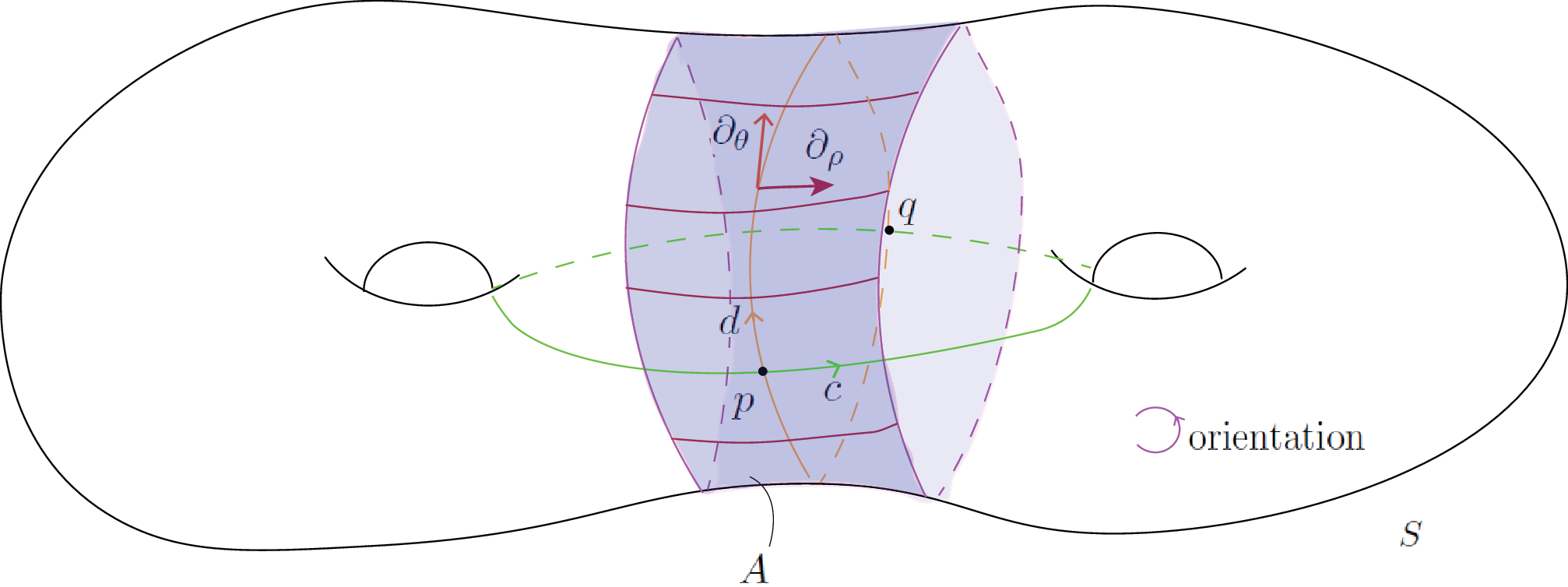}
    \caption{The orientation on \(S\)}
    \label{f.S-orientation}
\end{figure}

We now define a new coordinate \( \varphi \) so that \( (\rho, \theta, \varphi) \) will be a coordinate system on a neighbourhood of the torus \( T\) in the unit tangent bundle $U=T^1 S$. The metric \( g \) and the orientation on \( S \) allow us to measure the \textit{oriented angle} $\text{Angle}_m(v_1,v_2)$ between two tangent vectors $v_1$ and $v_2$ at the same point $m$ of \( S \). Restricting to \(A\), we define a coordinate \( \varphi \) on 
$$U_{A}:=T^1 A = \pi^{-1}(A)$$ 
as follows: 
    for $m \in A$ and $v \in T^1_m S$, we set
    \[
    \varphi(m, v) := \text{Angle}_m\left(\partial_\theta, v \right).
    \]
Adding $\varphi$ to the $(\rho,\theta)$-coordinate system, we get a coordinate system 
\[
(\rho, \theta, \varphi) : U_A \longrightarrow [-\epsilon, \epsilon] \times \SS^1 \times \SS^1.
\]
The fact below is an immediate consequence of the construction of the coordinates.

\begin{fact}
The torus \( T = \pi^{-1}(d) \) corresponds to the set \( \{\rho = 0\} \) in  \( (\rho, \theta, \varphi) \) coordinates. The left half $U_\ell\cap U_A$ and the right half $U_r\cap U_A$ of $U_A$ correspond to the regions $\{\rho<0\}$ and $\{\rho>0\}$ respectively. 
\end{fact}

Recall that, in the surface \( S \), the curve \( \{\rho = 0\} \) is the geodesic \( d \). Moreover, the vector field \( \partial_\theta \) is pointing according to the orientation of \( d \). As a consequence, in \( T^1 S \):

\begin{fact}
 The canonical lift \( \widetilde{d} \) of \( d \) in \( U=T^1 S \) is the curve \( \{\rho = 0, \varphi = 0\} \), and the canonical lift \( \tilde d^{-1} \) of \( d^{-1} \) is the curve \( \{\rho = 0, \varphi = \pi\} \).
\end{fact}

The orientation of \( A \) induces an orientation on \( U_A=T^1 A \) and, therefore, on $U=T^1S$. 

\begin{convention}
We endow \( U_A \) with the orientation for which 
$\left( \partial_\rho, \partial_\theta, \partial_\varphi \right)$ is a positively oriented basis.
\end{convention}

\subsection{Explicit sections of the unit tangent bundles \( U_\ell=T^1S_\ell \) and \( U_r=T^1S_r \)}
\label{ss.explicit-sections}  

 Recall that \( S_\ell \) and \( S_r \) are the connected components of the complement of the geodesic \( d \) in \( S \). Clearly, \( S_\ell \) and \( S_r \) are once-punctured tori. Hence the unit tangent bundles \( U_\ell\to S_\ell \) and \( U_r\to S_r \) are trivializable. We will now define  an explicit family of sections of the unit tangent bundles $U_\ell$ and $U_r$, \emph{i.e.}, families of unit vector fields on \( S_\ell\sqcup S_r \). We denote by $c_\ell$ and $c_r$ the left and right halves of the geodesic $c$, that $c_\ell := c \cap S_\ell$ and $c_{r} := c \cap S_r.$

\begin{proposition}
\label{p.SlSr-section}
There exists a continuous family \((V_\nu)_{\nu \in \SS^1}\) of unit vector fields on the non-compact surface \( S - d \) with the following properties
\begin{enumerate}
    \item The vector field \( V_0 \) is positively tangent to the closed oriented geodesics $a_\ell$ and $a_r$, and to the oriented geodesic arcs $c_\ell$ and $c_r$.
    \item The vector field \( V_{3\pi/2} \) is positively tangent to the closed oriented geodesics $b_\ell$ and $b_r$.
    \item For every \(\nu  \in \SS^1 \), the restriction to \( A - d \) of the vector field \( V_\nu \) has the following expression in the  \( (\rho, \theta) \) coordinate system:
    \begin{equation}
    \label{equation:explicit-section}
    V_\nu(\rho, \theta) = \lambda(\rho, \theta) \left( |\rho| \cos(\theta - \nu) \partial_\rho - \frac{\rho}{|\rho|}\sin(\theta - \nu) \partial_\theta \right),
    \end{equation}
    where \( \lambda(\rho, \theta) \) is a normalizing positive factor ensuring that $g(V_\nu,V_\nu) = 1$.
    \item For every \( m \in S - d \), the map 
    $$\begin{array}[t]{rcl}
    \SS^1 & \to & T_m^1 S \\
    \nu & \mapsto &  V_\nu(m)
    \end{array}$$
    is a homeomorphism.  With respect to the orientation of the fibers defined in subsection \ref{ss.coordinate-system} and the canonical orientation of $\SS^1=\RR/2\pi\ZZ$, this homeomorphism is orientation-preserving for $m\in S_\ell$ and orientation-reversing for $m\in S_r$. 
    \item For every \( \nu \in \SS^1 \), the vector field \( V_\nu \) anti-commutes with the symmetry \( \varsigma \):
    \[
    \varsigma_* V_\nu = -V_\nu.
    \]
\end{enumerate}
\end{proposition}

\begin{remark}
\begin{enumerate}
    \item The vector field $V_\nu$ is defined on the $S-d$ and cannot be extended to a vector field on $S$. This can seen on the Formula~\eqref{equation:explicit-section}, which is discontinuous at $\rho=0$. Actually, $S$ is a closed surface with non-zero Euler characteristic, hence the tangent bundle of $S$ is not trivializable, hence there does not exists any continuous unit vector field defined on the whole surface $S$.
    \item The expression of $V_\nu$ in the $(\rho,\theta)$ coordinate system (Formula~\eqref{equation:explicit-section}) admits a geometric explanation. Indeed, for $\nu=0$, the formula 
    $$
      \RR^+ . V_0(\rho, \theta) =\RR^+ . \left( |\rho| \cos (\theta) \partial_\rho - \frac{\rho}{|\rho|} \sin (\theta) \partial_\theta \right)
    $$
    corresponds to the expression in polar coordinates of the blow-up at the origin of a horizontal half-line field on $\mathbb{R}^2$. The half-line field $\RR^+.V_\nu$ is the image of the half-line field $\RR^+.V_0$ by a rotation of angle $\nu$ in these coordinates. 
    \item Figure \ref{f.V_0} represents the vector field $V_0$ on the surface $S$ which has been cut along the geodesics $a_\ell,b_\ell,d,a_r,b_r$. Notice that, in order to obtain the original surface $S$, one should glue together the lower side and the upper side of each of the two squares, glue together the left side and the right side of each of the two squares, and glue square to the right square along $d$ using a symmetry with respect to the vertical axis. For example, that the points $p$, $q$ on the left part should be glued to the points $p$, $q$ on the right part.  Figure~\ref{f.V_0-bis} is an attempt to represent $V_0$ without cutting $S$ along $d$, so that one can see that $V_0$ is discontinuous along $d$.
\end{enumerate}
\end{remark}

 \begin{figure}[htb]
    \centering
    \includegraphics[scale=0.32]{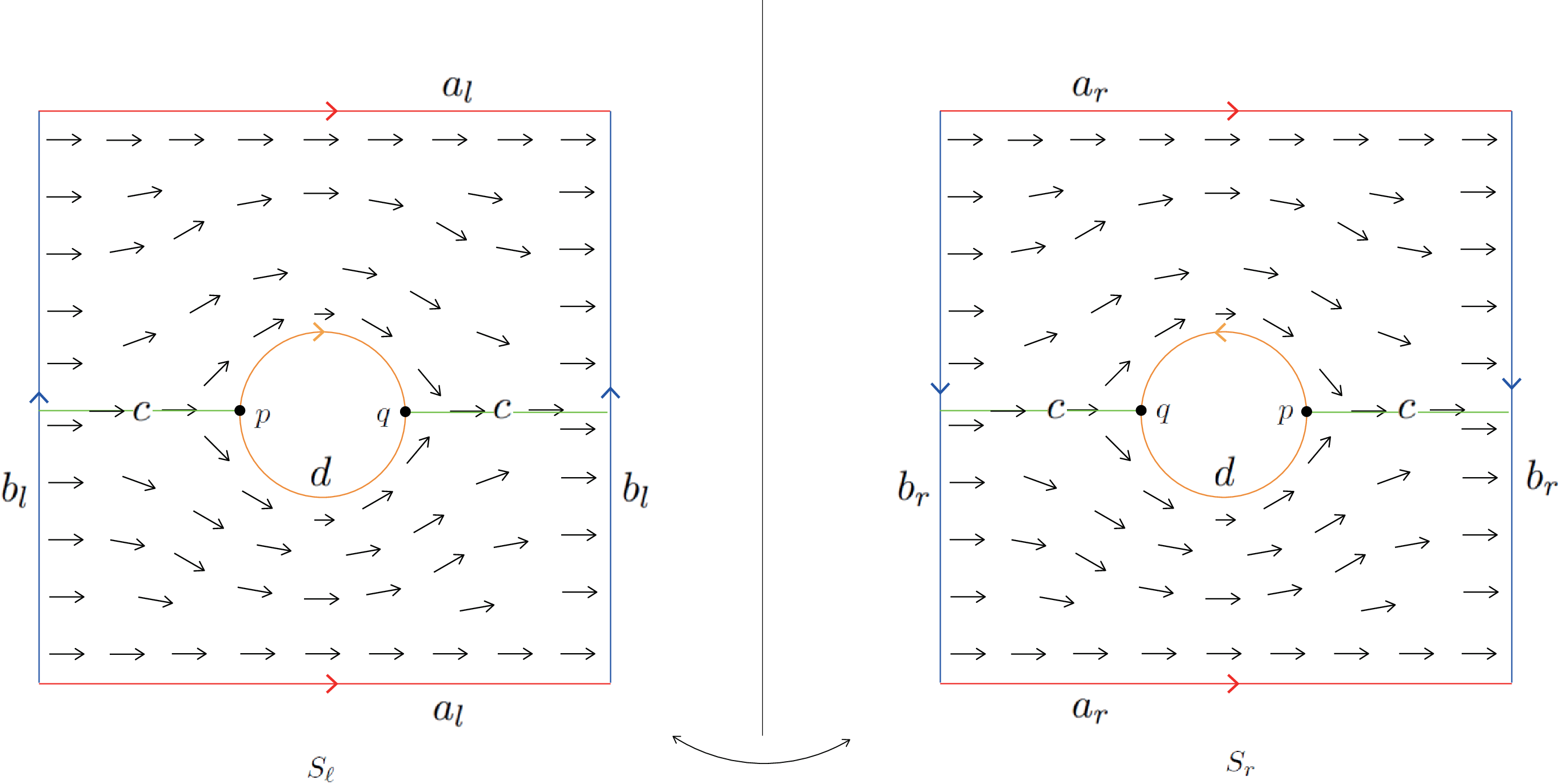}
    \caption{The vector field $V_0$ represented on the surface $S$ cut along the geodesics $a_\ell,b_\ell,d,a_r,b_r$.}
    \label{f.V_0}
\end{figure}

 \begin{figure}[htb]
    \centering
    \includegraphics[scale=0.35]{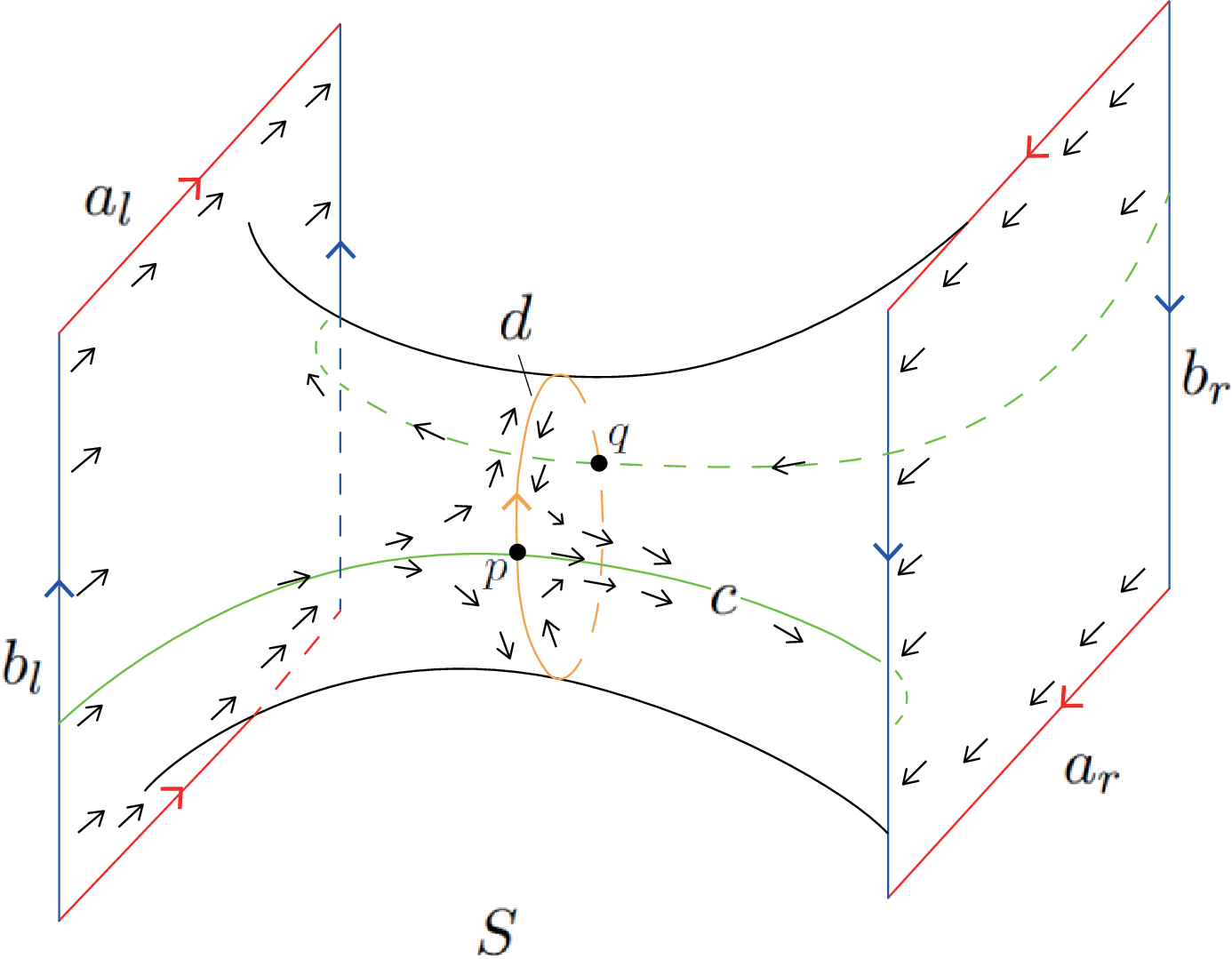}
    \caption{The vector field $V_0$ represented on the surface $S$ cut along the geodesics $a_\ell,b_\ell,a_r,b_r$. Note the discontinuity of $V_0$ along the geodesic $d$.}
    \label{f.V_0-bis}
\end{figure}

In the sequel, $\mathrm{pr}$ denotes the natural projection of $\mathbb{R}^2$ onto the torus $\mathbb{R}^2 /\mathbb{Z}^2$.  In order to prove Proposition \ref{p.SlSr-section}, we will construct a diffeomorphism $h$ mapping the surface $S_r$ to the standard punctured torus $\RR^2/\ZZ^2-\{\mathrm{pr}(0_{\RR^2})\}$, define an explicit family of vector fields $(\mathcal{V}_\nu)_{\nu\in\SS^1}$ on $\RR^2/\ZZ^2-\{\mathrm{pr}(0_{\RR^2})\}$, and pull back the $\mathcal{V}_\nu$'s on $S_r$ and $S_\ell$ thanks to the diffeomorphisms $h$ and $h \circ \varsigma$ in order to get the family of vector fields $(\mathcal{V}_\nu)_{\nu\in\SS^1}$ on $S\setminus d$.

\begin{lemma}\label{l.diffh}
There is a diffeomorphism \( h: S_{r} \to \RR^2/\ZZ^2-\{\mathrm{pr}\{0_{\RR^2}\} \) such that:
\begin{enumerate}
    \renewcommand{\labelenumi}{(\arabic{enumi})}
    \item \( h \) maps the orientation of \( S_{r} \) to the standard orientation of \( \mathbb{R}^2/\mathbb{Z}^2 \).
    \item \( h(a_{r}) = \mathrm{pr}\left(\mathbb{R} \times \left\{ \frac{1}{2}\right\}\right)  \quad \text{and} \quad h(b_{r}) = \mathrm{pr}\left(\left\{\frac{1}{2}\right\} \times \mathbb{R}\right)\).
    \item \( h(c_{r}) = \mathrm{pr}\left((\left[-\frac{1}{2},\frac{1}{2}\right]-\{0\})\times\{0\}\right) \), where $c_r = c \cap S_r$. 
    \item \( h (A \cap S_{r}) = \mathrm{pr}\left(\{(x,y) \mid 0 < x^2 + y^2 \leq \epsilon\}\right) \), and $h$ maps the coordinates \( (\rho, \theta) \) on \( A \cap S_{r} \) to (the projection in $\RR^2 /\ZZ^2$ of) the standard polar coordinates $(\bar\rho,\bar\theta)$ on \( \mathbb{R}^2 \).
\end{enumerate}
\end{lemma}

\begin{proof}
Let $\epsilon$ be the constant given by Fact~\ref{f.coordinates}. Denote by $\DD_\epsilon$ be the euclidean disc of radius $\epsilon$ centered at the origin in $\RR^2$. First, we define a diffeomorphism 
$$h_A : A \cap S_r \to \mathrm{pr}\left(\DD-\{0_{\RR^2}\}\right)$$
by setting \( h_A(\rho,\theta) = (\rho \cos \theta, \rho \sin \theta)  \), where $(\rho,\theta)$ are the coordinates defined in Fact~\ref{f.coordinates}. This ensures that $h_A$ maps $(\rho,\theta)$ to the standard polar coordinates $(\bar\rho,\bar\theta)$ on \( \mathbb{R}^2 \).

Then we observe that \( S_r- \text{int}(A) \)  is a torus minus a closed disc, and that \( a_r, b_r \) are two oriented essential simple closed curves in \( S_r - \text{int}(A) \) so that the algebraic intersection number \( \operatorname{Int}(a_r, b_r) = -1 \).
It follows that there exists a diffeomorphism 
\[
h_{A^c}: S_r- \text{int}(A) \to (\mathbb{R}^2 / \mathbb{Z}^2) - \mathrm{pr}\left( \DD_\epsilon \right)
\]
mapping \( a_r \) to $\mathrm{pr}\left(\mathbb{R} \times \left\{\frac{1}{2}\right\}\right))$ and 
\( b_r \) to \( \mathrm{pr}\left(\left\{\frac{1}{2}\right\} \times \mathbb{R}\right) \),
and preserving the orientations.

The diffeomorphisms $h_A$ and $h_{A^c}$ need not coïncide on the circle \( \partial (A\cap S_r)  \). But since any two orientation-preserving diffeomorphisms of the circle are isotopic, we may modify \( h_{A^c} \) on a small neighbourhood of $\partial (A\cap S_r)$ to get the coincidence. Then, we can glue $h_A$ and $h_{A^c}$ together, and this yields a diffeomorphism $h: S_{r} \to (\mathbb{R}^2/\mathbb{Z}^2) -\{\mathrm{pr}(0_{\RR^2})\}$ satisfying items~(1), (2) and~(4). 

We further modify \( h \) on a neighborhood of \( b_r \) so that $h(b_r \cap c) =\left\{\mathrm{pr}\left(\frac{1}{2}, 0\right)\right\}$.
Now $h$ satisfies item (2). After this modification, \( h(c \cap S_r) \) and $\mathrm{pr}\left(\{\RR-\{0\}\} \times \{0\}\right) $ are two arcs properly embedded in the punctured torus $\RR^2/\ZZ^2-\{\mathrm{pr}(0_{\RR^2})\}$. These two arcs coïncide in $h(A\cap S_r)$, and they have the same unique intersection point with the closed curve $\mathrm{pr}\left(\left\{\frac{1}{2}\right\}\times\RR\right)=h(b_r)$. So we can modify \( h \) on \( S_r - (A \cup a_r \cup b_r) \) so that, after modification, it maps \( c_r = c \cap S_r \) to $\mathrm{pr}\left(([-\frac{1}{2},\frac{1}{2}]-\{0\})\times\{0\}\right)$, so that  item (3) is satisfied.
\end{proof}

Now we will define a family \( \{ \mathcal{V}_\nu \}_{\nu\in \SS^1} \) of vector fields on \( \mathbb{R}^2 / \mathbb{Z}^2 - \{\mathrm{pr}(0_{\RR^2})\} \) and pull them back via the diffeomorphism \( h \) constructed in Lemma \ref{l.diffh}  to obtain of the family of vector fields \( \{ V_\nu \}_{\nu\in \SS^1}  \) on $S_r$. 
To construct the family \( \{ \mathcal{V}_\nu \}_{\nu\in \SS^1} \), we pick a smooth function \( \eta : (0, +\infty) \to [0, 1] \) so that
\[
\eta = 0 \quad \text{on} \quad (0, \epsilon] \quad \text{and} \quad \eta = 1 \quad \text{on} \quad \left[\tfrac{1}{4}, +\infty\right).
\]
Denoting by \( (\bar\rho, \bar\theta) \) the usual polar coordinates on \( \mathbb{R}^2 \), we define a vector field \(\widetilde{\mathcal{V}}_0\) on \( \mathbb{R}^2 - \{(0_{\RR^2})\}\) by the formula
\[
\widetilde{\mathcal{V}}_0(\bar\rho, \bar\theta) = \left[ (1 - \eta(\bar\rho)) \cdot \bar\rho  + \eta(\bar\rho)  \right]\cos (\bar\theta) \partial_{\bar\rho} - \sin (\bar\theta) \partial_{\bar\theta}.
\]
Then, for every \( \nu \in \SS^1 \), we denote by \( R_\nu \) the linear rotation of angle \( \nu \) on $\mathbb{R}^2$, and we define a vector field \(\widetilde{\mathcal{V}}_\nu\) on \( \mathbb{R}^2 - \{(0, 0)\}\) by setting
\[
\widetilde{\mathcal{V}}_\nu := (R_\nu)_* \widetilde{\mathcal{V}}_0.
\]
Observe that \(\widetilde{\mathcal{V}}_\nu\) admits the following expression in polar coordinates:
\[
\widetilde{\mathcal{V}}_\nu(\bar\rho, \bar\theta) = [(1 - \eta(\bar\rho)) \bar\rho + \eta(\bar\rho) ]  \cos(\bar\theta - \nu) \partial_{\bar\rho} - \sin(\bar\theta - \nu) \partial_{\bar\theta}.
\]

\begin{lemma}\label{l.proj} 
For every \( \nu \in \SS^1 \), the restriction of  \(\widetilde{\mathcal{V}}_\nu\) to the punctured square \hbox{\(\left[-\tfrac{1}{2}, \tfrac{1}{2}\right]^2 - \{0_{\RR^2}\}\)} projects to a vector field  \(\mathcal{V}_\nu\) on the punctured torus \hbox{\( \mathbb{R}^2/\mathbb{Z}^2 - \{\mathrm{pr}(0_{\RR^2})\} \)}.
\end{lemma}

\begin{proof}
    Recall that \( \eta(\bar\rho) = 1\) for $\bar\rho\geq  \frac{1}{4}$. As a consequence, close to the boundary of the square $\left[-\tfrac{1}{2}, \tfrac{1}{2}\right]^2$, the expression of the vector field \(\widetilde{\mathcal{V}}_\nu\) reads
    \[
    \widetilde{\mathcal{V}}_\nu(\bar\rho, \bar\theta) = \cos(\theta - \nu) \partial_{\bar\rho} - \sin(\bar\theta - \nu) \partial_{\bar\theta} = (R_\nu)_* \left( \cos (\bar\theta) \partial_{\bar\rho} - \sin (\bar\theta) \partial_{\bar\theta} \right) = (R_\nu)_* \partial x,
    \]
    where \(\partial x\) is the horizontal vector field on \(\mathbb{R}^2\). In particular, the restriction of \(\widetilde{\mathcal{V}}_\nu\) to a neighbourhood of the boundary of the square $\left[-\tfrac{1}{2}, \tfrac{1}{2}\right]^2$ is invariant under the translations by $(x,y)\mapsto (X+1,y)$ and $(x,y)\mapsto (x,y+1)$. The result follows.
\end{proof}

\begin{lemma}\label{g2:increasing}
    For every $m \in \mathbb{R}^2 -\{(0_{\RR^2})\}$, the map 
    \[\begin{array}[t]{rcl}
          \SS^1 & \to & \mathbb{P}_+(T_m\mathbb{R}^2)\\ 
          \nu & \mapsto & \RR_{>0}.\widetilde{\mathcal{V}}_\nu(m)
    \end{array}\]
    is an orientation-preserving homeomorphism when the circle $\mathbb{P}_+(T_m\mathbb{R}^2)$ is equipped with the orientation induced by the standard orientation of $\RR^2$.
\end{lemma}

\begin{proof}
This easily follows from the formula defining the vector field $\widetilde{\mathcal{V}}_\nu$. Indeed, if $m\in\RR^2-\{0_{\RR^2}\}$ is the point of polar coordinates $(\bar \rho_0,\bar \theta_0)$ (hence $\rho_0 > 0$), the curve 
  \[\nu \mapsto 
\widetilde{\mathcal{V}}_\nu(\bar\rho_0, \bar\theta_0) = [(1 - \eta(\bar\rho_0))\cdot\bar\rho_0 + \eta(\bar\rho_0) ]  \cos(\bar\theta_0 - \nu) \partial_{\bar\rho} - \sin(\bar\theta_0 - \nu) \partial_{\bar\theta}
\] 
 is an ellipse in the $(\bar\rho,\bar\theta)-$plane, parametrized according to the  orientation given by the basis $\left( \partial_{\bar\rho}, \partial_{\bar\theta} \right)$ which is a direct basis of $T_m\mathbb{R}^2$ for the standard orientation.
\end{proof}

\begin{proof}[Proof of Proposition \ref{p.SlSr-section}]
Looking at the formula defining the vector fields \(\left(\widetilde{\mathcal{V}}_\nu\right)_{\nu\in\SS^1}\) in polar coordinates, we observe that:
\begin{enumerate}
    \item \(h(a_r) = \mathrm{pr}\left(\mathbb{R}\times\{\frac{1}{2} \} \right)\) and \( h(c \cap S_r) = \mathrm{pr}\left( ([-\frac{1}{2},\frac{1}{2}] - \{0\}) \times \{0\})\right)\) are the orbits of the vector field \(\mathcal{V}_0\).
    \item \( h(b_r) = \mathrm{pr}(\{\frac{1}{2}\times \mathbb{R} \})\) is an orbit of the vector field \(\mathcal{V}_{\frac{3\pi}{2}}\).
    \item On $h(A\cap S_r)$, we have 
    $$\mathcal{V}_\nu(\bar\rho, \bar\theta) = \bar\rho  \cos(\bar\theta - \nu) \partial_{\bar\rho} - \sin(\bar\theta - \nu) \partial_{\bar\theta}$$
    (here we slightly abuse notations, by seeing the standard polar coordinates $(\bar\rho,\bar\theta)$ of $\RR^2$ as coordinates on a neighbourhood of the projection of the point $(0,0)$ in $\RR^2/\ZZ^2$).
    \item  For every \(m \in \mathbb{R}^2/\mathbb{Z}^2 - \{0_{\RR^2}\}\), the map 
    $$\begin{array}[t]{rcl}
          \SS^1 & \to & T_m(\mathbb{R}^2/\mathbb{Z}^2)\\ 
          \nu & \mapsto & \mathcal{V}_\nu(m)
          \end{array}$$
          is an orientation preserving homeomorphism (cf. Lemma \ref{g2:increasing}).
\end{enumerate}

For $\nu\in\SS^1$, we define the unit vector field \(V_\nu\) on \(S_r\) by pulling back \(\mathcal{V}_\nu\) via \(h^{-1}\) and normalizing with respect to the norm $\| \cdot \|_g$:
\[
V_\nu := \frac{(h^{-1})_* \mathcal{V}_\nu}{\| (h^{-1})_* \mathcal{V}_\nu \|_g}.
\]
We define the unit vector field \(V_\nu\) on \(S_\ell\) by pulling back \(\mathcal{V}_\nu\) via \((h \circ \varsigma)^{-1}\), taking the opposite, and normalizing:
\[
V_\nu := -\frac{\left((h \circ \varsigma)^{-1}\right)_* \mathcal{V}_\nu}{\left\| \left((h \circ \varsigma)^{-1}\right)_* \mathcal{V}_\nu \right\|_g}.
\]
Then the family of vector fields $(V_\nu)_{\nu\in\SS^1}$ obviously satisfies item (5) of Proposition~\ref{p.SlSr-section}. Items (1)...(4) above clearly imply that the vector fields $(V_\nu)_{\nu\in\SS^1}$ satisfy the items (1)...(4) of Proposition~\ref{p.SlSr-section} in restriction to $S_r$. Finally using the properties of the symmetry $\varsigma$ (Lemma~\ref{l.g-sigma} and Fact~\ref{f.expression-symmetry}), we deduce that the vector fields $(V_\nu)_{\nu\in\SS^1}$ also satisfy the items (1)...(4) of Proposition~\ref{p.SlSr-section} in restriction to~$S_\ell$. 
\end{proof}

For \( \nu \in \SS^1 \), the vector field $V_\nu$ given by Proposition \ref{p.SlSr-section} allows to define a section of the unit tangent bundle of the open surface $S- d$. Namely, for \( \nu \in \SS^1 \),  we set 
\[
S_{\ell,\nu} = \left\{ \big(m, V_\nu(m)\big) \mid m \in S_{l} \right\} \subset T^1S_{\ell}=U_\ell, 
\]
\[
S_{r,\nu} = \left\{ \big(m, V_\nu(m)\big) \mid m \in S_{r}\right\} \subset T^1S_{r}=U_r. 
\]
and 
$$S_\nu = S_{\ell,\nu} \sqcup  S_{r,\nu}.$$ 
By construction, $S_{r,\nu}$ is a section of the unit tangent bundle  \( U_r\). Moreover, according to item (4) of Proposition \ref{p.SlSr-section}, for every $m\in S_r$, the map $\nu \mapsto V_\nu(m)$ defines a homeomorphism from $\SS^1$ to $T^1_m S_r$. Hence, we can define a map
$$\begin{array}{ccl}
U_r=T^1S_r & \longrightarrow& \SS^1\\ 
(m,v) & \longmapsto & \text{the unique } \nu \text{ such that } v = V_\nu(m),
\end{array}$$
and this map is a fibration of $U_r$ over the circle, whose fibers are the surfaces $\left\{ S_{r,\nu} \right\}_{\nu \in \SS^1}$. Similarly, there is a fibration of $U_\ell$ over the circle, whose fibers are the surfaces $\left\{ S_{\ell,\nu} \right\}_{\nu \in \SS^1}$.  
Recall that, for every oriented geodesic (or geodesic arc) $e$ in $S$, we denote by $\widetilde{e}^+$ the canonical lift of $e$ to $U=T^1S$, and by $\widetilde{e}^-$ the canonical lift of the oppositely oriented geodesic $e^{-1}$. The vector fields $(V_\nu)_{\nu\in\SS^1}$ were chosen in such a way that the orbit of the geodesic flow $X^t$ lifting one of the oriented geodesics $a_\ell,b_\ell,a_r,b_r,c_\ell,c_r,d$ is contained in some section $S_\nu$ for some $\nu\in\SS^1$. More precisely, as an immediate consequence of items~(1) and~(2) of Proposition~\ref{p.SlSr-section}, we obtain:

\begin{corollary}
\label{c.geodesics-in-section}
\noindent
\begin{itemize}
    \item The closed lifted geodesics $\widetilde{a}^{+}_\ell$ and $\widetilde{a}^{+}_r$ are included in the surface $S_0$.
    \item The closed lifted geodesics $\widetilde{b}^{+}_\ell$ and $\widetilde{b}^{+}_r$ are included in the surface $S_{\frac{3\pi}{2}}$.
    \item The geodesic arcs $\widetilde{c}_\ell^+ $ and $\tilde c_r^+ $ are included in the $S_0$ (hence, the closed lifted geodesic $\widetilde{c}^{+}$ is included in the topological closure $\overline{S}_0$ of $S_0$ in $U$). 
\end{itemize}
\end{corollary}

\subsection{Closure of the sections \(S_{\ell,\nu}\) and $S_{r,\nu}$}
\label{ss.Clo-Sv}

The decomposition of the surface $S$ as $S_\ell\sqcup d\sqcup S_r$ leads to a decomposition of the unit tangent bundle $U$ as $U_\ell\sqcup T\sqcup U_r$. In the last subsection, we have defined families of sections \(\{S_{\ell,\nu}\}_{\nu \in \SS^1}\) and \(\{S_{r,\nu}\}_{\nu \in \SS^1}\) of the unit tangent bundles $U_\ell$ and $U_r$. We already know that these sections cannot be glued together along the torus $T$ to yield global sections of the unit tangent bundle $U$, since this bundle is not trivializable. Yet, our next task is to understand quite precisely the behavior of the sections \(\{S_{\ell,\nu}\}_{\nu \in \SS^1}\) and \(\{S_{r,\nu}\}_{\nu \in \SS^1}\) as they approach the torus $T$. This will be divided into the following two steps. First, we will consider the sections from a purely topological viewpoint and describe the intersection of their closures with the torus $T$ (Proposition \ref{p.top-bund}). Second, we will study the sections from a differential viewpoint and prove that their tangent planes behave well as one approaches the torus $T$ (Proposition \ref{p.par-scc}). 

 For \(\theta_0 \in \SS^1\), we will denote by \(p_{\theta_0}\) the point of coordinates \((0, \theta_0)\) in the $(\rho,\theta)$ coordinate system on the annulus $A$. Note that $p_{\theta_0}$ is a point of the geodesic $d$; hence, the fiber of $U=T^1 S$ over $p_{\theta_0}$ is included in the torus $T=\pi^{-1}(d)$. We denote by \(\widetilde{d}^+\) (resp. $\widetilde d^-$) the orbit of geodesic flow on $U$ which projects on the oriented geodesic $d$ (resp. $d^{-1}$). We will prove the following:

\begin{proposition}
\label{p.top-bund} 
For every \(\nu \in \SS^1\), the closure $\overline{S}_\nu$ of the surface $S_{\nu}=S_{\ell,\nu} \sqcup S_{r,\nu}$ in $U=T^1 S$ is a compact topological surface with boundary, with two boundary components that are precisely the lifted geodesics \(\widetilde{d}^+\) and \(\widetilde{d}^-\).
\end{proposition}

However, the intersection of the aforementioned compact topological surface $\overline{S}_\nu$ with the torus $T$ is not reduced to the lifted geodesics \(\widetilde{d}^+\) and \(\widetilde{d}^-\). During the proof of Proposition~\ref{p.top-bund}, we will obtain a complete description of this intersection.

\begin{addendum}
\label{a.topological-closure}
The intersection of the surface $\overline{S}_\nu$ with the torus $T$ is the union of:
\begin{itemize}
    \item the lifted geodesics \(\widetilde{d}^+\) and \(\widetilde{d}^-\), 
    \item the arc $d_\nu^-$ of the fiber \(T^1_{p_\nu} S\) running (according to the orientation of the fibers) from the point of intersection with the lifted geodesic \(\widetilde{d}^-\) to the point of intersection with the lifted geodesic \(\widetilde{d}^+\),
    \item the arc $d_{\nu}^+$ of the fiber \(T^1_{p_{\nu+\pi}} S\) running (according to the orientation of the fibers) from the point of intersection with the lifted geodesic \(\widetilde{d}^+\) to the point of intersection with the lifted geodesic \(\widetilde{d}^-\).
\end{itemize}
See Figure~\ref{f.topological_boundary}.
\end{addendum}

\begin{figure}
\centering
\includegraphics[scale=0.35]{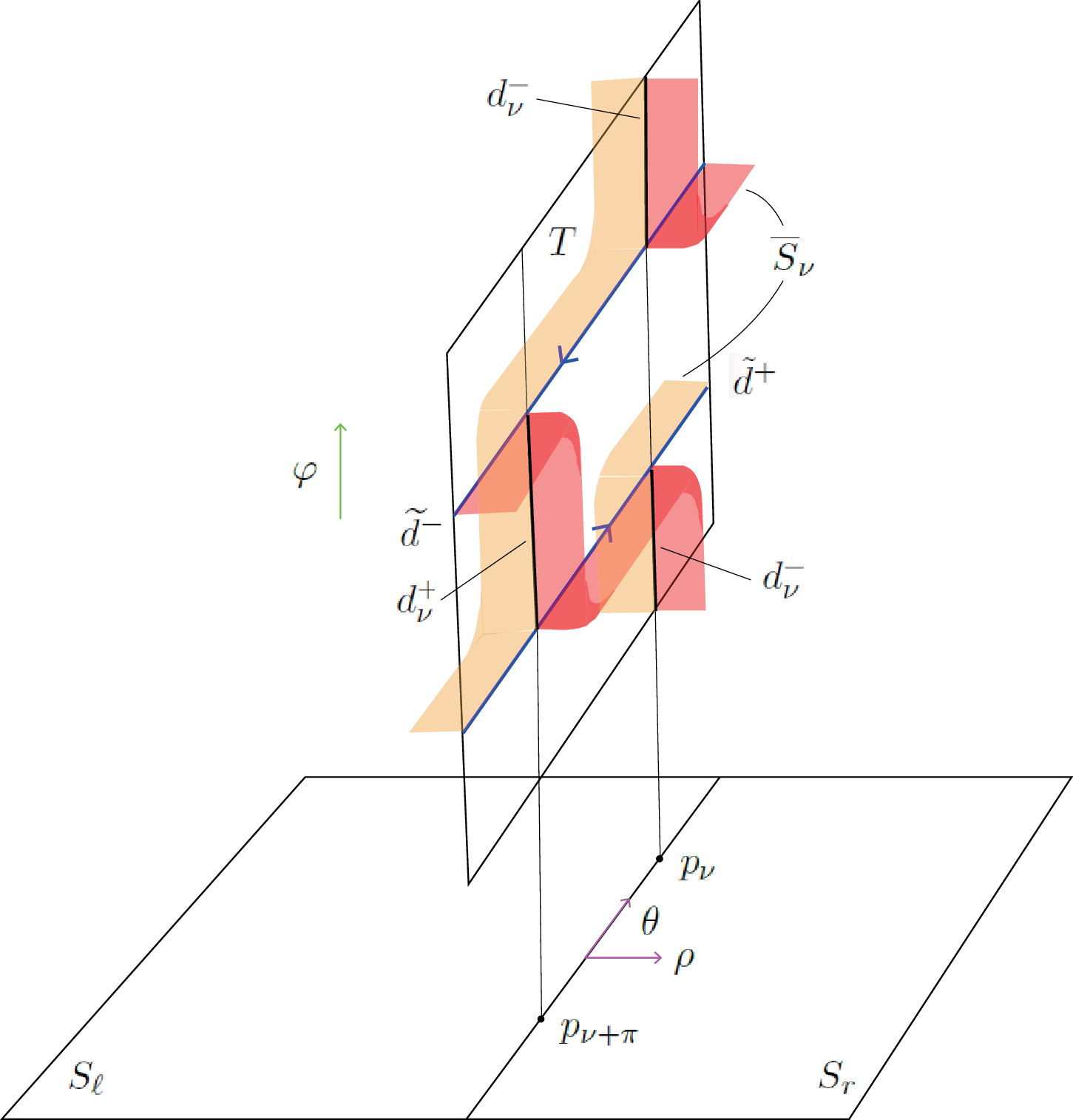}
\label{f.topological_boundary}
\caption{A local picture of the compact surface $\overline{S_\nu}$ in a neighbourhood of the torus $T$. The torus $T$ is represented as a square. The surface $\overline{S_\nu}$ is the union of the red and the tangerine part.  The boundary of $\overline{S_\nu}$ is made of the two horizontal lifted geodesics $\tilde d^-$ and $\tilde d^+$. }
\end{figure}

\begin{proof}[Proof of Proposition \ref{p.top-bund}]
We are interested in the behaviour of the surfaces $S_{\ell,\nu}$ and $S_{r,\nu}$ close to the torus $T$, and $U_A=T^1_{|A}U$ is a neighbourhood of $T$ in $U$, so we can work in the coordinate system 
$$(\rho, \theta, \varphi) : U_A \longrightarrow [-\epsilon, \epsilon] \times \SS^1 \times \SS^1$$ 
constructed in Section \ref{ss.coordinate-system}. We recall that $(\rho, \theta)$ is a coordinate system on the annulus $A$, and $\varphi$ is a coordinate along the fibers of the unit tangent bundle $U_A\to A$. More precisely, $\varphi$ is defined as follows:  for $m \in A$ and $v \in T^{1}_m A$, then
    \[
    \varphi(m,v) = \text{Angle}\left(\partial_\theta, v\right),
    \]
where the angle is measured with respect to the metric $g$, which reads $$g_{|A}=d\rho^2 + \cosh^2(\rho) d\theta^2.$$
Since $\varphi$ is the angle with a given direction, it is increasing when one runs along the fibers according to the orientation of the fibers induced by the orientation of the surface $S$. Also recall that the torus $T = T^1_{d}S$ corresponds to the set $\{\rho = 0\}$ in the $(\rho, \theta, \varphi)$ coordinates.
   
In view of Formula~(\ref{equation:explicit-section}) in Proposition \ref{p.SlSr-section}, the intersection of the surfaces \(S_{r,\nu}\) and $S_{\ell,\nu}$ with $U_A$ correspond to the following graphs over $A\cap S_r$ and $A\cap S_\ell$ respectively
$$S_{r,\nu} \cap U_A = \left\{(\rho,\theta,\varphi):\rho \in (0,\epsilon],\;
     \theta \in \SS^1, \varphi=\text{Angle}\left(
       \partial_\theta, \rho \cos(\theta - \nu) \partial_\rho - \sin(\theta - \nu) \partial_\theta\right)\right\},$$
    $$S_{\ell,\nu} \cap U_A = \left\{(\rho,\theta,\varphi): \rho \in [-\epsilon,0),
     \theta \in \SS^1, \varphi=\text{Angle}\left(
       \partial_\theta, -\rho\cos(\theta - \nu) \partial_\rho + \sin(\theta - \nu) \partial_\theta\right)\right\}.$$
Now observe that the angle
\[
\text{Angle}\left( \partial_\theta,\rho \cos(\theta - \nu) \partial_\rho - \sin(\theta - \nu) \partial_\theta \right) \]
behaves as follows when $\rho\to 0^+$. When \( \rho \to 0^+ \) and \(\theta\; \cancel{\to}\; \nu, \nu+\pi\), it tends to:
\[
\begin{cases} 
\pi & \text{when } \theta \in (\nu, \nu + \pi) \pmod{2\pi} \\
0 & \text{when } \theta \in (\nu - \pi, \nu) \pmod{2\pi}
\end{cases}
\]
When \(\rho \to 0^+\) and \(\theta \to \nu\), it accumulates on the whole segment \([-\pi, 0] \pmod{2\pi}\). When \(\rho \to 0^+, \theta \to \nu+\pi\), it accumulates the whole segment $[0,\pi] \pmod{2 \pi}$. This yields the following description of the closure $\overline{S}_{r,\nu}$ of  $S_{r,\nu}$ in the $(\rho,\theta,\varphi)$ coordinate system:
\begin{eqnarray*}
\overline{S}_{r,\nu} -S_{r,\nu} & = &  \left( \{ 0 \} \times \{ \nu \} \times [-\pi, 0] \right) \cup  \left( \{ 0 \} \times [\nu, \nu + \pi] \times \{ \pi \} \right)  \\
& & \cup\left( \{ 0 \} \times \{ \nu + \pi \} \times [0, \pi] \right) \cup \left( \{ 0 \} \times [\nu + \pi, \nu+2\pi ] \times \{ 0 \} \right).
\end{eqnarray*}
Now, recall that $p_{\theta_0} \in d$ is the point of the surface $S$ with coordinates \((\rho, \theta) = (0, \theta_0)\), hence the fiber over $p_{\theta_0}$ in $U$ corresponds to the circle $\{0\}\times\{\theta_0\}\times \SS^1$. Also recall the geodesics \(\widetilde{d}^+\) and \(\widetilde{d}^{-}\) correspond to the closed curves $\{\rho = 0,\; \varphi = 0 \}$ and  $\{\rho = 0,\; \varphi = \pi\}$ respectively. As a consequence, 
\begin{enumerate}
    \item the set \( \{ 0 \} \times \{ \nu \} \times [-\pi, 0] \) is the arc of the fiber over the point \(p_\nu\) running from the point $(0,\nu,-\pi)$ on the geodesic $\tilde d^-$ to the point $(0,\nu,0)$ on the geodesic $\tilde d^+$ as $\varphi$ increases; we denote this arc by $d_\nu^-$;
    \item the set \( \{ 0 \} \times [\nu, \nu + \pi] \times \{ \pi \} \) is the arc of the geodesic \(\widetilde{d}^{-}\) running from the point $(0,\nu,\pi)$ on the fiber of $p_\nu$ to the point $(0,\nu+\pi,\pi)$ on the fiber of $p_{\nu+\pi}$ as $\theta$ increases\footnote{Caution: the orientation for which $\theta$ is increasing is opposite to the dynamical orientation of $\tilde d^-$ as an orbit of the geodesic flow.};
    \item  the set \( \{ 0 \} \times \{ \nu+\pi \} \times [0,\pi] \) is the arc of the fiber over the point \(p_{\nu+\pi}\) running from the point $(0,\nu+\pi,0)$ on the geodesic $\tilde d^+$ to the point $(0,\nu+\pi,\pi)$ on the geodesic $\tilde d^-$  as $\varphi$ increases; we denote this arc by $d_{\nu}^+$; 
    \item \(\left( \{ 0 \} \times [\nu + \pi, \nu ] \times \{ 0 \} \right)\) is the arc of the geodesic \(\widetilde{d}^{+}\) running from the point $(0,\nu+\pi,0)$ on the fiber of $p_{\nu+\pi}$ to the point $(0,\nu,0)$ on the fiber of $p_\nu$ as $\theta$ increases.
\end{enumerate}
Hence, the closure \(\overline{S}_{r,\nu}\) of \(S_{r,\nu}\) is a topological surface with boundary, and its boundary looks (from a topological viewpoint) as Figure \ref{f.bdSt+} shows. 

 \begin{figure}[htb]
    \centering
    \includegraphics[scale=0.35]{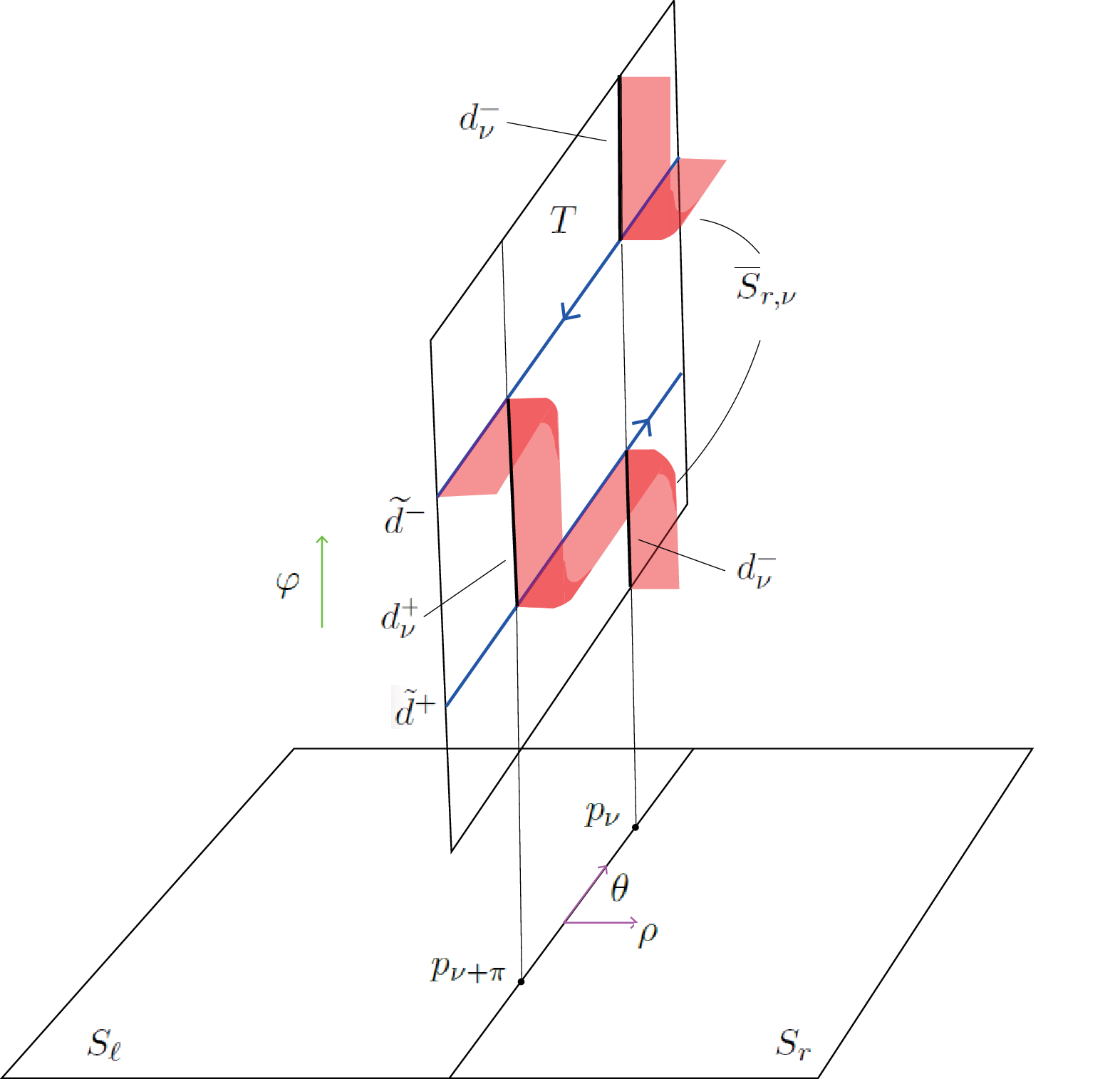}
    \caption{The surface with boundary  \(\overline{S}_{r,\nu}\) in the neighbourhood of the torus $T$}
    \label{f.bdSt+}
\end{figure}

Similarly, the angle 
\[
\text{Angle} \left(\partial_\theta,  -\rho \cos(\theta - \nu) \partial_\rho + \sin(\theta - \nu) \partial_\theta \right)
\]
behaves as follows. When \( \rho \to 0^- \) and $\rho\;\cancel{\to}\;\nu,\nu+\pi$, it tends to
\[
\begin{cases}
0 & \text{when } \theta \in (\nu, \nu + \pi) \mod 2\pi \\
\pi & \text{when } \theta \in (\nu + \pi, \nu + 2\pi) \mod 2\pi
\end{cases}
\]
When \( \rho \to 0^- \) and \( \theta \to \nu \), it 
accumulates on the whole segment \( [-\pi, 0] \) mod \( 2\pi \). And when \(\rho \to 0^-\) and \(\theta \to \nu+\pi\), it accumulates the whole segment $[0,\pi] \pmod{2 \pi}$.

As a consequence, the closure $\overline{S}_{\ell,\nu}$ of the surface $S_{\ell,\nu}$ admits the following description in the $(\rho,\theta,\varphi)$ coordinate system:
\begin{eqnarray*}
\overline{S}_{\ell,\nu} - S_{\ell,\nu} & = &  \left( \{ 0 \} \times \{ \nu \} \times [-\pi, 0] \right) \cup  \left( \{ 0 \} \times [\nu, \nu + \pi] \times \{ 0 \} \right)  \\
& & \cup\left( \{ 0 \} \times \{ \nu + \pi \} \times [0, \pi] \right) \cup \left( \{ 0 \} \times [\nu + \pi, \nu+2\pi ] \times \{ \pi \} \right).
\end{eqnarray*}
This means that $\overline{S}_{\ell,\nu}-S_{\ell,\nu}$ is the union of
\begin{enumerate}
    \item the arc $d_\nu^-$ of the fiber over the point \(p_\nu\) running from the point $(0,\nu,-\pi)$ on the geodesic $\tilde d^-$ to the point $(0,\nu,0)$ on the geodesic $\tilde d^+$ as $\varphi$ increases; 
    \item the arc of the geodesic \(\widetilde{d}^{+}\) running from the point $(0,\nu,0)$ on the fiber of $p_\nu$ to the point $(0,\nu+\pi,0)$ on the fiber of $p_{\nu+\pi}$ as $\theta$ increases;
    \item  the arc $d_{\nu}^+$ of the fiber over the point \(p_{\nu+\pi}\) running from the point $(0,\nu+\pi,0)$ on the geodesic $\tilde d^+$ to the point $(0,\nu+\pi,\pi)$ on the geodesic $\tilde d^-$  as $\varphi$ increases; 
    \item the arc of the geodesic \(\widetilde{d}^{-}\) running from the point $(0,\nu+\pi,\pi)$ on the fiber of $p_{\nu+\pi}$ to the point $(0,\nu,\pi)$ on the fiber of $p_\nu$ as $\theta$ increases.
\end{enumerate}
Hence the closure $\overline{S}_{\ell,\nu}$ of the surface $S_{\ell,\nu}$ is a topological surface with boundary, and its boundary looks (from a topological viewpoint) as Figure \ref{f.bdSt-} shows. 
 
 \begin{figure}[htb]
    \centering
    \includegraphics[scale=0.35]{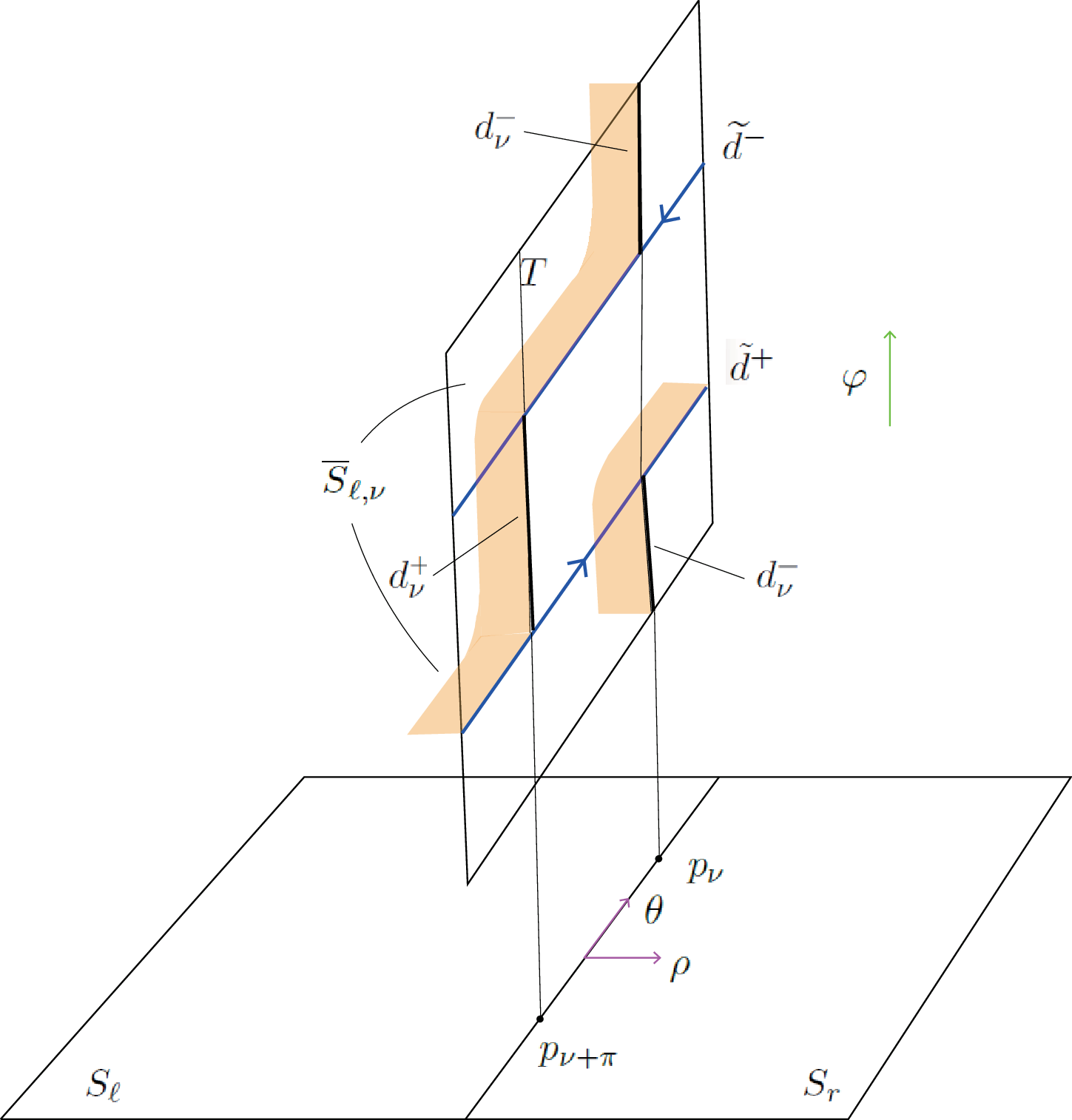}
    \caption{The surface with boundary  \(\overline{S}_{\ell,\nu}\) in the neighbourhood of the torus $T$}
    \label{f.bdSt-}
\end{figure}

Observe that the two arcs of fibers $d_\nu^-$ and $d_{\nu}^+$ contained the boundary of $\overline{S}_{\ell,\nu}$ are exactly the same as those contained the boundary of $\overline{S}_{r,\nu}$. On the other hand, the arcs of the geodesics $\tilde d^-$ and $\tilde d^+$ contained in the boundary of $\overline{S}_{\ell,\nu}$ are different from those contained in the boundary of~$\overline{S}_{r,\nu}$. From this discussion, we get that $\overline{S}_\nu=\overline{S}_{\ell,\nu}\cup\overline{S}_{r,\nu}$ is a topological surface whose boundary components are precisely the two lifted geodesics $\widetilde{d}^{+}$ and $\widetilde{d}^{-}$, and moreover that the intersection of $\overline{S}_\nu$ with the torus $T$ is exactly as described in Addendum~\ref{a.topological-closure}. See Figure~\ref{f.topological_boundary}. 
\end{proof}

As a consequence of Proposition~\ref{p.top-bund}, if we could crash down each of the two geodesics \(\widetilde{d}^{+}\) and \(\widetilde{d}^{-}\) to one point, then each \(\overline{S}_\nu\) would become a closed surface. This indicates that there is some hope that Dehn-Fried surgeries on \(\widetilde{d}^{+}\) and \(\widetilde{d}^{-}\) could turn \(\overline{S}_\nu\) into a closed surface. To understand this better, we need to blow up \(\widetilde{d}^{+}\) and \(\widetilde{d}^{-}\), and study the lift of the surfaces $S_{\ell,\nu}$ and $S_{r,\nu}$ in the blown up manifold.   

We denote by $\widehat U$ the 3-manifold with boundary obtained by blowing-up $U$ along the two lifted geodesics $\tilde d^+$ and $\tilde d^-$ (see Subsection~\ref{subsection:blow-up}). We will denote by $N_{\tilde d^-}$ (resp. $N_{\tilde d^+}$) the boundary component of $\widehat U$ corresponding to  $\tilde d^-$ (resp. $\tilde d^+$). Recall that $N_{\tilde d^\pm}$ naturally identifies with the projectivized normal bundle $\PP^+(TM/\RR X)_{|\gamma}$ of $\tilde d^\pm$. 

The non-compact surface $S_{\nu}=S_{\ell,\nu}\sqcup S_{r,\nu}$ is included in $U\setminus (\widetilde{d}^{+}\cup \widetilde{d}^{-})$; hence, it lifts homeomorphically in $\widehat U$. But the compact topological surface $\overline S_{\nu}$ needs not lift to $\widehat U$. We shall denote by $\widehat S_{\nu}$ the closure of the lift of $S_\nu$ in $\widehat U$. Our goal is to prove that $\widehat S_\nu$ is a compact surface with boundary and to understand its boundary (which, of course, is included in $N_{\tilde d^-}\cup N_{\tilde d^+}$). For that purpose, we need some explicit coordinates on $\widehat U$ in the neighborhood of the boundary tori $N_{\tilde d^\pm}$. 

Recall that the lifted geodesic $\tilde d^+$ corresponds to the curve $(\rho=0,\theta \in \SS^1,\varphi=0)$ in the \( (\rho, \theta, \varphi) \) coordinate system. Hence, a tubular neighbourhood of $\tilde d^+$ can be parametrized by $(\theta,\varphi,\rho)\in \SS^1\times\mathbb{D}_\eta$ for small positive $\eta$, where \[\mathbb{D}_{\eta} = \{(\varphi,\rho) \in \mathbb{R}^2: \rho^2 + \varphi^2 \leq \eta^2\}.\] 
Notice that $(\partial_\theta,\partial_\varphi,\partial_\rho)$ is a direct frame for the orientation of $U$ (this is important). By the definition of the blow-up manifold $\widehat U$, a coordinate system on a neighborhood of $N_{\tilde d^\pm}$ in $\widehat U$ is obtained by considering the polar coordinates in the $(\varphi,\rho)$-plane. Namely, let $(\tau^+,\psi^+)\in [0,\eta)\times \SS^1$ be defined by 
$$(\varphi,\rho)=(\tau^+\cos\psi^+,\tau^+\sin\psi^+)$$
then $(\theta,\tau^+,\psi^+)\in \SS^1 \times [0,\eta)\times\SS^1$ is a coordinate system on  neighbourhood of $N_{\tilde d^+}$ in $\widehat U$. In particular, $(\theta,\psi^+) \in \SS^1\times\SS^1$ is a coordinate system on $N_{\tilde d^+}$ and $(\partial_\theta,\partial_{\psi^+})$ is a direct frame in $N_{\tilde d^+}$.

Similarly, the lifted geodesic $\tilde d^-$ corresponds to the curve $(\rho=0,\theta\in \SS^1,\varphi=\pi)$ in the \( (\rho, \theta, \varphi) \) coordinate system, and therefore we get a coordinate system $(\theta,\tau^-,\psi^-)\in \SS^1\times [0,\eta)\times \SS^1$ which is a coordinate system on neighbourhood of $N_{\tilde d^-}$ in $\widehat U$ by setting 
$$(\varphi,\rho)=(\pi+\tau^-\cos\psi^-,\tau^-\sin\psi^-).$$
In particular, $(\theta,\psi^-) \in (\SS^1)^2$ is a coordinate system on $N_{\tilde d^-}$ and $(\partial_\theta,\partial_{\psi^+})$ is a direct frame in $N_{\tilde d^+}$.

The coordinates \( (\theta, \psi^{\pm}) \) allow us to identify some particular homology classes of \( N_{\widetilde{d}^\pm} \):

\begin{notation}
\label{notations.meridian}
 We denote by:
    \begin{itemize}
        \item \(\mu^+\) the class in \( H_1(N_{\tilde d^+}, \mathbb{Z}) \) of the curve \( (\theta = \text{constant}) \) oriented by ~$\partial_{\psi^{+}}$,
        \item \( \lambda^+ \) the class in \( H_1(N_{\tilde d^+}, \mathbb{Z}) \) of the curve \( (\psi^{+} = \text{constant}) \) oriented by $\partial_\theta$, 
         \item \(\mu^-\) the class in \( H_1(N_{\tilde d^-}, \mathbb{Z}) \) of the curve \( (\theta = \text{constant}) \) oriented by ~$-\partial_{\psi^{-}}$,
        \item \( \lambda^- \) the class in \( H_1(N_{\tilde d^-}, \mathbb{Z}) \) of the curve \( (\psi^{-} = \text{constant}) \) oriented by $-\partial_\theta$, 
    \end{itemize}
\end{notation}

The classes \( \mu^{\pm} \) and \( \lambda^{\pm} \) have some natural interpretations. Recall that we have defined a canonical oriented meridian and a canonical oriented longitude on the tori $N_{\tilde d^\pm}$ (see Defintions~\ref{definition:canonical-meridian} and~\ref{definition:canonical-longitude}).

\begin{lemma}
\( \mu^{\pm} \) and \( \lambda^{\pm} \) are, respectively, the canonical oriented meridian and the canonical oriented longitude of the torus \( N_{\tilde{d}^\pm} \).
\end{lemma}

\begin{proof}
The statement concerning $\mu^\pm$ immediately follows from Defintion~\ref{definition:canonical-meridian}, from the definition of $\mu^\pm$ above, and from the fact that the dynamical orientation of the lifted geodesic $\tilde d^\pm$ is pointing according to $\pm\partial_\theta$. 

In order to prove that $\lambda^{\pm}$ is the canonical longitude, let us observe that the metric $g$ is invariant, the geodesic $d$, and the definition of the coordinate $\varphi$ are invariant under translation of the coordinate $\theta$. Hence the geodesic flow and the local stable manifold of the lifted geodesics $\tilde d^\pm$ must also be invariant under translation of $\theta$. It follows that the sections of the projectivized normal bundle $N_{\tilde d^\pm}$ corresponding to the local stable of $\tilde d^\pm$ must be circles of the form $(\theta=\mathrm{constant})$ in the $(\theta,\psi^\pm)$ coordinate system\footnote{By writing down the explicit formulas for the derivative of the geodesic flow in the normal Fermi coordinate system $(\rho,\theta,\varphi)$ (see the proof of Lemma~\ref{l.blow-up-flow-transverse}), one can see that the sections corresponding to the local (resp. unstable) stable manifolds are the circle $(\theta=-\frac{\pi}{4}\;\mathrm{mod}\;\pi)$ (resp. $(\theta=\frac{\pi}{4}\;\mathrm{mod}\;\pi)$).}. In other words, the canonical longitude of  $N_{\tilde d^\pm}$ is the homology of the horizontal circles $\theta=\mathrm{constant}$, up to orientation. In order to check the coincidence of the orientations, one just needs to recall that the dynamical orientation of the orbit $\tilde d^\pm$ is given by the vector field $\pm\partial_\theta$.
\end{proof}

The following statement ensures that the topological compact surface $\overline S_\nu$ is actually differentiable close to its boundary $\tilde d^-\cup\tilde d^+$, and describes the behavior of the tangent planes along the boundary:

\begin{proposition}
\label{p.par-scc}
  For every $\nu\in \SS^1$, the closure $\widehat {S}_\nu$ of the lift of the surface $S_\nu$ in $\widehat U$ is a compact (topological) surface with boundary. The boundary of $\widehat{S}_\nu$ is made of the simple closed curve
  \[
    \partial^+\widehat S_\nu:=\left\{\psi^+ =  \theta - \nu + \pi\right\} \quad \text{on the torus} \quad N_{\widetilde{d}^+}
  \]
  and the simple closed curve
  \[
    \partial^-\widehat S_\nu:=\left\{\psi^- = \theta - \nu\right\} \quad \text{on the torus} \quad N_{\widetilde{d}^-}.
  \]
  See Figure~\ref{f.boundary-blow-up}.
\end{proposition}

\begin{figure}
\centering
 \includegraphics[scale=0.34]{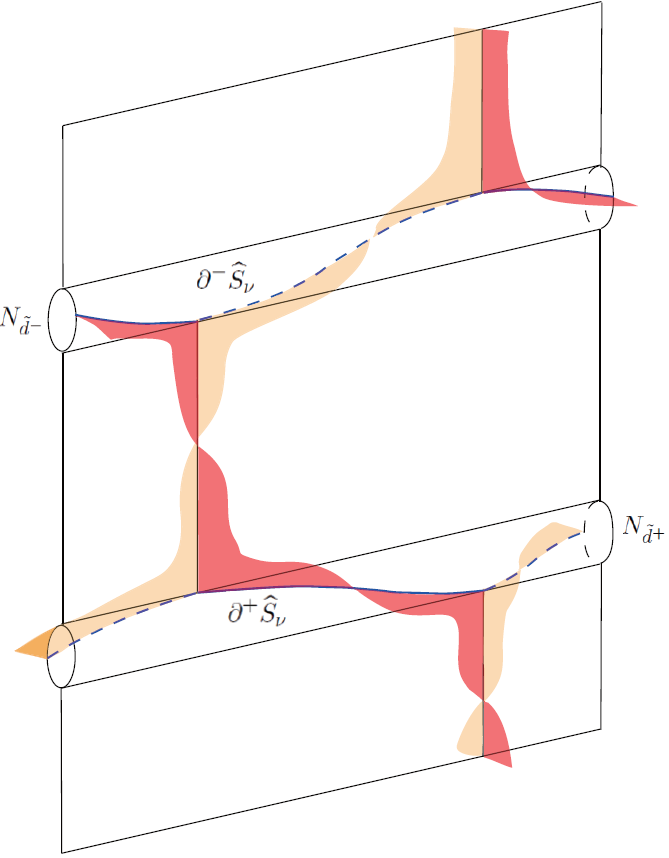}
\label{f.boundary-blow-up}
\caption{The surface $\widehat S_\nu$ in the neighbourhood of the lift in $\widehat U$ of the torus $T$.}
\end{figure}

\begin{corollary}
\label{coro:homology-class-boundary}
If we endow the curve $\partial^+\widehat S_\nu$ with the orientation for which $\theta$ is increasing, and the curve $\partial^-\widehat S_\nu$ with the orientation for which $\theta$ is decreasing, then
$$\left[\partial^{\pm}\widehat S_\nu\right]=\mu^\pm+\lambda^\pm \mbox{ in }H_1(N_{\widetilde{d}^+},\ZZ).$$
\end{corollary}

\begin{proof}
This follows immediately from the definition of the curves $\mu^\pm, \lambda^\pm$ (see notation~\ref{notations.meridian}), and the equations 
$(\psi^+=\theta - \nu + \pi, \psi^-= \theta - \nu) $ of the curves $\partial^{\pm}\widehat S_\nu$ given by Proposition~\ref{p.par-scc}.
\end{proof}

\begin{proof}[Proof of Proposition~\ref{p.par-scc}]
Recall that in the \((\rho, \theta, \varphi)\) coordinate system:
\begin{itemize}
    \item On the annulus $A$, the surface \(S_{r,\nu}\) is defined by the equation 
    \begin{equation}
    \label{e.boundary-r}
    \varphi = \text{Angle}\left( \partial_\theta, \rho \cos (\theta - \nu) \partial_\rho - \sin (\theta - \nu) \partial_\theta \right)
    \end{equation}
    where the oriented angle is measured with respect to the metric \(g = d\rho^2 + \cosh^2(\rho)  d\theta^2\).
    \item The lifted geodesic \(\widetilde{d}^+\) is the curve \(\{(\rho, \theta, \varphi) \mid \rho = 0, \varphi = 0\}\), and the intersection $\overline S_{r,\nu}\cap \widetilde{d}^+$ is the arc
    \[
    \{(\rho, \theta, \varphi) \mid \rho = 0, \, \nu + \pi \leq \theta \leq \nu + 2\pi, \, \varphi = 0\}.
    \]
\end{itemize}

Equation~\eqref{e.boundary-r} can be rewritten as:
\[
\begin{cases} 
\cos \varphi = -\dfrac{\cosh \rho \cdot \sin (\theta - \nu)}{\sqrt{\rho^2 \cos^2 (\theta - \nu) + \cosh^2 \rho \sin^2 (\theta - \nu)}} \\ 
\sin \varphi = -\dfrac{\rho \cos (\theta - \nu)}{\sqrt{\rho^2 \cos^2 (\theta - \nu) + \cosh^2 \rho \sin^2 (\theta - \nu)}} 
\end{cases}
\]

We first treat the case where $\theta - \nu \in (\pi,2\pi)$. Linearizing the two equalities above at \(\rho \approx 0\), \(\varphi \approx 0\) for $\theta - \nu \in (\pi,2\pi)$, we obtain:
\[
\begin{cases} 
1 = -\frac{\sin (\theta - \nu)}{|\sin (\theta - \nu)|} = 1 \\ 
\varphi = -\rho \dfrac{\cos (\theta - \nu)}{|\sin (\theta - \nu)|}=\rho \dfrac{\cos (\theta - \nu)}{\sin (\theta - \nu)} 
\end{cases}
\]
or equivalently
\[
\rho\cos (\theta - \nu) - \varphi \sin (\theta - \nu)  = 0.
\]
In $(\theta,\tau^+,\psi^+)$ coordinates, the above equation reads
\[
\tau^+\sin(\psi^+)\cos (\theta - \nu) - \tau^+\cos(\psi^+)\sin (\theta - \nu) = 0,
\]
which is equivalent to
$$\sin(\psi^+ - (\theta - \nu)) = 0,$$
that is $\psi^+ = \theta - \nu ~\operatorname{mod} \pi.$ But $\theta-\nu \in (\pi,2\pi)$ and $\psi^+ \in [0, \pi]$ (since $\rho \geq 0$ and $\rho=\tau^+\sin(\psi^+)$), so $\psi^+ = \theta - \nu -\pi$. Hence the closure of the lift of \(S_{r,\nu}\) in $\widehat U$ draws the curve:
\[
\psi^+ =  \theta - \nu + \pi, \quad \nu + \pi \leq \theta \leq \nu + 2\pi
\]
on the torus \(N_{\widetilde{d}^+}\).

We perform similar computations for the left part \(S_{\ell,\nu}\) of $S_\nu$, which is given by the equation: 
\begin{equation}
\label{e.boundary-l}
\varphi = \text{Angle}\left( \partial_\theta, -\rho \cos(\theta -\nu) \partial_\rho + \sin(\theta - \nu ) \partial_\theta \right) 
\end{equation}
and whose boundary intersects \(\widetilde{d}^+\) along the arc \(\{\rho = 0, \nu \leq \theta \leq \nu + \pi, \varphi = 0\}\).
Equationı\eqref{e.boundary-l} can be rewritten:
\begin{align*}
\cos \varphi &= \frac{\cosh \rho \cdot \sin(\theta - \nu)}{\sqrt{\rho^2 \cos^2(\theta - \nu) + \cosh^2 \rho \sin^2(\theta - \nu)}}, \\
\sin \varphi &= \frac{\rho \cos(\theta - \nu)}{\sqrt{\rho^2 \cos^2(\theta - \nu) + \cosh^2 \rho \sin^2(\theta - \nu)}}.
\end{align*}
Linearizing the two equalities above at \(\rho  \approx 0\), \(\varphi  \approx 0\), for $\theta - \nu \in (0,\pi)$, we get:
\[
\begin{cases} 
1 = \frac{\sin(\theta-\nu}{|\sin{\theta-\nu)|}} = 1 \\ 
\varphi = \rho \dfrac{\cos (\theta - \nu)}{|\sin (\theta - \nu)|}=\rho \dfrac{\cos (\theta - \nu)}{\sin (\theta - \nu)} 
\end{cases}
\]
As for the surface \( S_{r,\nu} \), it follows that the closure of the lift of \(S_{\ell,\nu}\) in $\widehat U$ draws the curve
\[
\psi^+ = \theta - \nu +\pi, \quad \nu \leq \theta \leq \nu + \pi
\] 
on the torus \( N_{\widetilde{d}^+} \).

As a conclusion, the closure $\widehat S_\nu$ of the lift in $\widehat U$ of the section \( S_\nu = S_{\ell,\nu} \cup S_{d,\nu} \) draws the closed curve 
\[
\psi^+ = \theta - \nu +\pi
\] 
on the torus \( N_{\widetilde{d}^+} \).

Similarly, if we linearize the (lift of the) surface \( S_{\nu} \) close to a point \( (0, \theta, \pi) \) of \( \widetilde{d}^- \), we get
\[
-\varphi + \pi = -\rho \frac{\cos(\theta - \nu)}{\sin(\theta - \nu)} ,
\]
\emph{i.e.},
\[
\rho \cos(\theta - \nu) - (\varphi-\pi)\sin(\theta - \nu)  = 0.
\]
In $(\theta,\tau^-,\psi^-)$ coordinates, the above equation reads
\[
-\tau^-\cos(\psi^-)\sin(\theta - \nu) +\tau^-\sin(\psi^-)\cos (\theta - \nu)= 0,
\]
which is equivalent to
$$\sin(\psi^- - \theta + \nu) = 0,$$
similar analysis as the previous cases show that $\psi^-=\theta-\nu$. This means that the closure $\widehat S_\nu$ of the lift of the surface \( S_\nu \) in $\widehat M$ draws the curve 
\[
\psi^- =  \theta - \nu
\] 
on the torus \( N_{\widetilde{d}^-} \).
\end{proof}

 \subsection{Turning the geodesic flow into an Anosov flow on a fibered manifold}
 \label{ss.YtM}

Recall that $\widehat U$ is the 3-manifold with boundary obtained from the unit tangent bundle $U$ by blowing-up the lifted geodesics $\tilde d^-$ and $\tilde d^+$. The boundary components of $\widehat U$ are the tori $N_{\widetilde{d} ^-}$ and $N_{\widetilde{d} ^+}$. According to Proposition~\ref{p.par-scc} and Corollary~\ref{coro:homology-class-boundary}, the curves $(\partial^\pm\widehat S_\nu)_{\nu\in\SS^1}$ are the fibers of a circle fibration of the torus $N_{\tilde d^\pm}$. These fibers (equipped with the appropriate orientation) lie in the homology class $\mu^\pm+\lambda^\pm\in H_1(N_{\tilde d^\pm},\ZZ)$ where $\mu^\pm$ and $\lambda^\pm$ are the canonical meridian and the canonical longitude of the torus $N_{\tilde d^\pm}$. 

Also recall that $X^t$ denotes the geodesic flow on the unit tangent bundle $U$. We denote by $\widehat X$ the lift of the vector field $X$ to the blown-up manifold $\widehat U$. Recall that the restriction of $\widehat X$ to the two boundary components $N_{\tilde d^\pm}$ of $\widehat M$ is induced by the derivative of $X$ (through the identification of $N_{\tilde d^\pm}$ to the projectivized normal bundle along $\tilde d^\pm$).

\begin{lemma}
\label{l.blow-up-flow-transverse}
The orbits of vector field $\widehat X$ are topologically transverse to the family of curves $(\partial^\pm \widehat S_\nu)_{\nu\in\SS^1}$ on the torus $N_{\tilde d^\pm}$.
\end{lemma}

\begin{proof}
Let us treat the case of the torus $N_{\tilde d^+}$, the case of the torus $N_{\tilde d^-}$ being similar. Let us first recall that the lifted geodesic $\tilde d^+$ corresponds to the curve $\rho=\varphi=0$ in the $(\rho,\theta,\varphi)$ coordinate system. Since $\rho$, $\theta$ and $\varphi$ are nothing but the normal Fermi coordinates along $\tilde d^+$, the action of the geodesic flow and its derivative along this geodesics reads (see \textit{e.g.}~\cite[Proposition 2.3.7]{Wilkinson}:
$$X^t(\theta,0,0)=(\theta+t,0,0)\quad\mbox{ and }\quad  DX^t_{(\theta_0,0,0)}\cdot \left ( \begin{array}{c} \partial_\theta \\ \partial_\rho \\ \partial_\varphi \end{array} \right ) = \left (\begin{array}{ccc} 1 & 0 & 0 \\ 0 & \cosh t & \sinh t \\  0 & \sinh t & \cosh t \end{array} \right ) \left ( \begin{array}{c} \partial_\theta \\ \partial_\rho \\ \partial_\varphi \end{array} \right ). $$
Now, recall that the torus $N_{\tilde d^+}$ is parametrized by the coordinate system $(\theta,\psi^+)$ where $\psi^+$ is defined by $(\varphi,\rho)=(\sqrt{\varphi^2+\tau^2}\cos\psi^+,  \sqrt{\varphi^2+\tau^2}\sin\psi^+)$.
Hence the action of the flow $\widehat X$ on the torus $N_{\tilde d^+}$ reads 
$$\widehat X^t(\theta,\psi^+) = \left(\theta+t , \arctan\frac{\cosh t \sin\psi^+ + \sinh t \cos \psi^+}{\sinh t \sin\psi^+ + \cosh t \cos \psi^+}\;\mathrm{mod}\;\pi\right).$$
It follows that, if we denote by $\tilde\theta$ the lift of the coordinate $\theta\in\SS^1=\RR/2\pi\ZZ$ to $\RR$, then the orbits of $\widehat X^t$ are graphs of the form 
$$\psi^+\left(\tilde\theta\right) = \arctan\left(\frac{\cosh \left(\tilde\theta\right) \sin\left(\psi^+_0\right) + \sinh \left(\tilde\theta\right) \cos \left(\psi^+_0\right)}{\sinh \left(\tilde\theta\right) \sin\left(\psi^+_0\right) + \cosh \left(\tilde\theta\right) \cos \left(\psi^+_0\right)}\right)\;\mathrm{mod}\;\pi.$$
The derivative of the function above writes
$$\frac{d\psi^+}{d\tilde\theta}\left(\tilde\theta\right) = \frac{\cos\left(2\psi^+_0\right)}{\cosh\left(2\tilde\theta\right)+ \sinh\left(2\tilde\theta\right)\sin\left(2\psi^+_0\right)}.$$
One can easily check that the right-hand side of the equality above is bounded by $1$ (either it is positive and its maximum is equal to $1$, or it is negative and its minimum is equal to $-1$)). This means that, in $(\theta,\psi^+)$ coordinates, the orbits of the vector field $\widehat X$ are graphs of slope at most $1$. Now, recall that, according to Proposition~\ref{p.par-scc} the curves $(\partial^+\widehat S_\nu)_{\nu\in\SS^1}$ are circles of constant slope equal to $1$ in $(\theta,\psi^+)$ coordinates. The lemma follows. Figure~\ref{f.transversality} shows the orbits of the flow $\widehat X^t$ and the curves $(\partial^+\widehat S_\nu)_{\nu\in\SS^1}$ on the torus $T_{\tilde d^+}$. 
\end{proof}

\begin{figure}
\centering
\includegraphics[scale=0.34]{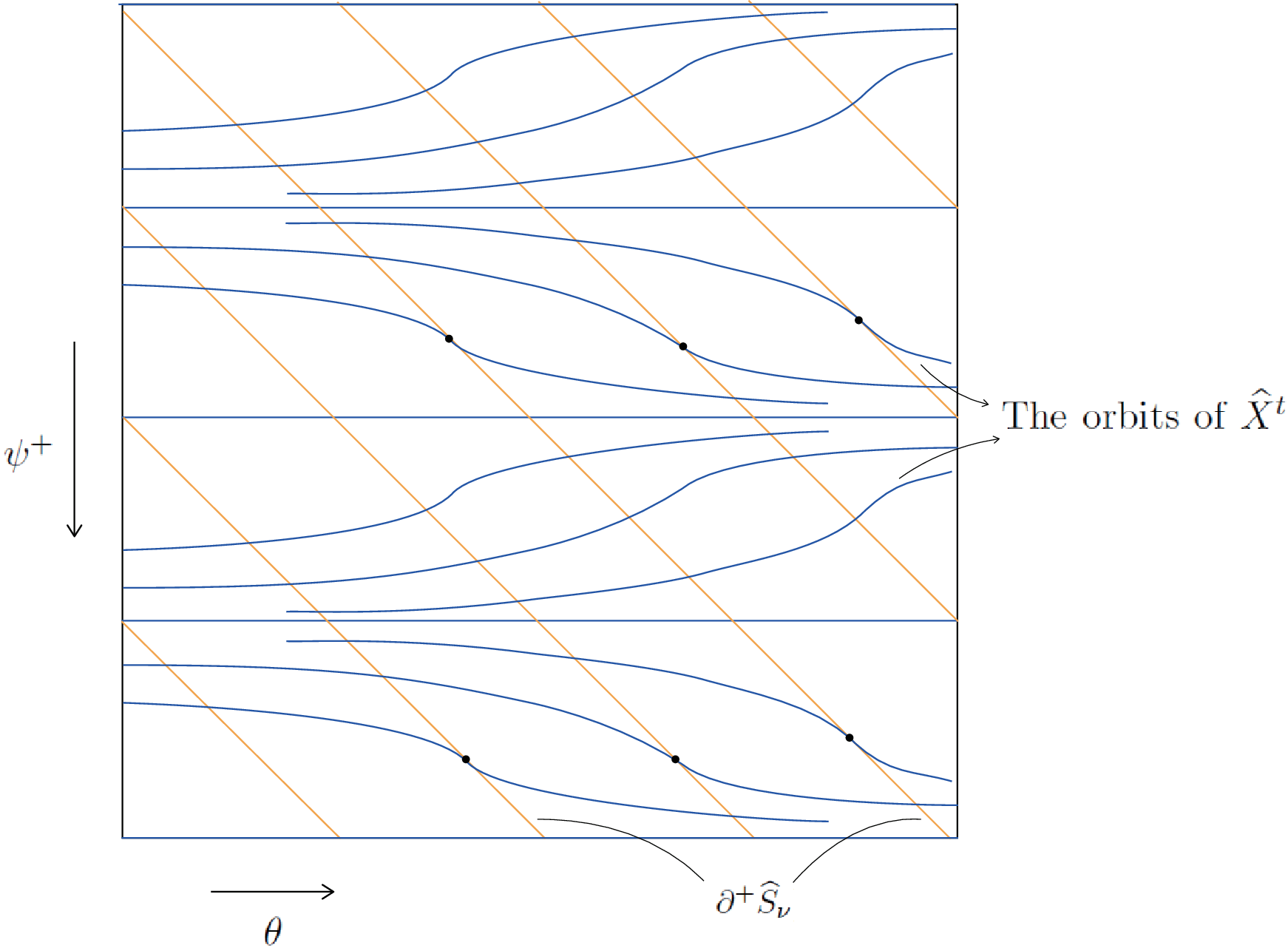}
\label{f.transversality}
\caption{The orbits of the flow $\widehat X^t$ (in blue) and the curves $(\partial^+\widehat S_\nu)_{\nu\in\SS^1}$ (in yellow) on the torus $T_{\tilde d^+}$ }
\end{figure}

\begin{remark}
The orbits of $\widehat X^t$ are topologically transverse to the $\partial^+\widehat S_\nu$'s, but have a tangencies with those curves, since the maximum slope of the orbits is precisely equal to~$1$. See Figure~\ref{f.transversality}. So we obtain exactly the property we need, with no safety margin. 
\end{remark}
 
\begin{notation}
\label{notation-fibered-manifold}
We denote by $\mathscr{U}$ the closed oriented 3-manifold obtained from $\widehat U$ by collapsing each curve $\partial^\pm\widehat S_\nu$, $\nu\in\SS^1$, to a point. 
\end{notation}

Lemma~\ref{l.blow-up-flow-transverse} ensures that the oriented one-dimensional foliation of $\widehat U$ by the orbits of the flow $\widehat X^t$ projects down in $\mathscr{U}$ to a continuous oriented one-dimensional foliation in $\mathscr{U}$. 

\begin{notation}
We denote by $\mathscr{X}$ a continuous non-singular vector field on $\mathscr{U}$, positively tangent to the oriented one-dimensional foliation obtained by projecting the orbits of $\widehat X^t$. 
\end{notation}

\begin{proposition}
\label{p.we-get-Dehn-Fried}
The vector field $\mathscr{X}$ is obtained from the geodesic vector field $X$ by performing an index $+1$ Dehn-Fried surgery on the lifted geodesic $\tilde d^+$ and an index $+1$ Dehn-Fried surgery on the lifted geodesic $\tilde d^+$. In particular $\mathscr{X}$ is orbit equivalent to a transitive Anosov vector field. 
\end{proposition}

\begin{proof}
The definition of $\mathscr{X}$ is precisely those of a vector field obtained from $X$ by some Dehn-Fried surgeries on the lifted geodesics $\tilde d^\pm$ (see Subsection~\ref{subsection:Dehn-Fried}), where we use Lemma~\ref{l.blow-up-flow-transverse} to ensure that the curves we collapsed are transverse to the orbits of the blown-up flow, as required by the construction. The indices of the Dehn-Fried surgeries are given by Corollary~\ref{coro:homology-class-boundary}. Shannon has proved that the result of such Dehn-Fried surgeries is orbit equivalent to a transitive Anosov vector field (\cite{shannon2020dehn}). 
\end{proof}

\begin{notations}
\noindent 
\begin{itemize}
\item We denote by $r_\nu^\pm$ the point of $\mathscr{U}$ which is the projection of the curve $\partial^\pm \widehat S_\nu$. 

\item We denote by $\mathscr{d}^\pm$ the curve in $\mathscr U$ which is the projection of the torus $N_{\tilde d^\pm}$. Our construction entails that $\mathscr{d}^-$ and $\mathscr{d}^+$ are two periodic orbits of $\mathscr{X}$.

\item We denote by $\mathscr{i}$ the natural identification of $\left(\mathscr{U}-\mathscr{d}^\pm,\mathscr{X}_{|\mathscr{U}-\mathscr{d}^\pm}\right)$ with $\left(U-\tilde d^\pm,X_{|U-\tilde d^\pm}\right)$.

\item We denote by $\mathscr{S}_\nu$ the projection in $\mathscr{U}$ of the surface $\widehat S_\nu$.

\item We denote by $\mathscr{d}_\nu^\pm$ the projection in $\mathscr{U}$ of the arc $d_\nu^\pm$, and we set $\mathscr{d}_\nu:=\mathscr{d}_\nu^-\cup\mathscr{d}_\nu^+$.
\end{itemize}
\end{notations}

\begin{proposition} 
\label{c.Anosov-tor-fib}
For each $\nu\in\SS^1$, $\mathscr{S}_\nu$ is a closed (\emph{i.e} compact boundaryless) genus two surface, and $\mathscr{d}_\nu$ is a simple closed curve, passing through the points $r_\nu^-$ and $r_\nu^+$, dividing the surface $\mathscr{S}_\nu$ into two once-punctured tori $\mathscr{S}_{\ell,\nu}$ and $\mathscr{S}_{r,\nu}$. 
\end{proposition}

Figure~\ref{f.fibers} and Figure~\ref{f.local-Sigma} below both are attempts to display, for a given $\nu\in\SS^1$, the surface $\widehat S_\nu$ and its projection $\mathscr{S}_\nu$. Figure~\ref{f.fibers} is a global but abstract representation of these surfaces. Figure~\ref{f.local-Sigma} represents the local situation in $\widehat U$ and $\mathscr{U}$, in a neighbourhood of the torus $N_{\tilde d^\pm}$ and their projections. 

\begin{figure}[htb]
\centering
 \includegraphics[scale=0.34]{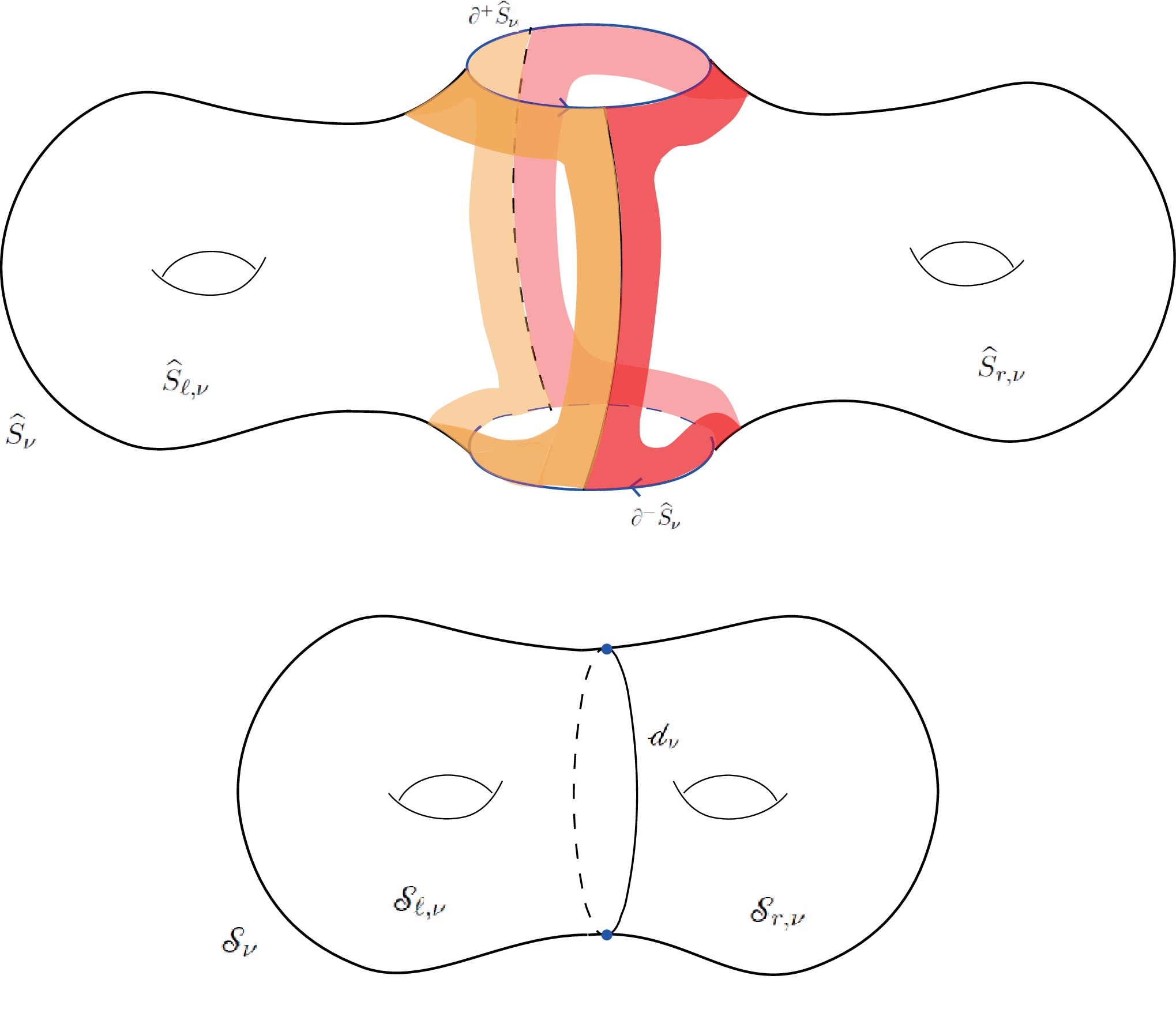}
    \caption{The surface $\widehat{S}_\nu$ and its projection $\mathscr{S}_\nu$}
    \label{f.fibers}
\end{figure}

\begin{figure}[htb]
\centering
 \includegraphics[scale=0.32]{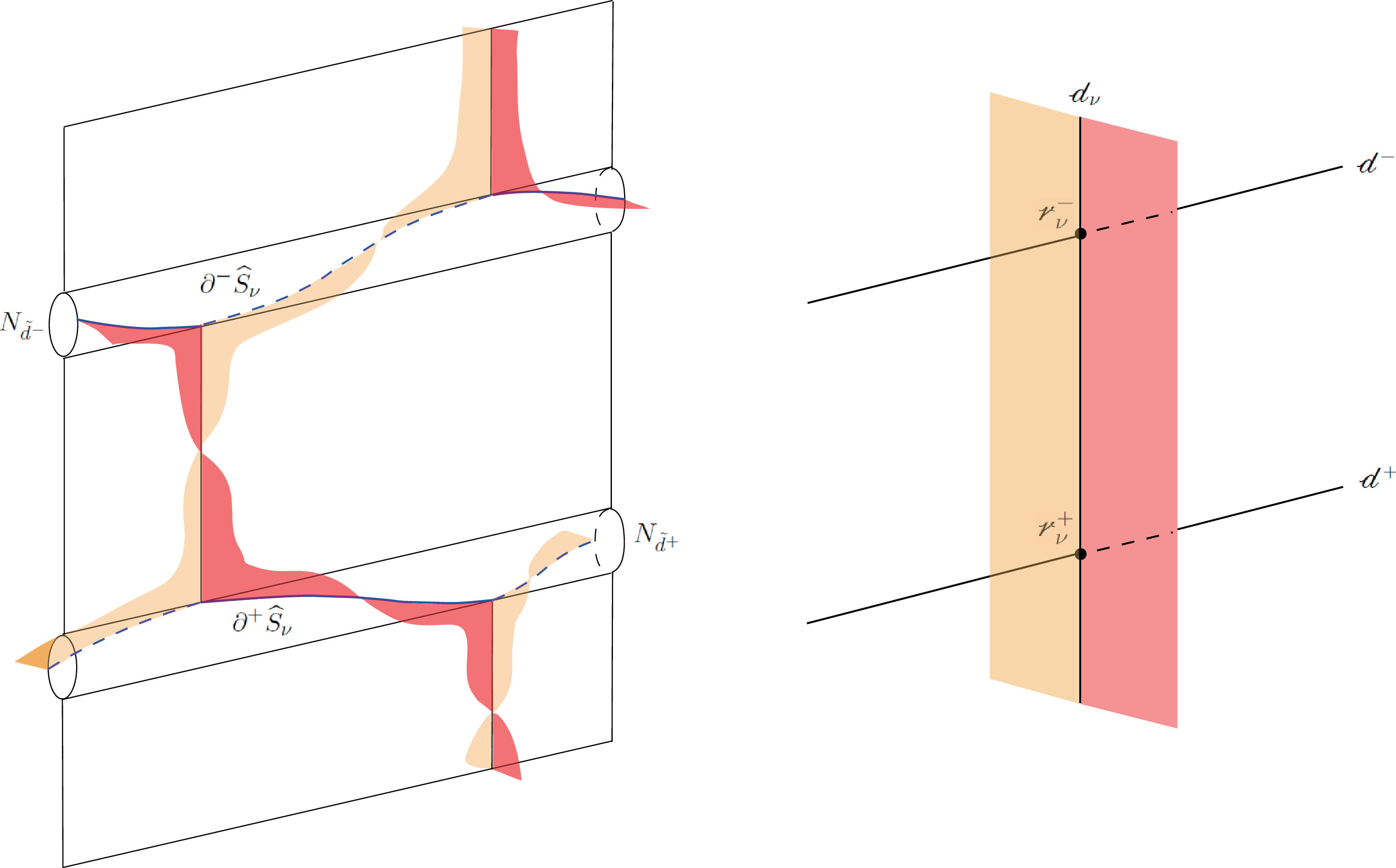}
    \caption{The surface $\widehat S_\nu$, the tori $N_{\tilde d^-},N_{\tilde d^+}$, the curves $\partial^-\widehat S_\nu, \partial^+\widehat S_\nu$ in $\widehat U$. The projections of some of these objects in $\mathscr{U}$: the surface $\mathscr{S}_\nu$, the orbits $\mathscr{d}^-,\mathscr{d}^+$, the curve $\mathscr{d}_\nu$; the points $\mathscr{r}_\nu^-, \mathscr{r}_\nu^+$.}
    \label{f.local-Sigma}
\end{figure}

\begin{proof}
Fix $\nu_0\in\SS^1$. The boundary components of the surface $\widehat S_{\nu_0}$ are the curves $\partial^{-}\widehat S_{\nu_0}$ and $\partial^{+}\widehat S_{\nu_0}$. The surface $\widehat S_{\nu_0}$ is disjoint from the curves $\partial^\pm\widehat S_{\nu}$ for $\nu\neq\nu_0$. By definition, the manifold $\mathscr{U}$ is obtained from $\widehat U$ by collapsing each curve $\partial^\pm\widehat S_{\nu}$, for $\nu\in\SS^1$, to a point. Hence the surface $\mathscr{S}_{\nu_0}$ is obtained by collapsing each of the two boundary components $\partial^{-}\widehat S_{\nu_0}, \partial^{+}\widehat S_{\nu_0}$ of the surface $\widehat S_{\nu_0}$ to a point.

Now, the surface $\widehat S_{\nu_0}$ is obtained by gluing the two once punctured tori $S_{\ell,\nu_0}$ and $S_{r,\nu_0}$ on the neighbourhood of $\widehat S_{\nu_0}\cap T$ described in the previous subsection and represented in Figure~\ref{f.boundary-blow-up}). It follows that $\widehat S_{\nu_0}$ is a connected orientable genus two surface with two boundary components, \emph{i.e.}, is diffeomorphic to a closed orientable genus two surface minus two open discs. As a further consequence, $\mathscr{S}_{\nu_0}$ is homeomorphic to a closed (\emph{i.e.}, compact and boundaryless) genus two surface. See Figure~\ref{f.fibers}.

By definition, $d_{\nu_0}^-$ and $d_{\nu_0}^+$ are two disjoint arcs contained in the surface $\overline{S}_{\nu_0}$, disjoint from each other (see  Addendum~\ref{a.topological-closure}). The ends of these two arcs lie on the geodesics $\tilde d^\pm$, but their interiors are disjoint from theses geodesics. It follows that $d_{\nu_0}^-$ and $d_{\nu_0}^+$ lift homeomorphically to two disjoint arcs $\hat d_{\nu_0}^-$ and $\hat d_{\nu_0}^+$ in $\widehat U$, whose ends lie on the tori $N_{\tilde d^\pm}$ and whose interiors are disjoint from these tori. From the equations of the curves $\partial^-\widehat S_{\nu_0}$ and $\partial^+\widehat S_{\nu_0}$ in Proposition~\ref{p.par-scc}, we get that both the arcs $\hat d_{\nu_0}^-$ and $\hat d_{\nu_0}^+$ join the curve $\partial^-\widehat S_{\nu_0}$ to the curve $\partial^+\widehat S_{\nu_0}$. It follows that the projections $\mathscr{d}_{\nu_0}^-$ and $\mathscr{d}_{\nu_0}^+$ of $\hat d_{\nu_0}^-$ and $\hat d_{\nu_0}^+$ are two arcs with disjoint interiors, joining the point $\mathscr{r}_{\nu_0}^-$ to the point $\mathscr{r}_{\nu_0}^+$ in the surface $\mathscr{S}_{\nu_0}$. Hence, the union $\mathscr{d}_{\nu_0}=\mathscr{d}_{\nu_0}^-\cup\mathscr{d}_{\nu_0}^+$ is a simple closed curve in $\mathscr{S}_{\nu_0}$ passing through the points $\mathscr{r}_{\nu_0}^-$ to the point $\mathscr{r}_{\nu_0}^+$. This simple closed curve  divides $\mathscr{S}_{\nu_0}$ into the two once-punctured tori, since its lift $\mathscr{d}_{\nu_0}^-\cup\mathscr{d}_{\nu_0}^+\cup\partial^{-}\widehat S_{\nu_0}\cup\partial^{+}\widehat S_{\nu_0}$ in $\widehat U$ divides the lift $\widehat S_{\nu_0}$ of the surface $\mathscr{S}_{\nu_0}$ into the two once-punctured tori $S_{\ell,\nu_0}$ and $S_{r,\nu_0}$, and since the later project homeomorphically in $\mathscr{U}$. See again Figure~\ref{f.fibers}.
\end{proof}

\begin{proposition}
\label{proposition:fibration-3}
The surfaces $(\mathscr{S}_\nu)_{\nu \in \SS^1}$ are the fibers of a fibration of $\mathscr{p}:\mathscr{U}\to\SS^1$.
\end{proposition}

\begin{proof}
The surfaces $(\mathscr{S}_\nu)_{\nu \in \SS^1}$ are the leaves of a codimension one foliation on $\mathscr{U}$. All the leaves of this foliation are closed and orientable, and $\mathscr{U}$ is also orientable, hence this foliation is a fibration. 
\end{proof}

 \subsection{Position of certain orbits and their stable manifolds with respect to the fibration} \label{ss.loc-smfd}

 Our next task is to prove that the Anosov flow $\mathscr{X}^t$ admits many horizontal periodic orbits, \emph{i.e.}, periodic orbits included in some fibers of the fibration $\mathscr{p}:\mathscr{U}\to\SS^1$, and to compute the twist number of the local stable manifolds of these periodic orbits with respect to the fibration (see Section~\ref{subsection:twist}). 

\begin{proposition}
\label{p.horizontalorbits} 
The following periodic orbits of the Anosov flow $\mathscr{X}^t$ are horizontal:
\begin{itemize}
\item \(\mathscr{a}_\ell := \mathscr{h}(\tilde{a}_\ell^{+})\) and \(\mathscr{a}_r := \mathscr{h}(\tilde{a}_r^{+})\) are included in the fiber \(\mathscr{S}_0\).
\item \(\mathscr{b}_\ell := \mathscr{h}(\tilde{b}_\ell^{+})\) and \(\mathscr{b}_r := \mathscr{h}(\tilde{b}_r^{+})\) are included in the fiber \(\mathscr{S}_{\frac{3\pi}{2}}\).
\end{itemize}
Moreover, the local stable manifolds of all these orbits are horizontal with respect to the fibration $\mathscr{p}$ (see Definition \ref{definition:twist-1}).
\end{proposition}

\begin{proof}
The arguments are similar for all the orbits. We treat the case of  $\mathscr{a}_r$.

According to Corollary~\ref{c.geodesics-in-section}, the lifted geodesic \(\widetilde{a_r}^+\) is included in \(S_0 - \pi^{-1}(d)\). Hence \(\mathscr{h}(\widetilde{a_r}^{+})=\mathscr{a}_r\) is included in  \(\mathscr{i}(S_0 -\tilde d^\pm)=\mathscr{S}_0-\mathscr{d}^\pm\). 

In order to study the position of the local stable manifold of this orbit, we consider three sections of the projectivized normal bundle $N_{\tilde a_r^+}$ in $U$:
\begin{itemize}
\item a section $\sigma_{W_{\text{loc}}^s}$ defined by stable direction, 
\item a section $\sigma_{S_{0}}$ defined by the direction tangent to the surface $S_{0}$,
\item a section $\sigma_{\mathrm{Fiber}}$ defined by the direction of the fibers of the unit tangent bundle $U$.
\end{itemize}
By construction, the surface $S_{0}$ is a section of the unit tangent bundle over $S - d$. In particular, $S_{0}$ is transverse to the fibers of the unit tangent bundle. On the other hand, it is well-known that the local stable manifolds of the orbits of $X^t$ are transverse to the fibers of the unit tangent bundle. Consequently, $W_{\text{loc}}^s(\widetilde{a_r}^{+},X^t)$ is also  transverse to the fibers of the unit tangent bundle. This means that  neither $\sigma_{W_{\text{loc}}^s(\widetilde{a_r}^{+})}$ nor $\sigma_{S_{r,0}}$ intersects $\sigma_{\mathrm{Fiber}}$. Hence, the homological intersection of $\sigma_{W_{\text{loc}}^s(\widetilde{a_r}^{+})}$ and $\sigma_{S_{r,0}}$ must be zero. By definition, this means that 
$$\mathrm{Twist}_{\widetilde{a_r}^{+}}\left(W^s_{\mathrm{loc}}\left(\widetilde{a_r}^{+},X^t\right),\,S_{0}\right)=0.$$
Pushing all the object by $\mathscr{i}$, one gets that 
$$\mathrm{Twist}_{\mathscr{a}_r}\left(W^s_{\text{loc}}\left(\mathscr{a}_r,\mathscr{X}^t\right),\,\mathscr{S}_{0}\right)=0.$$
In other words, the local stable manifold $W_{\text{loc}}^s(\mathscr{a}_r,\mathscr{X}^t)$ is horizontal. 
\end{proof}

The case of the periodic orbit $\mathscr{c} = \mathscr{i}(\widetilde{c}^+)$ is considerably subtler. 

\begin{proposition}
\label{p.intnumber1}
The periodic orbit \(\mathscr{c} = \mathscr{i}(\widetilde{c}^+)\) of the flow  \(\mathscr{X}^t\) is included in the surface~\(\mathscr{S}_0\). Moreover, the twist number of the local stable manifold of \(\mathscr{c}\) with respect to \(\mathscr{S}_0\) is $+1$:
\[
\operatorname{Twist}_{\mathscr{c}}\left( W_{\mathrm{loc}}^s (\mathscr{c}, \mathscr{X}^t), \mathscr{S}_0 \right) = +1.
\]
\end{proposition}

\begin{proof}
The first assertion is straightforward to check. Indeed, recall that the construction of the surfaces\(\{\overline{S}_\nu\}_{\nu \in  \SS^1}\) implies that the lifted geodesic \(\tilde c^+\) is included in \(\overline S_0\) (see Corollary \ref{c.geodesics-in-section}). Moreover, the orbit \(\widetilde{c}^+\) is disjoint from the surgeried orbits \(\widetilde{d}^{+}\) and \(\widetilde{d}^{-}\). Hence, \(\mathscr{c}=\mathscr{i}(\widetilde {c}^+)\) is included in \(\mathscr{i}(\overline{S}_0-\tilde d^\pm) \subset \mathscr{S}_0\).

So we are left to compute the twist number. Using \(\mathscr{i}^{-1}\) we send all the objects in $U = T^1S$: 
\[
\operatorname{Twist}_{\mathscr{c},\mathscr{U}} \left( W_{\mathrm{loc}}^s (\mathscr{c}, \mathscr{X}^t), \mathscr{S}_0 \right) = \operatorname{Twist}_{\tilde c^+,U} \left( W_{\mathrm{loc}}^s (\tilde c^+, X^t), \overline{S_0} \right).
\]

Now we are left to compute the integer $\operatorname{Twist}_{\tilde c^+} \left( W_{\mathrm{loc}}^s (\tilde c^+, X^t), \overline{S_0} \right)$. By definition (see Definition~\ref{definition:twist-1}), this integer is equal to the intersection number $\operatorname{Int}(\sigma_{W_{\mathrm{loc}}^s}, \sigma_{\overline S_0})$ where $\sigma_{W_{\mathrm{loc}}^s}$ and $\sigma_{\overline S_0}$ are some dynamically oriented sections of the projectivized normal bundle $N_{\tilde c^+}=\PP^+(TM/\RR X)_{|\tilde c^+}$. More precisely,  $\sigma_{W_{\mathrm{loc}}^s}$ is one of the two opposite sections defined by the weak stable planes along $\tilde c^+$ and $\sigma_{\overline S_0}$  one of the two opposite sections defined  by the tangent plane to the surface $S_{r,0}$ along $\tilde c^+$. 

To compute $\operatorname{Int}(\sigma_{W_{\mathrm{loc}}^s}, \sigma_{\overline S_0})$, we introduce another section of $N_{\tilde c^+}$: we denote by $\sigma_{\mathrm{Fiber}}$ the section corresponding to the direction of the oriented fibers of the unit tangent bundle $U \to S$. We have already used the classical fact that the weak stable directions of the geodesic flow $X^t$ are everywhere transverse to the fibers of $U$. In other words, the sections $\sigma_{\mathrm{Fiber}}$ and $\sigma_{W_{\mathrm{loc}}^s}$ do not intersect. In particular, their homological intersection \(\operatorname{Int}(\sigma_{W_{\mathrm{loc}}}, \sigma_{\mathrm{Fiber}})\) is equal to $0$, and therefore
\[
\operatorname{Int}(\sigma_{W_{\mathrm{loc}}^s}, \sigma_{\overline S_0}) = \operatorname{Int}(\sigma_{\mathrm{Fiber}}, \sigma_{\overline S_0}).
\]
So we are left to compute \(\operatorname{Int}(\sigma_{\mathrm{Fiber}}, \sigma_{\overline S_0})\). 

For this purpose, let us first recall that, by construction, $\overline{S_0}-T=S_0$ is a section of the unit tangent bundle $U-T=T^1(S-d)\to S-d$. This means that, over $S - d$, the surface \(\overline{S_0}\) is never tangent to the fibers. Equivalently, the only possible intersection of the sections \(\sigma_{\mathrm{Fiber}}\) and \(\sigma_{\overline S_0}\) must be over points of the geodesic $d$. Now, recall that the geodesics \(c\) and \(d\) intersect at two points \(p\) and \(q\). We denote by $\tilde{p}$ and $\tilde{q}$ the lifts of $p$ and $q$ in $U$ that belongs to the lifted geodesic $\tilde{c}$.  Then the sections \(\sigma_{\mathrm{Fiber}}\) and \(\sigma_{\overline S_0}\) can intersect only at $\tilde{p}$ and $\tilde{q}$.  Note that this already implies that
\[
\left| \operatorname{Int}(\sigma_{\mathrm{Fiber}}, \sigma_{\overline S_0}) \right| \leq 2.
\]
In order to be more precise, we need to take care of the orientations.

Recall that, according to the proof of Proposition \ref{p.top-bund}, the boundary $\partial S_{r,0}$ of the surface $\overline S_{r,0}$ can be decomposed into four arcs: an arc of the lifted geodesic $\tilde d^+$, an arc of the lifted geodesic $\tilde d^-$, an arc of the fiber over $p=p_0$ and an arc of the fiber over $q=p_\pi$ (see Figure~\ref{f.bdSt+}). The first two arcs are tangent to the direction $\mathbb{R} \partial_\theta$, while the other two are tangent to the direction $\mathbb{R} \partial_\varphi$. Observe that the vectors $-\partial_\theta$ and $\partial_\varphi$ induce the same orientation of the curve $\partial S_{r,0}$ (see again Figure~\ref{f.bdSt+}). Also note that $\partial_\rho$ is transverse to the torus $T = T_d^1S$ along $\widetilde{d}^{\pm}$, and therefore lifts as a vector field $\widetilde{\partial_\rho}$ along $\partial S_{r,0} \subset U$ , which is tangent to $\overline S_0$ and transverse to $\partial S_{r,0}$. 

The points $\widetilde p$ and $\widetilde{q}$ belong to the two vertical arcs of $\partial S_{r,0}$ (actually, they are precisely the middle of these arcs with respect to the parametrization by $\varphi$).

The surface $S$ is equipped with the orientation for which $\left( \partial_\rho, \partial_\theta \right)$ is a direct basis. Since $S_{r,0}$ is a section of the unit tangent bundle $U_r=T^1 S_r\to S_r$ and $S_r$ is a subsurface of $S$, the orientation of the surface $S$ can be lifted to an orientation of $S_{r,0}$, which induces an orientation of $\overline S_0$ since $S_{r,0}$ is a subsurface of $S_0$. Note that this orientation of $\overline S_0$ is, however, not the orientation induced by that of $S_\ell$ (see Figure~\ref{f.computation-twist}).

Since the vectors $-\partial_\theta$ and $\partial_\varphi$ induce the same orientation of the curve $\partial S_{r,0}$, it follows that the orientation of $\overline S_0$ at $\widetilde p$ and $\widetilde{q} $ is given by the directed basis $(\widetilde{\partial_\rho},-\partial_\varphi)$. In the surface $S$, the geodesic $c$ is tangent to $\RR.\partial_\rho$ both at $p$ and $q$. Moreover, the vector $\partial_\rho$ points according to the orientation of $c$ at $p$, and opposite to the orientation of $c$ at $q$. It follows that $X=\widetilde{\partial_\rho}$ at $\widetilde p$ and $X=-\widetilde{\partial_\rho}$ at $\widetilde{q}$. As a further consequence, $\left(X,-\partial_\varphi\right)$ is a direct basis for the orientation of $\overline S_{0}$ at $\tilde p$, and  $\left(-X,-\partial_\varphi\right)$, or equivalently $\left(X,\partial_\varphi\right)$, is a direct basis for the orientation of $\overline S_{0}$ at $\widetilde{q}$. 

The surface $\overline{S_0}$ induces two opposite sections of $N_{\tilde c^+}$, and  $\sigma_{\overline{S_0}}$ can be chosen as any of these two opposite sections. For the computations we make the following choice: we choose \(\sigma_{\overline S_0}\) such that, for any point $r\in\tilde c^+$, the equivalence class \(\sigma_{\overline S_0}(r)\) in $\PP^+(T_rM/\RR X(r))$ is represented by a vector \(v \in T_r \overline S_0 \setminus \RR X(r)\) where \((X(r), v)\) is a direct basis of \(\overline S_0\). As a consequence, 
$$\sigma_{\overline S_0}(\tilde p) = \left[ -\partial_\varphi(\tilde p) \right] \mbox{ and } \sigma_{\overline S_0}(\widetilde{q}) = \left[ \partial_\varphi(\widetilde{q}) \right].$$
Now recall that the fibers of the unit tangent bundle are tangent to $\partial\varphi$ and oriented by this vector field. Hence, 
$$\sigma_\mathrm{Fiber}(\tilde p)=\left[\partial_\varphi(\tilde p)\right]\mbox{ and }\sigma_\mathrm{Fiber}(\widetilde {q})=\left[\partial_\varphi(\tilde q)\right].$$
Hence, the sections $\sigma_{\overline S_0}$ and $\sigma_\mathrm{Fiber}$ coincide at $\widetilde {q}$ and are opposite at $\widetilde {p}$. In particular, $\widetilde{q}$ is the only point where the sections $\sigma_{\overline S_0}$ and $\sigma_{\text{Fiber}}$ intersect (see Figure~\ref{f.computation-twist}), and thus
\[
\left| \text{Int} \left(\sigma_{\text{Fiber}},\sigma_{\overline S_0} \right) \right| = 1.
\]

\begin{figure}
 \centering
    \includegraphics[scale=0.37]{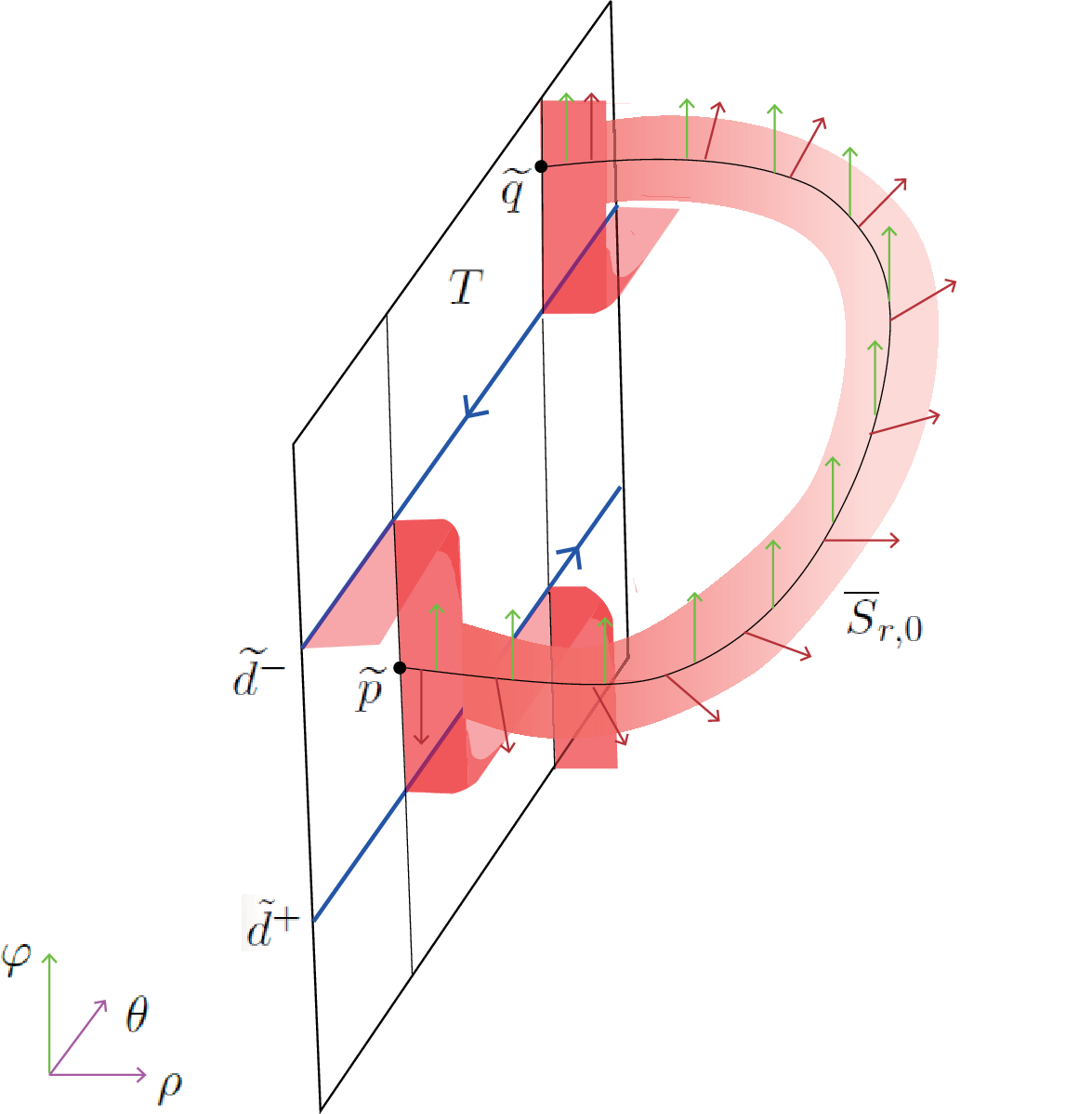}
\caption{The position of the boundary of the surface $\overline S_{r,0}$ with respect to the geodesics $\tilde d^\pm$, the fibers $T^1_{p}S$ and $T^1_{q}S$, and the orbit $\tilde c^+$. The sections $\sigma_{\overline{S}_{r,0}}$ and $\sigma_{\mathrm{Fiber}}$ along the lifted geodesic arc $\tilde c^+\cap S_{0,r}$ correspond to the red vectors and the green vectors respectively.}
\label{f.computation-twist}
\end{figure}

It remains to determine the sign of the intersection at $\widetilde{q}$. For this purpose, we look at the surface $\overline S_{0}$ in the $(\rho, \theta, \varphi)$ coordinate system. Recall that $\widetilde{q}$ has coordinates $(0, \pi, \frac{\pi}{2})$. From the proof of Proposition \ref{p.par-scc}, the surface $\overline S_{r,0}$ is given by the equations:
\[
\begin{cases}
\cos \varphi = - \dfrac{\cosh \rho \cdot \sin \theta}{\sqrt{\rho^2 \cos^2 \theta + \cosh^2 \rho \cdot \sin^2 \theta}} \\
\sin \varphi = - \dfrac{\rho \cos \theta}{\sqrt{\rho^2 \cos^2 \theta + \cosh^2 \rho \cdot \sin^2 \theta}}
\end{cases}
\]
Linearizing the first equation for $(\rho, \theta, \varphi) = (d\rho, \pi+d\theta, \frac{\pi}{2} +d\varphi)$ with variations $(d\rho, d\theta, d\varphi)$:
\[
-d\varphi =  \frac{d\theta}{\sqrt{d\rho^2 + d\theta^2}} \quad \text{or equivalently} \quad d\theta +\sqrt{d\rho^2 + d\theta^2}d\varphi =0.
\]
This means that, in the neighbourhood of $\tilde q$, 
$$\sigma_{\overline S_0}\sim \left[-\sqrt{\rho^2+(\theta-\pi)^2}\partial_\theta+\partial_\varphi\right].$$
On the other hand, recall that 
$$\sigma_{\mathrm{Fiber}}=\left[\partial_\varphi\right].$$
Now the sections are dynamically oriented, \emph{i.e.}, oriented by the flow orientation of $\tilde c$. Near $\widetilde{q}$, this is the orientation for which $\rho$ is decreasing. Moreover, since $\left(\partial_\rho,\partial_\theta,\partial_\varphi\right)$ is a direct basis of $U$, it follows that $\left(\left[\partial_\theta\right],\left[\partial_\varphi\right]\right)$ is an indirect basis of $TU/\RR.X$ (for the orientation of $N_{\tilde c^+}$). For a point on the lifted geodesic $\tilde c^+$ that is just before $\widetilde{q}$ according to the orientation of this geodesic, $\rho$ is very small and positive, and therefore
$$\mathrm{Angle}\left(\sigma_{\mathrm{Fiber}},\sigma_{\overline S_0})\right)=\mathrm{Angle}\left(\left[\partial_\varphi\right],\left[-\sqrt{\rho^2+(\theta-\pi)^2}\partial_\theta+\partial_\varphi\right]\right)$$
is positive for the orientation of the plane $TU/\RR.X$. This means that this angle decreases as the geodesic passes through $\widetilde{q}$, and therefore 
\(
\text{Int} \left( \sigma_{\text{Fiber}}, \sigma_{S_0} \right) = +1.
\)  
We conclude that 
\(
\text{Twist} \left( W_{\text{loc}}^s (\mathscr{c}, \mathscr{X}^t), \mathscr{S}_0 \right) = +1\). 
\end{proof}

 \subsection{Monodromy of the fibration and proof of Theorem~\ref{p.perorb-lk}}
 \label{ss.Monodromy}
 
In the previous sections, we have constructed a closed oriented $3$-manifold $\mathscr{U}$, equipped with a transitive Anosov flow $\mathscr{X}^t$ so that there is a fibration $\mathscr{p}:\mathscr{U}\to\SS^1$. Our next task is to determine the monodromy of this fibration: we will prove that it is isotopic to the right Dehn twist $\tau_{\mathscr{d}_0}^{-2}$ along the curve $\mathscr{d}_0\subset\mathscr{S}_0=\mathscr{p}^{-1}(\{0\})$(see Notation \ref{notation-fibered-manifold} for $\mathscr{d}_0$). Actually, while doing this, we will construct a quite explicit diffeomorphism between $\mathscr{U}$ and the mapping torus of the Dehn twist $\tau_{\mathscr{d}_0}^{-2}$, which will allow us to complete the proof of Theorem~\ref{p.perorb-lk}. 

Recall that the fibers of the fibration $\mathscr{p}$ are the surfaces $(\mathscr{S}_\nu)_{\nu\in\SS^1}$. In order to compute the monodromy of the fibration $\mathscr{p}:\mathscr{U}\to\SS^1$, we will construct a non-singular vector field $\mathscr{W}$ on $\mathscr{U}$ transverse to these fibers. The isotopy class of the monodromy of the fibration will therefore be the isotopy class of the Poincaré return map $\mathscr{f}_{\mathscr{W}} : \mathscr{S}_0 \to \mathscr{S}_0$ of $\mathscr{W}$ on the fiber $\mathscr{S}_0$. The vector field $\mathscr{W}$ will be chosen so that $\mathscr{f}_{\mathscr{W}}$ will be the identity everywhere except on a small annular neighbourhood of the curve $\mathscr{d}_0$. 

Before proceeding to the construction of the vector field $\mathscr{W}$, recall that there is a canonical identifcation
$$\mathscr{i}:\left(\mathscr{U}-\mathscr{d}^\pm,\mathscr{X}_{|\mathscr{U}-\mathscr{d}^\pm}\right)\longrightarrow\left(U-\tilde d^\pm,X_{|U-\tilde d^\pm}\right).$$
Also recall that we have defined a coordinate system $(\rho,\theta,\varphi):U_A\to [-\epsilon,\epsilon]\times\SS^1\times\SS^1$ where $U_A=T^1_{|A}S$ is a neighbourhood of the torus $T=T^1_{|d} S$ in the unit tangent bundle $U=T^1 S$. Now we pick $\epsilon'$ so that $0<\epsilon'<\epsilon$, and we denote by  $U_{A'}$ the subset of $U_A$ corresponding $-\epsilon'\leq \rho\leq\epsilon'$. Then we denote by $\mathscr{U}_{A}$ and $\mathscr{U}_{A}'$ the regions of $\mathscr{U}$ corresponding to $U_{A}$ and $U_{A}'$, \emph{i.e.},
\begin{align*}
\mathscr{U}_A &= \mathscr{i}\left(U_A -\left(\tilde{d}^+ \cup \tilde{d}^-\right)\right) \cup (\mathscr{d}^{+} \cup \mathscr{d}^{-}), \\
\mathscr{U}_{A}' &= \mathscr{i}\left(U_{A'} - \left(\tilde{d}^{+} \cup \tilde{d}^-\right)\right) \cup (\mathscr{d}^{+} \cup \mathscr{d}^{-}).
\end{align*}
  
\begin{lemma}\label{l.Z-tr-Sig}
There exists a non-singular vector field $\mathscr{W}$ on $\mathscr{U}$ such that:
\begin{enumerate}
    \item $\mathscr{W}$ is transverse to the fibers $\{\mathscr{S}_\nu\}_{\nu \in \SS^1}$.
    \item The three regions $\mathscr{U}_A'$, $\mathscr{U}_A-\mathscr{U}_A'$ and  $\mathscr{U}-\mathscr{U}_A$ and are invariant under the flow of $\mathscr{W}$.
    \item Every orbit of $\mathscr{W}$ in $\mathscr{U}_A'$ and in $\mathscr{U}-\mathscr{U}_A$ is closed, and intersects the surface $\mathscr{S}_\nu$ at exactly one point for each $\nu\in\SS^1$. Actually the orbits of $\mathscr{W}$ in $\mathscr{U}-\mathscr{U}_A$ are precisely the image under the orbit equivalence $\mathscr{i}$ of the fibers of $U-U_A$.
    \item On $U_A-U_A'=\mathscr{i}^{-1}(\mathscr{U}_A-\mathscr{U}_A')$ the vector field $W:=\mathscr{i}^{-1}_*\mathscr{W}$ reads:
    \[
    \mathscr{W}( {\rho},  {\theta},  {\varphi}) = \kappa( {\rho},  {\theta},  {\varphi}) \left( \sin(\chi( {\rho})) \partial_{ \varphi} + \cos(\chi( \rho)) \partial_{ \theta} \right),
    \]
    where $\chi : [-\epsilon, \epsilon] \to [-\frac{\pi}{2}, \frac{\pi}{2}]$ is an increasing $C^\infty$ function such that:
    \[
    \chi(-\epsilon) = -\frac{\pi}{2}, \quad \chi|_{[-\epsilon', \epsilon']} = 0, \quad \chi(\epsilon) = +\frac{\pi}{2},
    \]
    and $\kappa$ is a $C^{1}$ positive factor.
\end{enumerate}
\end{lemma}

\begin{proof}[Proof of Lemma~\ref{l.Z-tr-Sig}]
We will define a vector field $W$ on the unit tangent bundle $U$, and then push it by the homomeorphism $\mathscr{i}$ to get the vector field $\mathscr{W}$ on $\mathscr{U}$. 

Let $\chi: [-\epsilon_0, \epsilon_0] \to [-\frac{\pi}{2}, \frac{\pi}{2}]$ be a function satisfying the required properties in item (3). We first consider the vector field $W$ defined on $U_A \subset U$ by
$$W(\rho, \theta, \varphi) := \sin(\chi(\rho)) \partial_\varphi + \cos(\chi(\rho)) \partial_\theta.$$
We will extend $W$ to a vector field on $U$. For this purpose, recall that the torus $T$ divides $U$ into two halves $U_\ell$ and $U_r$, and that $U_A$ is a collar neighborhood of $T$ in $U$ saturated by fibers (see subsection~\ref{ss.coordinate-system}). On the torus $\{\rho = +\epsilon\}\subset U_r$, the vector field $W$ is equal to $\partial_\varphi$. Since $\partial_\varphi$ is positively tangent to the fibers of the unit tangent bundle $U$, we can extend $W$ on $U_r-U_A$ by assigning $W$ to be a non-zero vector field positively tangent to the fibers. Similarly, on the torus $\{\rho = -\epsilon\}\subset U_\ell$, the vector field $W$ equals $-\partial_\varphi$, hence it is negatively tangent to the fibers of the unit tangent bundle $U$. So we may extend $W$ on $U_\ell-U_A$ by setting it to a non-zero vector field negatively tangent to the fibers. Now $W$ is defined on the entire closed manifold $U$.

On $U_A$, the vector field $W$ is equal to $\partial_\theta$. This implies that the lifted geodesics $\tilde{d}^+$ and $\widetilde{d}^-$ are two orbits of $W$, up to orientations (indeed recall that $\widetilde{d}^+$ reads ($\rho = 0$, $\varphi = 0$) and $\widetilde{d}^-$ read  ($\rho = 0$, $\varphi = \pi$) in the $(\rho,\theta,\varphi)$ coordinate system). It also implies that, when we blow-up the geodesics $\widetilde{d}^\pm$, then the lift of $W$ also reads $\partial_\theta$ in the $(\theta, \psi)$ coordinate system on the projectivized normal bundle $N_{\widetilde{d}^{\pm}}$ of $\tilde d^\pm$ boundary tori $N_{\widetilde{d}^{\pm}}$. In particular, the lift of $W$ is transverse to the 
circles $(\psi=\theta + \text{const.})$ on $N_{\widetilde{d}^{\pm}}$. It follows that $W$ induces a vector field $\mathscr{W}$ on $\mathscr{U}$ (which is obtained by collapsing each circle $(\psi=\theta + \text{const.})$ to a point). More precisely, $W$ induces a half-line field on $\mathscr{U}$, and we choose a vector field $\mathscr{W}$ tangent to this half-line field. By construction, $\mathscr{d}^\pm$  are two periodic orbits of $\mathscr{W}$, and the canonical identification $\mathscr{i}:U-\tilde d^\pm\to \mathscr{U}-\mathscr{d}^\pm$ maps the orbits of $W$ different from $\tilde d^\pm$ to the orbits of $\mathscr{W}$ different form $\mathscr{d}^\pm$. Moreover, item (4) is an immediate consequence of the construction. 

By construction, the regions $U_A'$ and $U-U_A$ are invariant under the flow of the vector field $W$, and all the orbits of $W$ in these regions are periodic: 
\begin{itemize}
\item on $U-U_A$, the reason is that the orbits of $W$ are fibers (up to orientation)~;
\item on $U_A'$, the reason is that $W$ is equal to $\partial_\theta$, hence the orbits are the circles $(\rho = \text{const.}$, $\varphi = \text{const.})$.
\end{itemize}
It follows that the regions $\mathscr{U}_A'$ and $\mathscr{U}-\mathscr{U}_A'$ are invariant under the flow of $\mathscr{W}$, and all the orbits of $\mathscr{W}$ in these regions are periodic. This proves items (2) and (3).

So we are left to check item (1), i.e. the transversality of the vector field $\mathscr{W}$ to the fibers $\{\mathscr{S}_\nu\}_{\nu \in \SS^1}$. Let us fix $\nu\in\SS^1$.
\begin{itemize}
    \item Let us consider the region $\mathscr U-\mathscr U_A$. The vector field $W$ is tangent to the fibers of the unit tangent bundle $U \to S$ in $U - U_A$. Therefore $W$ is transverse to the surfaces $S_{r,\nu}$ and $S_{\ell,\nu}$ on $U - U_A$ (recall that $S_{r,\nu}$ and $S_{r,\nu}$ are partial sections of the unit tangent bundle). It follows that $\mathscr W$ is transverse to the surface $\mathscr{S}_\nu$ on $\mathscr{U} -\mathscr{U}_A$.
    \item Now we turn to the region $\mathscr{U}_A$. Observe that $W$ has no component in the $\partial_\rho$ direction. This implies that the transversality of $W$ to the surface $\overline S_\nu$ on $U_A$ can be checked inside each torus $T_{\rho_0}=\{\rho=\rho_0\}$ for $\rho_0\in [-\epsilon,\epsilon]$. 
    \begin{itemize}
        \item For $\rho_0>0$, we observe that the restriction of the vector field $W$ to the torus $T_{\rho_0}$ is pointing in the upper-right quadrant with respect to the basis $\left(\partial\theta,\partial\varphi\right)$. On the other hand, formula~\eqref{e.boundary-r} implies that $S_\nu\cap T_{\rho_0} = S_{r,\nu} \cap T_{\rho_0}$ is the graph of a decreasing function $\varphi=s_{r,\nu}(\theta)$. It follows that $W$ is transverse to $\overline S_\nu$ in $T_{\rho_0}$.
        \item For $\rho_0<0$, the argument is similar, but the restriction of $W$ to $T_{\rho_0}$ is pointing in the lower-right quadrant, and formula~\eqref{e.boundary-l} implies that $\overline S_\nu\cap T_{\rho_0} = S_{\ell,\nu} \cap T_{\rho_0}$ is the graph of an increasing function $\varphi=s_{\ell,\nu}(\theta)$, which implies that $W$ is again transverse to $S_\nu$ in $T_{\rho_0}$.
        \item Finally, for $\rho_0=0$, the restriction of $W$ to the torus $T_0=T$ is equal to $\partial_\theta$. Our analysis in Subsection~\ref{ss.Clo-Sv} shows that $(\overline S_\nu\cap T)-(\tilde d^-\cup\tilde d^+)$ is made of two arcs of fibers, hence tangent to $\partial_\varphi$. In particular, $W$ is transverse to $S_\nu$ in $T_0-\tilde d^\pm$.
    \end{itemize}
    The three items above show that the vector field $W$ is transverse to the surface $\overline S_\nu$ on $U_A-\tilde d^\pm$. Pushing all the objects by $\mathscr{i}$, one gets that the vector field $\mathscr{W}$ is transverse to the surface $\mathscr{S}_\nu$ on $\mathscr U_A - \mathscr{d}^\pm$. 
    \item Finally, the vector field $\mathscr{W}$ is also transverse to the surface $\mathscr{S}_\nu$ near the orbits $\mathscr{d}^\pm$ since $\mathscr{S}_\nu$ is transverse to $\mathscr{d}^\pm$, and since by construction, $\mathscr{d}^\pm$ are orbits of $\mathscr{W}$. 
\end{itemize}
So we have proved that the vector field $\mathscr{W}$ is transverse to the surface $\mathscr{S}_\nu$ everywhere. This completes the proof of item (1) and the proof of the lemma.
\end{proof}

We will further denote $\mathscr{A}_0 = \mathscr{U}_A\cap\mathscr{S}_0\mbox{ and }\mathscr{A}_0' = \mathscr{U}_A'\cap\mathscr{S}_0$. Clearly, $\mathscr{A}_0$ and $\mathscr{A}_0'$ are two collar neighbourhoods of the curve $\mathscr{d}_0$ in the surface $\mathscr{S}_0$, and $\mathscr{A}_0'\subset\mathscr{A}_0$. 

\begin{proposition}
\label{prop:monodromy}
The Poincaré return map $\mathscr{f}_\mathscr{W}:\mathscr S_0\to\mathscr S_0$ is supported in $\mathscr{A}_0-\mathscr{A}_0'$ and is isotopic, by an isotopy supported in $\mathscr{A}_0$, to $\tau_{\mathscr{d}_0}^{-2}$, where $\tau_{\mathscr{d}_0}$ is the left Dehn twist along the curve $\mathscr{d}_0$ in the surface $\mathscr{S}_0$. 
\end{proposition}

As a straightforward consequence, we obtain:

\begin{corollary}\label{coro:mono}
The monodromy of the fibration $\mathscr{p}:\mathscr{U}\to\SS^1$ is the Dehn twist $\tau_{\mathscr{d}_0}^{-2}$.
\end{corollary}

\begin{proof}[Proof of Proposition~\ref{prop:monodromy}] 
Item (2) of Lemma~\ref{l.Z-tr-Sig} implies that the annuli  $\mathscr{A}_0$ and $\mathscr{A}_0'$ are invariant under the return map $\mathscr{f}_{\mathscr{W}}$. Item~(3) implies that $\mathscr{f}_{\mathscr{W}}$ is supported on $\mathscr{A}_0-\mathscr{A}_0'$. 

So we are left to analyse the restriction of $\mathscr{f}_{\mathscr{W}}$ to each of the two connected components of $\mathscr{A}_0-\mathscr{A}_0'$. 
Using the homeomorphism $\mathscr{i}$, it is equivalent to analyse the return map $f_W$ of the vector field $W$ on $\mathscr{i}^{-1}(\mathscr{A}_0-\mathscr{A}_0')\subset (S_0\cap (U_A-U_A')$. For that purpose, we will use the explicit expression of the vector field $W=\mathscr{i}^{-1}_*\mathscr{W}$ will use the $(\rho, \theta, \varphi)$-coordinate system (item (4) of Lemma~\ref{l.Z-tr-Sig}). 

The surface $S_0\cap U_0$ is a graph over the annulus $A=\{(\rho,\theta)\in [-\epsilon,\epsilon]\times \SS^1\}$, hence, the tori $\{ \rho = \text{constant} > 0 \}$ induce a circle foliation of the annulus $S_{r,0} \cap (U_A- U_A')\subset S_{r,0}$. The leaves are isotopic to $\partial \overline{S}_{r,0}$, the boundary of $S_{r,0}$ in $U$. The description of the boundary of $\overline{S}_{r,0}$ in the proof Proposition \ref{p.top-bund} implies the leaves of this circle foliation are in the homotopy class of the circle tangent to $\xi_+:=\partial_\theta - \partial_\varphi$. 

According to item (4) of Lemma \ref{l.Z-tr-Sig}, the restriction of the vector field $W$ to $U_r - U_1$ reads
\[
    W(\rho, \theta, \varphi) = \lambda(\rho, \theta, \varphi) \left( \partial_\theta - \frac{\sin(\chi(\rho))}{\cos(\chi(\rho))+\sin(\chi(\rho))} \xi^+ \right).
\]  
This shows the return map of $W$ on the circle ($\rho = \text{constant}$),  oriented by $\xi^+$, is the rigid rotation of angle:
\[
-2\pi \frac{\sin(\chi(\rho))}{\cos(\chi(\rho))+\sin(\chi(\rho))}
\]
When $\rho$ ranges from $\epsilon'$ to $\epsilon$, the angle decreases from $0$ to $-2\pi$. Recall that the orientation of $S_{0}$ was set to be defined by the lift to $S_{r,0}$ of the orientation of $S$. This orientation is such that $\left(\partial_\rho, \xi^+\right)$ is a direct basis. It follows that the isotopy class of the restriction of $P_{W}$ to the annulus $S_0 \cap \{ \epsilon' \leq \rho \leq \epsilon \}$ is  the right-Dehn twist $\tau^{-1}$ along the core of the annulus. As a further consequence, the isotopy class of the restriction of $\mathscr{W}$ to the annulus $\mathscr{S}_0 \cap \{ \epsilon' \leq  \rho \leq -\epsilon \}$ is $\tau_{\mathscr{d}_0}^{-1}$.

Similarly, $\{ \rho = \text{constant} \}$ induces a circle foliation on the annulus $S_0 \cap \{ -\epsilon \leq \rho \leq -\epsilon' \}\subset S_{\ell,0}$, with leaves homotopic to the circles tangent to $\xi^-:=-\partial_\theta - \partial_\varphi$. Writing
\[
W(\rho, \theta, \varphi) = \lambda(\rho, \theta, \varphi) \left( \partial_\theta - \frac{\sin(\chi(\rho))}{\cos(\chi(\rho)) - \sin(\chi(\rho))} \xi^- \right),
\]
we see that the return map on the circle ($\rho = \text{constant}$) oriented by the vector $\xi^-$ is the rotation of angle:
\[
-2\pi \frac{\sin(\chi(\rho))}{\cos(\chi(\rho)) - \sin(\chi(\rho))}.
\]
As $\rho$ ranges from $-\epsilon$ to $-\epsilon'$, this angle decreases from $2\pi$ to $0$. The orientation $S_0$ is such that $\left(\partial_\rho, \xi^-\right)$ is a direct basis. It follows that the return map of $\mathscr{W}$ on the annulus $\{-\epsilon\leq \rho\leq -\epsilon'\}$ is isotopic to the right Dehn twist $\tau_{\mathscr{d}_0}^{-1}$. 

Finally, the $\mathscr{f}_{\mathscr{W}}$ is isotopic in $\mathscr{A}_0$ to:
$$\tau_{\mathscr{d}_0}^{-1} \circ \tau_{\mathscr{d}_0}^{-1}=\tau_{\mathscr{d}_0}^{-2}.$$
This completes the proof of Proposition~\ref{prop:monodromy}.
\end{proof}

\begin{proof}[Proof of Theorem~\ref{p.perorb-lk}] 
The flow of the vector field $\mathscr{W}$ provides a explicit diffeomorphism 
$$\Phi_{\mathscr{W}}:\mathscr{U}\to \mathscr{M}:=\mathscr{S}_0\times [0,2\pi]_{/(x,2\pi)\sim (\mathscr{f}_\mathscr{W}(x),0)}.$$
By Proposition~\ref{prop:monodromy}, $\mathscr{f}_\mathscr{W}$ is supported in the annulus $\mathscr{A}_0$, and is isotopic to the Dehn twist $\tau_{\mathscr{d}_0}^{-2}$ in this annulus (in particular, $\mathscr{M}$ can be seen as the mapping torus of $\tau_{\mathscr{d}_0}^{-2}$). Pushing the Anosov vector field $\mathscr{X}$ by $\Phi_{\mathscr{W}}$, we get an Anosov vector field $\mathscr{Y}$ on $\mathscr{M}$. According to Proposition~\ref{p.horizontalorbits} and Proposition~\ref{p.intnumber1}, the curves $\Phi_{\mathscr{W}}(\mathscr{a}_{\ell})$, $\Phi_{\mathscr{W}}(\mathscr{a}_r)$, $\Phi_{\mathscr{W}}(\mathscr{b}_\ell)$, $\Phi_{\mathscr{W}}(\mathscr{b}_r)$ and $\Phi(\mathscr{c})$ are periodic orbits of the Anosov vector field $\mathscr{Y}$, contained in the surfaces $\mathscr{S}_0\times\{0\}$, $\mathscr{S}_0\times\{0\}$, $\mathscr{S}_0\times\{\frac{3\pi}{2}\}$, $\mathscr{S}_0\times\{\frac{3\pi}{2}\}$ and  $\mathscr{S}_0\times\{0\}$ respectively. Proposition~\ref{p.horizontalorbits} and Proposition~\ref{p.intnumber1} also guarantee that the local stable manifolds of the four first orbits above are horizontal, and that the twist number of local stable manifold of the orbit $\Phi_{\mathscr{W}}(\mathscr{c})$ with respect to the surface $\mathscr{S}_0\times\{0\}$ is $+1$. 

Let $\mathrm{pr}:\mathscr{S}_0\times [0,2\pi)\to \mathscr{S}_0$ be the vertical projection, and denote $\mathscr{a}_{\ell,0}, \mathscr{a}_{r,0}, \mathscr{b}_{\ell,0}, \mathscr{b}_{r,0}, \mathscr{c}_{0}\subset \mathscr{S}_0$ the images of the periodic orbits $\Phi_{\mathscr{W}}(\mathscr{a}_{\ell})$, $\Phi_{\mathscr{W}}(\mathscr{a}_r)$, $\Phi_{\mathscr{W}}(\mathscr{b}_\ell)$, $\Phi_{\mathscr{W}}(\mathscr{b}_r)$ and $\Phi(\mathscr{c})$ under this projection. 

\begin{claim}
\label{claim:diffeo-between-surfaces}
The surface $\mathscr{S}_0$ equipped with the curves $\mathscr{a}_{\ell,0}, \mathscr{a}_{r,0}, \mathscr{b}_{\ell,0}, \mathscr{b}_{r,0}, \mathscr{c}_{0}, \mathscr{d}_{0}$ is diffeomorphic to the surface $S$ equipped with the curves $a_\ell,b_\ell,a_r,b_r,c,d$. 
\end{claim}

Claim~\ref{claim:diffeo-between-surfaces} provides a diffeomorphism 
$$\Psi:\mathscr{M}\to M:=S\times [0,\pi]/_{(x,2\pi)\sim(\tau_d^{-2}(x),0)}$$
such that:
\begin{itemize}
\item[--] the push forward of $\mathscr{Y}$ by $\Psi$ is a transitive Anosov vector field $Y$ on $M$, 
\item[--] the images under $\Psi$ of the orbits $\Phi_{\mathscr{W}}(\mathscr{a}_{\ell})$, $\Phi_{\mathscr{W}}(\mathscr{a}_r)$, $\Phi_{\mathscr{W}}(\mathscr{b}_\ell)$, $\Phi_{\mathscr{W}}(\mathscr{b}_r)$ and $\Phi(\mathscr{c})$ of $\mathscr{Y}$ are precisely the curves $a_\ell\times\{0\}$, $b_\ell\times\{\frac{3\pi}{2}\}$, $a_r\times\{0\}$, $b_r\times\{\frac{3\pi}{2}\}$ and $c\times\{0\}$; of course, these curves are periodic orbits of the Anosov vector field $Y$,
\item[--] the local stable manifolds of the four first orbits above are horizontal, and that the twist number of local stable manifold of the orbit $c\times\{0\}$ with respect to the surface $S\times\{0\}$ is $+1$,
\end{itemize}
completing the proof of Theorem~\ref{p.perorb-lk}. So we are left to prove Claim~\ref{claim:diffeo-between-surfaces}. By standard surface topology argument, it is enough to prove that the ``configuration" of the curves $\mathscr{a}_{\ell,0}, \mathscr{a}_{r,0}, \mathscr{b}_{\ell,0}, \mathscr{b}_{r,0}, \mathscr{c}_{0}, \mathscr{d}_{0}$ in the surface $\mathscr{S}_0$ is the same as ``configuration" of the curves $a_\ell,b_\ell,a_r,b_r,c,d$ in the surface $S$. More precisely,  Claim~\ref{claim:diffeo-between-surfaces} is a consequence from the following facts:
\begin{enumerate}
\item $\mathscr{S}_0$ is a closed orientable genus two surface;
\item $\mathscr{d}_0$ is a simple closed curve cutting $\mathscr{S}_0$ into two once-punctured tori, one of them containing the curves $\mathscr{a}_{\ell,0}$ and  $\mathscr{b}_{\ell,0}$, the other containing the curves $\mathscr{a}_{r,0}$ and  $\mathscr{b}_{r,0}$;
\item $\mathscr{a}_{\ell,0}$ (resp. $\mathscr{a}_{r,0}$) intersects transversally $\mathscr{b}_{\ell,0}$ (resp. $\mathscr{b}_{r,0}$) at a single point and does not intersect $\mathscr{a}_{r,0}$, $\mathscr{b}_{r,0}$ and $\mathscr{c}_0$ (resp. $\mathscr{a}_{\ell,0}$, $\mathscr{b}_{\ell,0}$ and $\mathscr{c}_0$);
\item $\mathscr{b}_{\ell,0}$ (resp. $\mathscr{b}_{r,0}$) intersects transversally  $\mathscr{a}_{\ell,0}$ (resp. $\mathscr{a}_{r,0}$) and $\mathscr{c}_{0}$ at a single point and does not intersect $\mathscr{a}_{r,0}$ and $\mathscr{b}_{r,0}$ (resp. $\mathscr{a}_{\ell,0}$ and $\mathscr{b}_{\ell,0}$);
\item $\mathscr{c}_0$ intersects transversally $\mathscr{d}_0$ at two points.
\end{enumerate}
Items~1 and 2 follow from Proposition~\ref{c.Anosov-tor-fib}. Item~2 implies that $\mathscr{a}_{\ell,0}$ and $\mathscr{b}_{\ell,0}$ do not intersect $\mathscr{a}_{r,0}$ and $\mathscr{b}_{r,0}$. Recall that in the region $\mathscr{U}-\mathscr{U}_A$, the orbits of the vector field $\mathscr{W}$ coincide with the fibers of the unit tangent bundle, and the surface $\mathscr{S}_0$ coincides with the surface section of the unit tangent bundle $S_0$. The curves $\mathscr{a}_{\ell,0}$ and $\mathscr{b}_{\ell,0}$ are contained in this region. It follows that the curves $\mathscr{a}_{\ell,0}$ and $\mathscr{b}_{\ell,0}$ coincide with the projections of the lifted geodesics $\widetilde a_\ell$ and $\widetilde b_\ell$ on $S_0$ along the fibers of the unit tangent bundle, which intersect transversally at a single point. Hence,  $\mathscr{a}_{\ell,0}$ and $\mathscr{b}_{\ell,0}$ intersect transversally at a single point. The same argument applies for $\mathscr{a}_{r,0}$ and $\mathscr{b}_{r,0}$. Similarly, the part of $\mathscr{c}_{0}$ contained in $\mathscr{U}_A$ coincides with the lifted geodesic $\widetilde c\subset S_0$. It follows that $\mathscr{a}_{\ell,0}$ (resp. $\mathscr{a}_{r,0}$) does not intersect $\mathscr{c}_{0}$, and that $\mathscr{b}_{\ell,0}$ (resp. $\mathscr{a}_{r,0}$) intersects $\mathscr{c}_{0}$ transversally at a single point. 
The lifted geodesic $\tilde c$ is in $\overline{S}_0-\tilde d^\pm$, hence $\mathscr{c}$ is in $\mathscr{S}_0$, hence $\mathscr{c}_0=\mathscr{c}$. On the other hand, $\mathscr{d}_0$ is the image in $\mathcal{U}$ of $\partial \overline{S}_0$ (through the canonical identification $\mathscr{i}:U-\tilde d^\pm\to\mathscr{U}-\mathscr{d}^\pm$. We know that $\tilde c$ intersects transversally $\partial \overline{S}_0$ at exactly two points (namely the points $\tilde p$ and $\tilde q$). Hence, $\mathscr{c}_0$ and  $\mathscr{d}_0$ intersect transversally at two points. This completes the proof of items 3 and 4, and of the theorem.
\end{proof}

\begin{remark}
The difficulty in the preceding proof comes from the fact that there is no natural identification between the surface $\mathscr{S}_0$ and the abstract surface $S$. Indeed, the composition of the restriction  of natural identification $\mathscr{i}^{-1}:\mathscr{U}-\mathscr{d}^\pm$ and the unit tangent bundle projection $\pi:U\to S$ provides a  natural diffeomorphism between $\mathscr{S}_0-\mathscr{d}_0\to S-d$. But this diffeomorphism does not extend to a diffeomorphism from $\mathscr{S}_0$ to $S$, since its restriction to the right halves of the surfaces is orientation-preserving whereas its restriction to the left halves of the surfaces is orientation-reversing. This is due to the fact that the fibers of the unit tangent bundle $U$ intersect the section $S_0$ in opposite directions (for the orientation of $S_0$) in the left part and the right part. A consequence is that the diffeomorphism of Claim~\ref{claim:diffeo-between-surfaces} cannot preserve the orientations of all the curves (see Remark~\ref{r.incoherent-orientations}). 
\end{remark}

\section{From genus two to arbitrary genus} 
\label{s.cover}
In Section~\ref{s.g-2-fiber}, we have constructed a particular fibered manifold, with genus two fibers, carrying a transitive Anosov flow. In the present section, we will use a covering trick to deduce an analogous result with fibers of arbitrarily large genus. 

For $g\geq 3$, let $S_g$ be the genus $g$ closed, oriented surface equipped with the system of unoriented simple closed curves represented  on Figure~\ref{arbi.abcd} below. Denote by $\tau_{a_i}$, $\tau_{b_j}$, $\tau_{c_k}$ and $\tau_{d_l}$ the left Dehn twists on $S_g$ along the curves $a_i$, $b_j$, $c_k$ and $d_l$ for $1\leq i,j\leq g$ and $1\leq k,l\leq g-1$. Let $\widehat\tau_d:=\tau_{d_1}^{-2}\tau_{d_2}^{-2}\dots,\tau_{d_{g-1}}^{-2}$. Note that the product is commutative since the support of the Dehn twists in this product are pairwise disjoint.  

\begin{figure}[htb]
    \centering
    \includegraphics[scale=0.25]{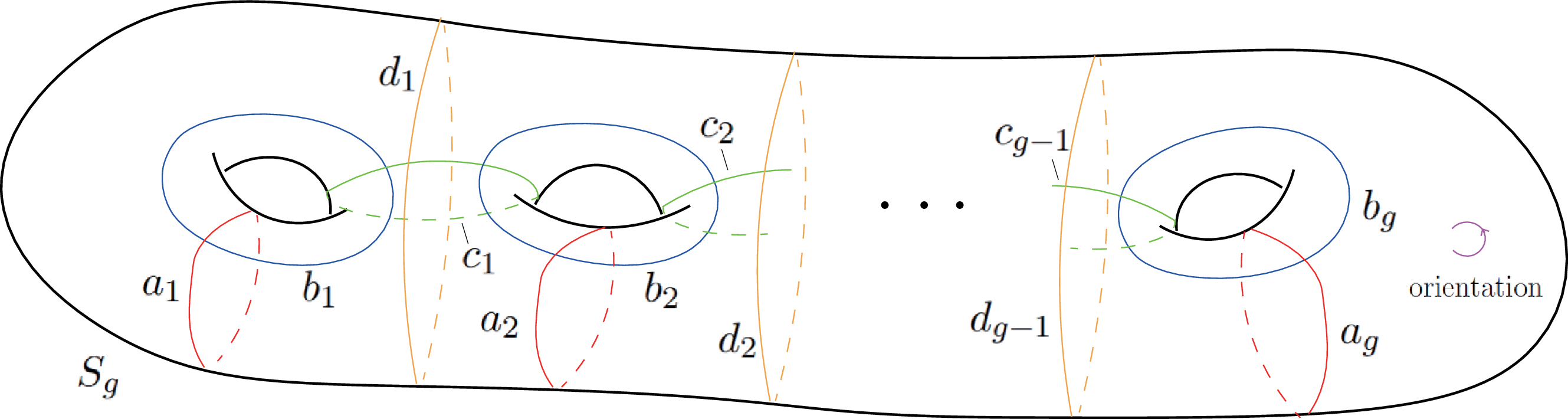}
    \caption{The simple closed curves  $ a_i,b_j, c_k, d_l$ on $ S_g $ for $g$.}
    \label{arbi.abcd}
\end{figure}

\begin{theorem}\label{cover.perorb-lk}
The mapping torus 
$$S_g \times [0,2\pi]_{/ (x,2\pi) \sim (\widehat\tau_d(x),0)}$$
carries a transitive Anosov flow $\widetilde Y$ with orientable stable and unstable foliations such that, denoting by $p: M \to \SS^1=\mathbb{R}/2\pi\mathbb{Z}$ the fibration induced by the projection of $S_g \times [0,2\pi]$ on the second coordinate, for $1 \leq i,j \leq g, 1 \leq  k \leq g-1$,
\begin{enumerate}
    \item the horizontal curves $a_i \times \{0\}$, $b_j \times \{\frac{3\pi}{2}\}$ and $c_k \times \{0\}$ on $M$ are periodic orbits of $\widetilde{Y}^t$,
    \item the local stable manifolds of the orbits $a_i \times \{0\}$ and $b_j \times \{\frac{3\pi}{2}\}$ are horizontal with respect to the fibration $p$,
    \item the twist number of the local stable manifold of the orbit $c_k \times \{0\}$ with respect to the fibration $p$ is equal to $1$.
\end{enumerate}
\end{theorem}

\subsection{An explicit covering}

For $g\geq 3$, let $S_g$ be the genus $g$ closed oriented surface with labeled unoriented curves $a_i$, $b_j$, $c_k$ and $d_l$, with $1\leq i,j\leq g,\;1\leq k,l\leq g-1$, represented on Figure \ref{arbi.abcd}. Let $S_2$ be the genus $2$ closed oriented surface with labeled unoriented curve $a_\ell, a_r, b_\ell, b_r, c, d$ represented on Figure~\ref{C.S}. 

\begin{figure}[htb]
    \centering
    \includegraphics[scale=0.3]{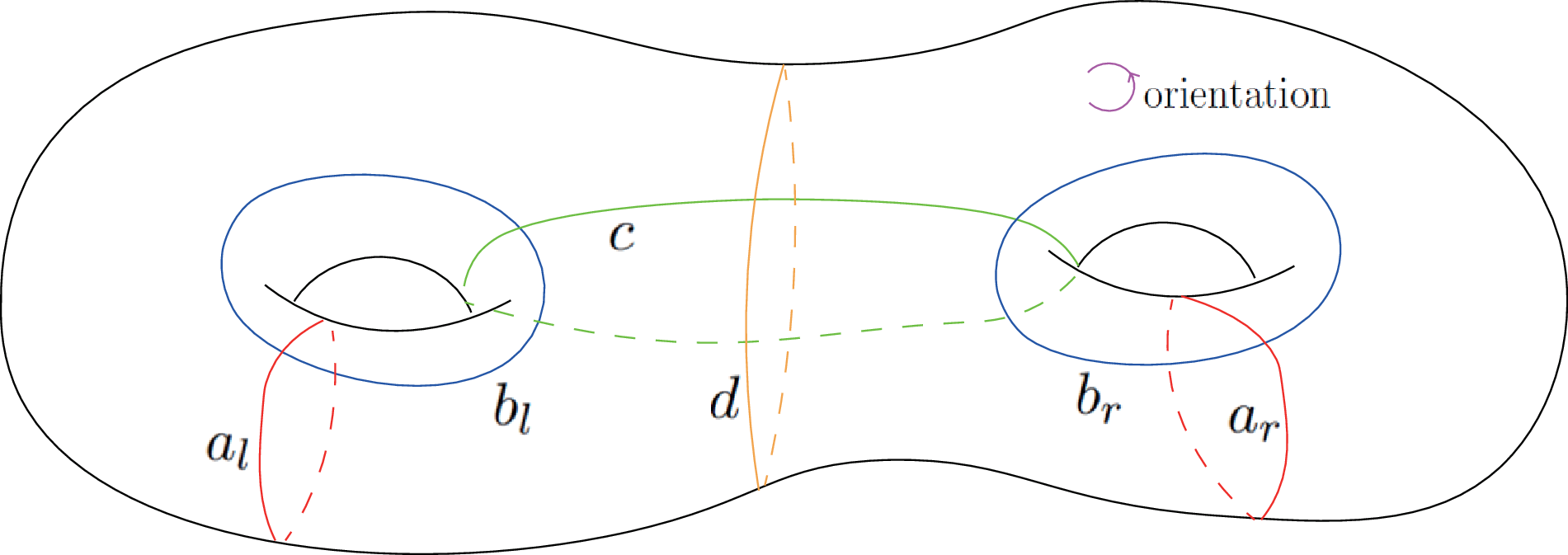}
    \caption{The genus-2 surface \(S_2\) with the labeled curves}
    \label{C.S}
\end{figure}

\begin{proposition}\label{thm:finitecover}
    For every $g\geq 3$, there exists an orientation preserving covering map $\Psi: S_g \to S_2 $ that maps: 
    \begin{itemize}
    \item the curve $a_i$ to the curve $a_\ell$ when $i$ is odd and to the curve $a_r$ when $i$ is even,
    \item the curve $b_j$ to the curve $b_\ell$ when $j$ is odd and to the curve $b_r$ when $j$ is even, 
    \item the curve $c_k$ to the curve $c$ for every $k$, 
    \item the curve $d_l$ to the curve $d$ for every $l$.
    \end{itemize}
\end{proposition}

Recall that $S_{\ell}$ and $S_d$ are respectively the left and the right connected components of $S_2 - d$. Let $S_{1,2}$ be the genus-$1$ surface with two boundary components with the labeled unoriented arc system represented on Figure \ref{cover.S12} on the left. Denote by \(c_\ell=c\cap S_\ell\) and \(c_r=c\cap S_r\) the left and right halves of the curve $c\subset S_2$.
\begin{figure}[htb]
    \centering
    \includegraphics[scale=0.42]{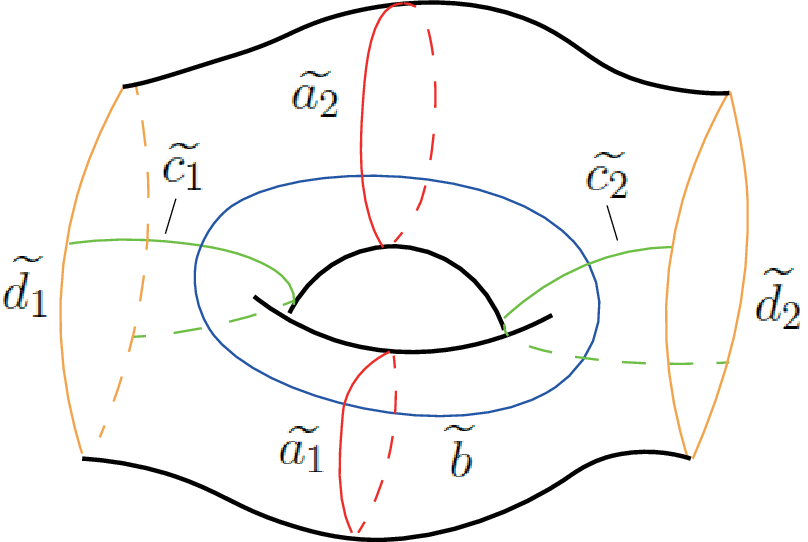}
    \caption{$S_{1,2}$ with the labeled  arc system $\widetilde{a_1},\widetilde{a_2},\widetilde{b},\widetilde{c_1},\widetilde{c_2},\widetilde{d_1},\widetilde{d_2}$}
    \label{cover.S12}
\end{figure}

\begin{lemma}\label{lemma:doublecover}
    There is a two-sheeted orientation preserving covering map $\iota_\ell: S_{1,2} \to S_\ell$  that sends  $\widetilde{a_i}$ to $a_\ell$, $\widetilde{b}$ to $b_\ell$, $\widetilde{d_i}$ to $d$ and $ \widetilde{c_i}$ to $c_\ell$. Similarly, there is a two-sheeted orientation preserving covering map $\iota_r: S_{1,2} \to S_r$  that sends  $\widetilde{a_i}$ to $a_r$, $\widetilde{b}$ to $b_r$, $\widetilde{d_i}$ to $d$ and $ \widetilde{c_i}$ to $c_r$.
\end{lemma}

\begin{proof}
    The covering maps $\iota_\ell$ and  $\iota_r$ are given by the quotient map of the rotation of angle $\pi$ around the axis represented on Figure \ref{cover.S12coverT}.
\end{proof}

\begin{figure}[htb]
    \centering
    \includegraphics[scale=0.39]{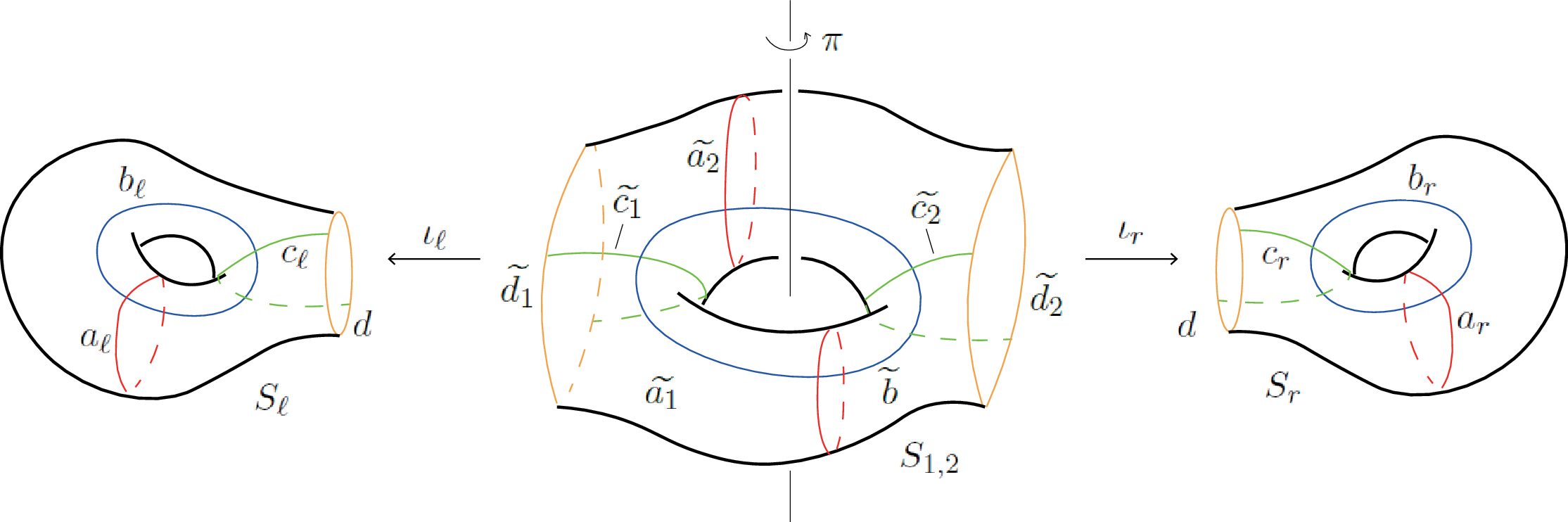}
    \caption{The covering maps $\iota_\ell:S_{1,2} \to S_{\ell}$ and $\iota_r:S_{1,2} \to S_r$}
    \label{cover.S12coverT}
\end{figure}

\begin{proof}[Proof of Proposition \ref{thm:finitecover}]
We assume $g$ is even. The argument is similar when $g$ is odd. The covering map $\Psi$ is given as follows. Cutting $S_g$ along the curves $d_1,\dots,d_{g-1}$, we will get $g$ surfaces with boundary. Label these surfaces by $A_1,A_2,\dots,A_g$, where $A_1$ is bounded by $d_1$, $A_s$ is bounded by $d_{s-1}$ and $d_s$ for $2\leq s\leq g-1$, and $A_g$ is bounded by $d_{g-1}$. The restriction of $\Psi$ to $A_1$ will be the trivial one-sheeted covering map mapping $A_1$ to $S_\ell$. For $s$ even in $\{2,\dots,g-1\}$, the restriction of $\Psi$ to $A_s$ will be the two-sheeted covering map $\iota_r$ mapping $A_s$ to $S_r$.  For $s$ odd in $\{2,\dots,g-1\}$, the restriction of $\Psi$ to $A_s$ will be the two-sheeted covering map $\iota_\ell$ mapping $A_s$ to $S_\ell$. Finally, the restriction of $\Psi$ to $A_g$ will be the trivial one-sheeted covering map mapping $A_g$ to $S_r$. See Figure \ref{f.cover-S4S2} for an illustration in the case $g =4$. One can check that verify that the map $\Psi$ defined in this manner is indeed a covering map. By construction, the curves $a_i, b_j, c_k, d_l$ as required by the statement of Proposition~\ref{thm:finitecover}.
\end{proof}

\subsection{Proof of Theorem \ref{cover.perorb-lk}}

\begin{figure}[htb]
    \centering
    \includegraphics[scale=0.39]{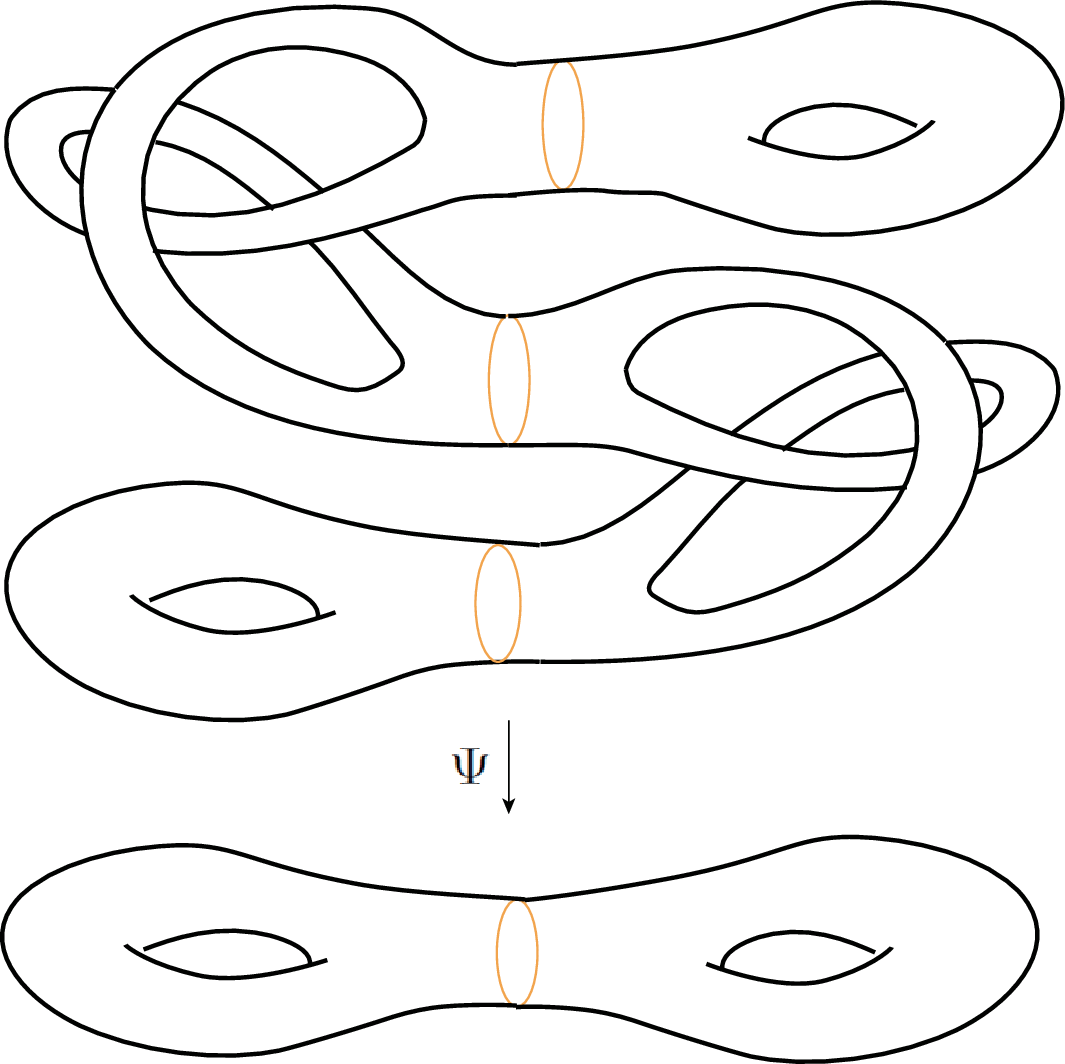}
    \caption{The covering map $\Psi:S_4 \to S_2$}
    \label{f.cover-S4S2}
\end{figure}

\begin{proof}[Proof of Theorem \ref{cover.perorb-lk}]
Let $\Psi: S_g \to S_2$ be the covering map provided by Proposition \ref{thm:finitecover}. The product of left Dehn twists $\widehat\tau_d := \tau_{d_1}^{-2} \tau_{d_2}^{-2} \cdots \tau_{d_{g-1}}^{-2} : S_g \to S_g$ is a lift under $\Psi$ of the Dehn twist $\tau^{-2}_d: S_2 \to S_2$, that is $\Psi \circ \tau_d^{-2} = \widehat\tau_d \circ \Psi$. 
It follows that $\Psi$ induces a covering map 
   $$\widetilde{\Psi}:\begin{array}[t]{rcl}
    M_{S_g,\widehat\tau_d}  & \longrightarrow & M_{S_2,\tau^{-2}_d} \\
    (x,t) & \longmapsto &  (\Psi(x),t)
    \end{array}$$  
where 
$$M_{S_g,\widehat\tau_d}:=S_g\times [0,2\pi]_{/(x,2\pi)\sim (\widehat\tau_d(x),0)}\quad\mathrm{and}\quad
M_{S_2,\tau^{-2}_d}:=S_2\times [0,2\pi]_{/(x,2\pi)\sim (\tau_d^{-2}(x),0)}.$$ 
Let $Y^t$ be the transitive Anosov flow on $M_{S_2,\tau^{-2}_d}$ satisfying the properties listed in Theorem \ref{p.perorb-lk}. Lifting $Y^t$ by $\widetilde\Psi$, we obtain an Anosov flow $\widetilde Y^t$ on $M_{S_g,\widehat\tau_d}$. We are left to check that items $(1),(2),(3)$ in Theorem \ref{cover.perorb-lk} hold. The oriented curves $a_i \times \{0\}, b_j \times \{\frac{3\pi}{2}\}, c_k \times\{0\}$ are lifts under $\Psi$ of the oriented curves $a_\ell \times \{0\}$, $a_r \times \{0\}$, $b_{\ell} \times \{\frac{3\pi}{2}\}$, $b_r \times \{0\}$ and $c_k \times\{0\}$ which are periodic orbits of the flow $Y^t$. It follows that $a_i \times \{0\}, b_j \times \{\frac{3\pi}{2}\}, c_k \times\{0\}$ are periodic orbits of the flow $\widetilde Y^t$. Moreover, by the construction of $\Psi$, each of the orbits $a_i \times \{0\}$ and $c_k \times\{0\}$ admits on neighbourhood on which $\widetilde{\Psi}$ is a one to one (\emph{i.e.} is a homeomorphism on its image). Since the twist number is a local property, it implies that item $(3)$ and the part of item $(2)$ concerning $a_i \times \{0\}$ hold. 
By construction, $\widetilde{\Psi}$ is a $2$-sheeted or $1$-sheeted covering map on a neighborhood of $b_j \times \{\frac{3\pi}{2}\}$ (cf. the proof of Theorem \ref{thm:finitecover}). If $\widetilde{\Psi}$ is $1$-sheeted, then $b_j \times \{\frac{3\pi}{2}\}$ is horizontal by the same argument as above for  $a_i \times \{0\}$ and $c_k \times\{0\}$. If $\widetilde{\Psi}$ is $2$-sheeted on a neighborhood of $b_j \times \{\frac{3\pi}{2}\}$, the twist number of the local stable manifold with respect to the horizontal fibration on $M_{S_g,\widehat\tau_d}$ is twice of the twist number of the local stable manifold of $b_{\ell/r}\times \{\frac{3\pi}{2}\}$ with respect to the horizontal fibration on $M_{S_2,\tau_d^{-2}}$. The later being equal zero, the former is also equal to zero. This completes the proof.
\end{proof}

\section{Proof of Theorem~\ref{t.generators}}
\label{s.proof}
\begin{proof}[Proof of Theorem~\ref{t.generators}]
Let $S_g$ be the closed oriented surface of genus $g$, and $a_1,\dots,a_g,b_1,\dots,b_g,c_1,\dots,c_{g-1},d_1,\dots,d_{g-1}$ be the simple closed curves represented on Figure~\ref{arbi.abcd}. Let \( \tau_{a_i} \), \( \tau_{b_j} \), \( \tau_{c_k} \), \( \tau_{d_l} \) be the left Dehn twists along the curves \( a_i, b_j, c_k, d_l \), and let \( \widehat\tau_d = \tau_{d_1}^{-2} \tau_{d_2}^{-2} \cdots \tau_{d_{g-1}}^{-2} \). Recall that we consider the family of product of Dehn twists 
$$\mathcal{T}:=\left\{\widehat\tau_{d}^{\small{-2}}
\tau_{b_1}^{q_1}\dots\tau_{b_{g}}^{q_g}
\tau_{c_1}^{r_1}\dots\tau_{c_{g-1}}^{r_{g-1}}
\tau_{a_1}^{p_1}\dots\tau_{a_{g}}^{p_g}\right\}_{\substack{r_1,\dots,r_{g-1}\,=\,0\; \mathrm{or}\;-2 \\
q_1,\dots,q_g\,\in\,\ZZ \hfill \\  p_1,\dots,p_q\,\in\,\ZZ\hfill }}$$
and a non-trivial product $\varphi\in\mathrm{Homeo}^+(S_g)$ of elements of $\mathcal{T}$. For notational simplicity, we will prove Theorem~\ref{t.generators} for 
$$\varphi= \sigma_1 \circ \sigma_2 \circ \dots \circ \sigma_l$$
where each $\sigma_i$ belongs to the subfamily 
$$\mathcal{T}_0:=\{\widehat\tau_d\}\cup \{\widehat\tau_d\tau_{a_i}^{\pm 1}\}_{1 \leq i \leq g}\cup \{\widehat\tau_d\tau_{b_j}^{\pm 1}\}_{1 \leq j \leq g}\cup \{\widehat\tau_d\tau_{c_k}^{-2 }\}_{1 \leq k \leq g-1},$$ 
of $\mathcal{T}$, and explain how to get the general case at the end of the proof.

By Theorem~\ref{cover.perorb-lk}, there exists a transitive Anosov flow \(Y^t\) on the mapping torus
\[
M=S_g\times [0,2\pi]_{/(x,2\pi)\sim(\widehat\tau_d(x),0)}.
\]
The mapping torus 
\[
\widehat M=\left(\bigcup_{s=0}^{l-1} S_g\times [2s\pi,2(s+1)\pi]\right)_{\big{/}\begin{array}[t]{ll}(x,2s\pi^-)\sim(\widehat\tau_d(x),2s\pi^+),\; 1\leq s\leq l-1\\
(x,2l\pi)\sim (\widehat\tau_d(x),0).
\end{array}}
\]
is naturally a \(l\)-sheeted cover of $M$. Lifting the transitive Anosov flow $Y^t$, we get a transitive Anosov flow $\widehat Y^t$ on $\widehat M$. Let $\mathrm{pr}$ be the projection of $S_g \times [0,2\pi]$ to $M$ and $\widehat{\mathrm{pr}}$ be the projection of $S_g \times [0,2l\pi]$ to $\widehat M$. Since taking a finite cover in the $\SS^1-$direction of a fibration $M \to\SS^1$ does not change the twist number for a periodic orbit $\gamma$ contained in a fiber with respect to the fibration, Theorem~\ref{cover.perorb-lk} implies that: 
\begin{itemize}
    \item for every \(i,j\in \{1,\dots,g\}\), the periodic orbits 
    \(\mathrm{pr}(a_i\times \{0\})\) and  \(\mathrm{pr}(b_j\times \{\frac{3\pi}{2}\})\) of the vector field $Y$ lift to periodic orbits 
    \[
    \widehat{\mathrm{pr}}(a_i\times \{s\pi\}) \mbox{ and }\ \widehat{\mathrm{pr}}(b_j\times \{\tfrac{3\pi}{2}+s\pi\}),~ s\in\{1,\dots,k\},
    \]
    of the vector field $\widehat{Y}$ and the local stable manifolds of these lifted orbits are horizontal with respect to the fibration $\widehat M\to \SS^1$.
    \item for every \(k \in \{1,\dots,g-1\}\), the periodic orbit \(\mathrm{pr}(c_k\times \{0\})\) of the vector field $Y$ lift to periodic orbits 
    \[
    \widehat{\mathrm{pr}}(c_l\times \{s\pi\}), ~s\in\{1,\dots,k\},
    \]
    of the vector field $\widehat{Y}$ and the twist number of the local stable manifolds of these lifted orbits with respect to the fibration $\widehat M \to S^1$ are all equal to \(+1\).
\end{itemize}
For each $s\in\{1,\dots,l\}$, we perform a Dehn–Fried surgery as follows.
\begin{itemize}
\item  If $\sigma_s = \widehat\tau_d\tau_{a_i}^{\omega}$ for some $i\in\{1,\dots,g\}$ and $\omega=\pm 1$ (resp. $\widehat\tau_d\tau_{b_j}^{\omega}$ for some $j\in\{1,\dots,g\}$ and $\omega=\pm 1$), then we perform an index $\omega$  Dehn–Fried surgery on the periodic orbit $\widehat{\mathrm{pr}}(a_i\times \{s\pi\})$ (resp. the orbit $\widehat{\mathrm{pr}}(b_j\times \{\frac{3\pi}{2}+s\pi\})$). 
\item If $\sigma_s = \widehat\tau_d\tau_{c_k}^{-2}$ for some $k$, we perform an index \(+2\) Dehn–Fried surgery on the periodic orbit $\widehat{\mathrm{pr}}(c_k\times \{s\pi\})$. 
\end{itemize}
We observe that all the Dehn-Fried surgeries above satisfy the hypotheses of Proposition~\ref{proposition:preserves-fibration-I}~and Proposition~\ref{proposition:preserves-fibration-II}: 
\begin{itemize}
\item since the local stable manifolds of the orbits $\widehat{\mathrm{pr}}(a_i\times \{s\pi\})$ and $\widehat{\mathrm{pr}}(b_j\times \{\frac{3\pi}{2}+s\pi\})$ are horizontal, any Dehn-Fried surgery on these orbits falls in Proposition~\ref{proposition:preserves-fibration-I}~;
\item since the twist number of the local stable manifold of the orbit $\widehat{\mathrm{pr}}(c_k\times \{s\pi\})$ with respect to the fibration is equal to $+1$, the index $+2$ Dehn-Fried surgery on this orbit falls in Case 1 of Proposition~\ref{proposition:preserves-fibration-II}.
\end{itemize}
Let \((\widehat M',\widehat Y'^t)\) be the manifold and the topological Anosov flow obtained by performing the above Dehn–Fried surgeries for every $s$ on  \((\widehat M,\widehat Y^t)\). According the generalization of Proposition~\ref{p.multi-Dehn-Fried} stated in Remark~\ref{r.pieces-of-monodromy}, $\widehat M'$ is homeomorphic to the mapping torus of $\varphi$. And by the work of Shannon (\cite{shannon2020dehn}), $\widehat Y'^t$ is orbit equivalent to a genuine transitive Anosov flow. Hence, the  mapping torus of $\varphi$ carries a transitive Anosov flow. This completes the proof of Theorem~\ref{t.generators} in the case where $\varphi$ is a product of elements of $\mathcal{T}_0$.

In the general case, where $\varphi$ is a product of elements $\sigma_1,\dots,\sigma_l$ of $\mathcal{T}$ (rather than $\mathcal{T}_0$), one just needs to perform more Dehn-Fried surgeries. Namely, for each $s$, if $$\sigma_s=\widehat\tau_{d}^{\small{-2}}
\tau_{b_1}^{q_{s,1}}\dots\tau_{b_{g}}^{q_{s,g}}
\tau_{c_1}^{r_{s,1}}\dots\tau_{c_{g-1}}^{r_{s,g-1}}
\tau_{a_1}^{p_{s,1}}\dots\tau_{a_{g}}^{p_{s,g}},$$ 
then one must perform an index $p_{s,i}$ Dehn-Fried  surgery
on the orbit $a_i\times\{s\}$ if $p_{s,i}\neq 0$, an index $-r_{s,i}=2$ Dehn-Fried surgery on the orbit $c_i\times\{s\}$ if  $r_{s,i}\neq 0$, and an index $q_{s,i}$ Dehn-Fried surgery on the orbit $b_i\times\{\frac{3\pi}{2}\}$ if  $q_{s,i}\neq 0$. All these orbits are pairwise disjoint, allowing to apply Proposition~\ref{p.multi-Dehn-Fried} and Remark~\ref{r.pieces-of-monodromy}. The rest of the proof follows exactly the same arguments as the previous case.
\end{proof}

\section{The Dehn twists generate a finite index subgroup of \(Sp(2g,Z)\)}
\label{s.F-index-subgp}
Let $ g \geq 2 $ and $S_g$ be the closed connected orientable surface of genus $g$. Consider simple closed curves  $ a_1, \dots, a_g, b_1, \dots, b_g, c_1, \dots, c_{g-1}$ on $S_g$ as Figure \ref{f.aibici} shows. We equip $S_g$ with the orientation chosen so that $\mathrm{Int}(a_i,b_i)=-1$. We denote by $ \tau_{a_i}, \tau_{b_i}, \tau_{c_i}$ the left Dehn twists along $a_i, b_i$ and $c_i$.  

For any element $f$ of $\M(S_g)$, we denote by $ \bar f \in \Sp(\rm{H}_1(S_g, \ZZ)) $ the action of $f$ on the first homology group of $S_g$ with coefficients in $\ZZ$ (the symplectic structure on $ \rm{H}_1(S_g, \ZZ)$ being given by the intersection form). We denote by $\Gamma$ the subsemigroup of $\Sp(\rm{H}_1(S_g, \ZZ))$ which is generated by the Dehn twists
$$\bar{\tau}_{a_1}^{\pm1}, \dots, \bar{\tau}_{a_g}^{\pm1},\, \bar{\tau}_{b_1}^{\pm1}, \dots, \bar{\tau}_{b_g}^{\pm1},\,  \bar{\tau}_{c_1}^{-2}, \dots, \bar{\tau}_{c_{g-1}}^{-2}.$$ 
(note that the power $-2$ on the $\bar{\tau}_{c_i}$'s). A priori, $\Gamma$ is a sub-semigroup, but needs not be a subgroup, since the set of generators is not symmetric ($\bar{\tau}_{c_1}^{-2}$ is a generator, but not $\bar{\tau}_{c_1}^{2}$). Our goal is to prove:  

\begin{figure}[htb]
    \centering
    \includegraphics[scale=0.3]{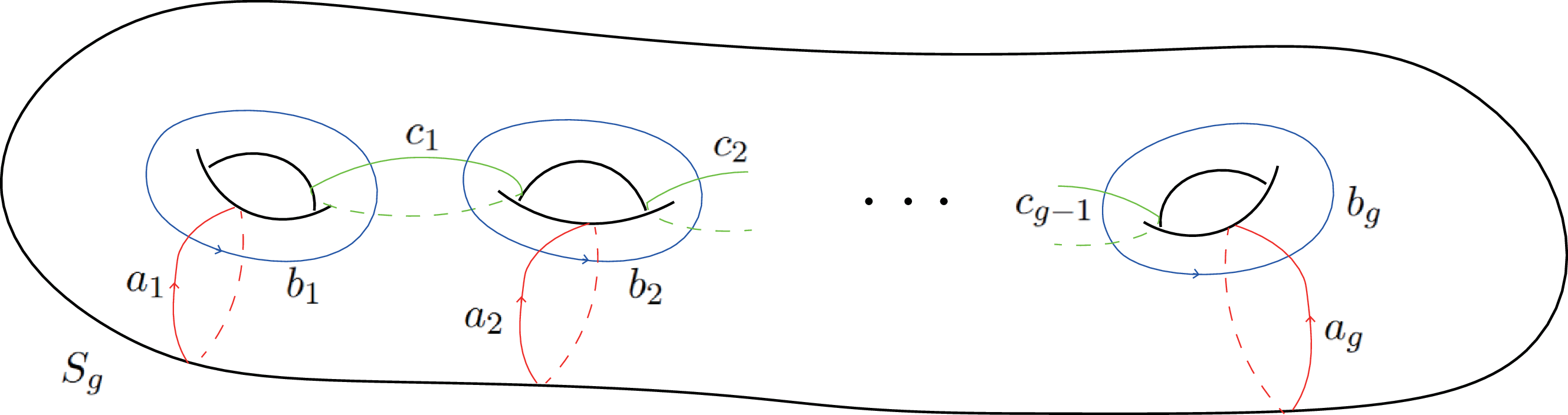}
    \caption{The simple closed curves  $ a_1,   \dots, a_g$, $ b_1,  \dots, b_g$, $ c_1, \dots, c_{g-1} $ on $ S_g $}
    \label{f.aibici}
\end{figure}

\begin{proposition}\label{proposition:subgroup-finiteindex}
 $\Gamma$ is a subgroup of finite index in $ \Sp(\rm{H}_1(S_g, \ZZ))$.  
\end{proposition}

We will proceed in two steps : first prove that $\Gamma$ is a subgroup (not only a sub-semigroup), then prove that this subgroup has finite index in $ \Sp(\rm{H}_1(S_g, \ZZ))$.

\subsection{$\Gamma$ is a subgroup}

\begin{proposition}
\label{proposition:subgroup}
    The semigroup $\Gamma$ is actually a subgroup of $\Sp(\rm{H}_1(S_g, \ZZ))$.
\end{proposition}

\begin{remark}
    Observe that $\Gamma$ is not a normal subgroup of $\Sp(\rm{H}_1(S_g, \ZZ))$. Indeed, by definition, $\Gamma$ contains the Dehn twist along the simple closed curve $a_1$. It is well-known that, for every non-separating simple close curve $e$, there exists a homeomorphism $f$ of the surface $S_g$ mapping $a_1$ on $e$, which implies that $\bar\tau_e=\bar f\bar\tau_{a_1}\bar f^{-1}$. Hence, if $\Gamma$ were a normal subgroup, then it would contain the Dehn twist along any simple closed curve of $S_g$. But it is well-known that the mapping class group of $S_g$ is generated by the Dehn along the simple closed curves, so $\Gamma$ would be equal to the whole group $\Sp(\rm{H}_1(S_g, \ZZ))$, which is obviously not true. Indeed, since ${\tau}^{-2}_{c_i}$ acts identically on $\rm{H}_1(S_g, \mathbb{Z}/2\mathbb{Z})$, any element of $\Gamma \mod 2$ has the form $\rm{diag}\{M_1,M_2,\cdots,M_g \}$ with respect to the basis $\{[a_1],[b_1],[a_2],\cdots, [a_g],[b_g]\}$, where $M_i$ is a $2\times2$ symplectic matrix. But we know that not all elements in $\Sp(\rm{H}_1(S_g, \ZZ/2\ZZ))$ are of this form.
\end{remark}

\begin{proof}[Proof of Proposition~\ref{proposition:subgroup}.]
Recall that $\Gamma$ is, by definition, the subsemigroup of $\Sp(\rm{H}_1(S_g, \ZZ))$ generated by $\bar{\tau}_{a_1}^{\pm1}, \dots, \bar{\tau}_{a_g}^{\pm1},\, \bar{\tau}_{b_1}^{\pm1}, \dots, \bar{\tau}_{b_g}^{\pm1},\,  \bar{\tau}_{c_1}^{-2}, \dots, \bar{\tau}_{c_{g-1}}^{-2}.$ So, in order to prove that $\Gamma$ is a subgroup, it is enough to prove that $\bar\tau_{c_i}^{2}$ can be expressed as a product of the above generators for every $i\in \{1,\dots,g-1\}$. 

We start with the case of $\bar{\tau}_{c_1}^{2}$. We consider the chain of curves $c_1, b_1, a_1$ (see Figure \ref{f.boundaryofa1b1c1}). The boundary of a small regular neighborhood of $c_1 \cup b_1 \cup a_1$ is made of two curves $d$ and $d'$, where the two boundary curves are oriented as boundaries of the neighbourhood (see Figure \ref{f.boundaryofa1b1c1}).

\begin{figure}[htb]
    \centering
    \includegraphics[scale=0.4]{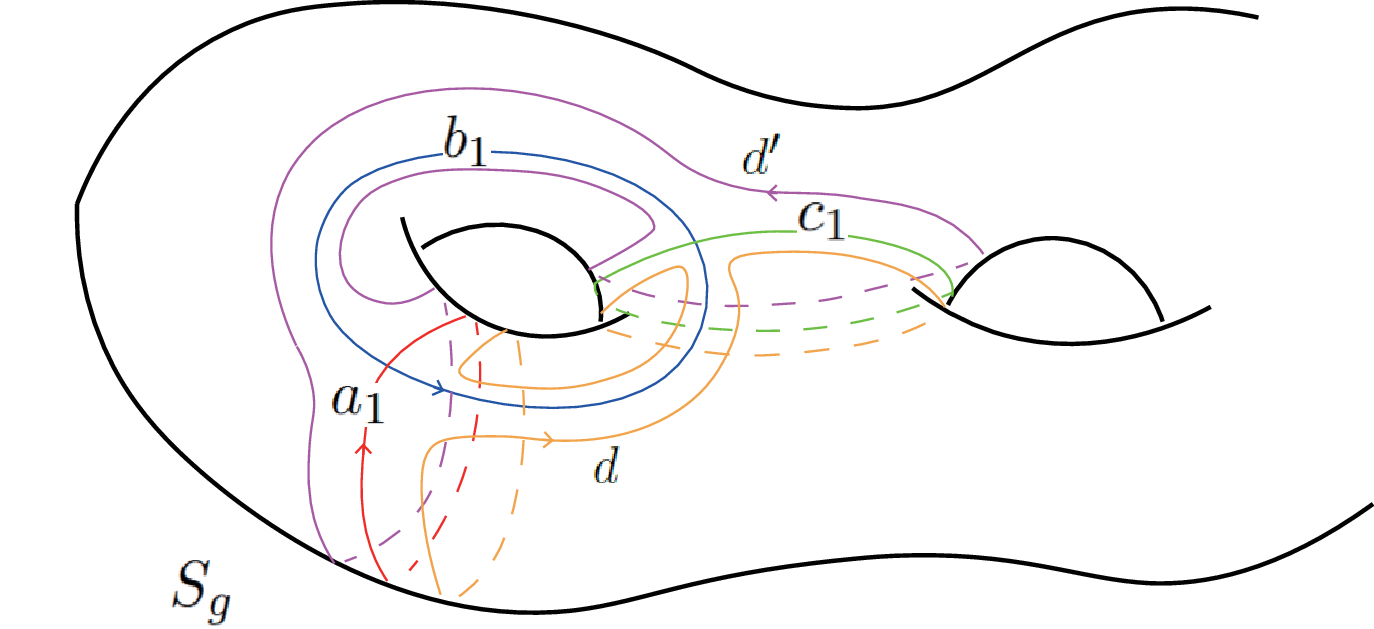}
    \caption{The boundary curves of $ a_1 \cup b_1 \cup c_1$ on $ S_g$}
    \label{f.boundaryofa1b1c1}
\end{figure}

The chain relation in $\M(S_g)$ (see \cite[Section 4.4.1]{FarbMargalit12}) states that
$$
\left( \tau^{2}_{c_1} \tau_{b_1} \tau_{a_1} \right)^3 = \tau_d \tau_{d'}.
$$ 
We observe that $d$ is homotopic to an initial curve $a_2$, and $d'$ is homotopic to the curve $a_2'$ as drawn in Figure \ref{f.isotopyboundaries}. Hence, in $\M(S_g)$, 
$$
\left( \tau^2_{c_1} \tau_{b_1} \tau_{a_1} \right)^3 = \tau_{a_2} \tau_{a_2'}.
$$

\begin{figure}[htb]
    \centering
    \includegraphics[scale=0.35]{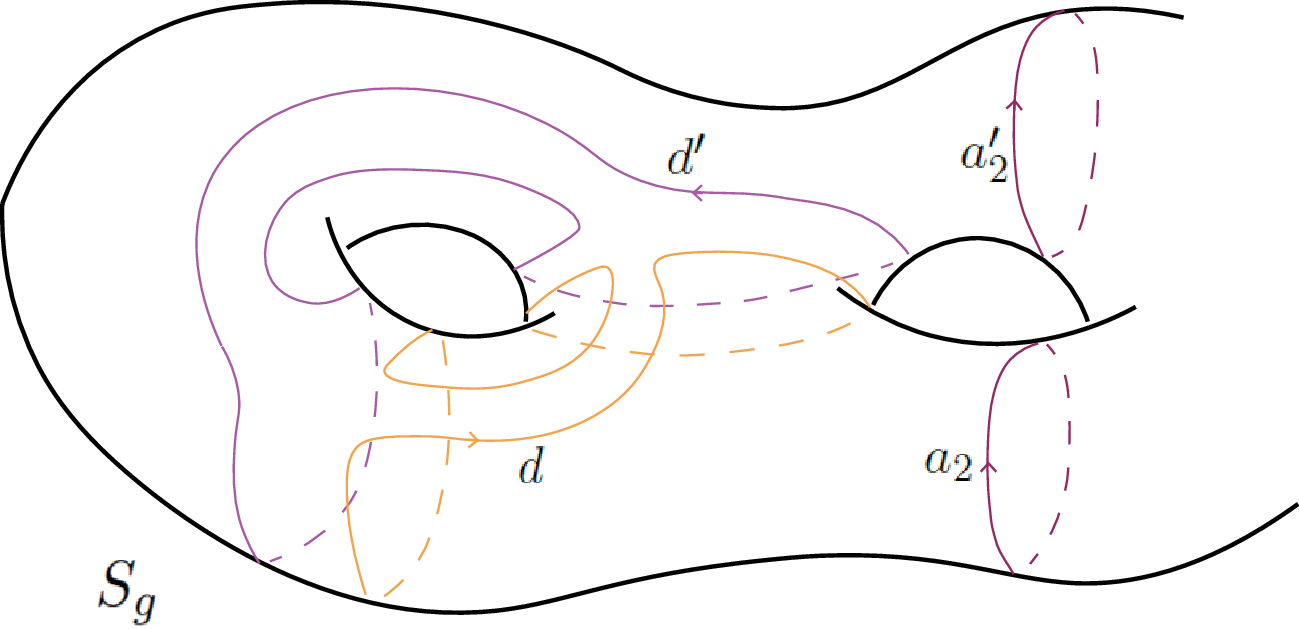}
    \caption{The curves $d,d'$ are homotopic to the curves $a_2,a'_2$ respectively}
    \label{f.isotopyboundaries}
\end{figure}

Then we observe that the curves $a_2$ and $a_2'$ bound a subsurface of $S_g$, and therefore $a_2'$ is homologous to $a_2^{-1}$. The left Dehn twist along a curve does not depend on the orientation of this curve, and the Dehn twists along two homologous curves have the same image in $\Sp(\rm{H}_1(S_g, \ZZ))$ (see e.g., \cite[Proposition 6.3]{FarbMargalit12}). Hence, in $\Sp(\rm{H}_1(S_g, \ZZ))$,
$$
\bar{\tau}_{a_2} = \bar{\tau}_{a_2'}.
$$
Hence
$$
\left( \bar{\tau}_{c_i}^2 \bar{\tau}_{b_i} \bar{\tau}_{a_i} \right)^3 = \tau_{a_2}^2
$$
or equivalently
$$
\bar{\tau}_{a_2}^2\left(\bar{\tau}_{a_1}^{-1}\bar{\tau}_{b_1}^{-1}\bar{\tau}_{c_1}^{-2}\right)^2\bar{\tau}_{a_1}^{-1}\bar{\tau}_{b_1}^{-1}
 = \bar{\tau}_{c_1}^{2}
$$
In particular, $\bar{\tau}_{c_1}^{2} \in \Gamma$ as desired.

The case of $\bar{\tau}_{c_{g-1}}^{2}$ is entirely similar (replace $a_1, b_1, c_1$ by $a_g, b_g, c_{g-1}$), hence it is in $\Gamma$.

So we are left to consider the case of $\bar{\tau}_{c_i}^{2}$ for $2 \leq i \leq g-2$. It will follow from similar arguments as above, but we need to slightly change the curves we consider. To this end, consider the chain of curves $a_i, b_i, c_i$ (see Figure \ref{f.boundaryofaibici}), and the two curves $d, d'$ that bounding a regular neighborhood of $c_i \cup b_i \cup a_i$ (see Figure \ref{f.boundaryofaibici}).

\begin{figure}[htb]
    \centering
    \includegraphics[scale=0.35]{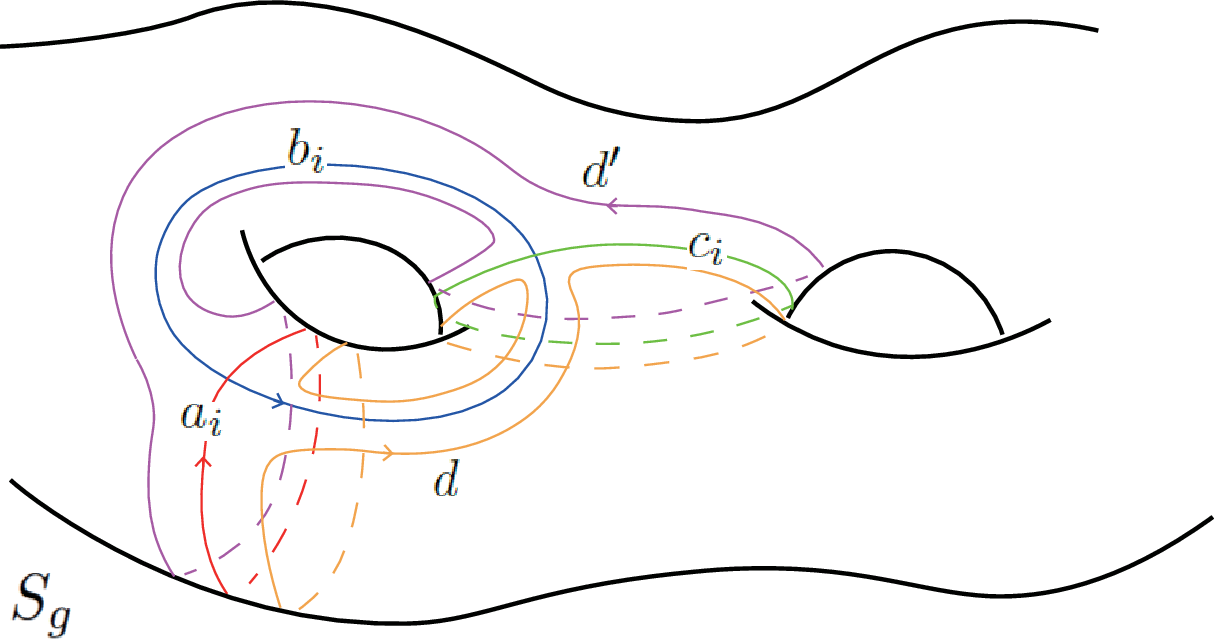}
    \caption{The boundary curves $ d,d'$ on $ S_g $}
    \label{f.boundaryofaibici}
\end{figure}

By the chain relation already used in the case of $\bar\tau_{c_1}^{2}$ (see \cite[Section 4.4.1]{FarbMargalit12}), we have in $\M(S_g)$,
$$\left( \tau_{c_i}^2 \tau_{b_i} \tau_{a_i} \right)^3 = \tau_d \tau_{d'}.$$
Then, we notice that the curves $d$ and $d'$ are homotopic to the curves $a_{i+1}$ and $a_{i+1}'$ drawn in Figure \ref{f.isotopydi}. Hence, in $\M(S_g)$,
$$\left( \tau_{c_i}^2 \tau_{b_i} \tau_{a_i} \right)^3 = \tau_{a_{i+1}} \tau_{a_{i+1}'}$$

\begin{figure}[htb]
    \centering
    \includegraphics[scale=0.34]{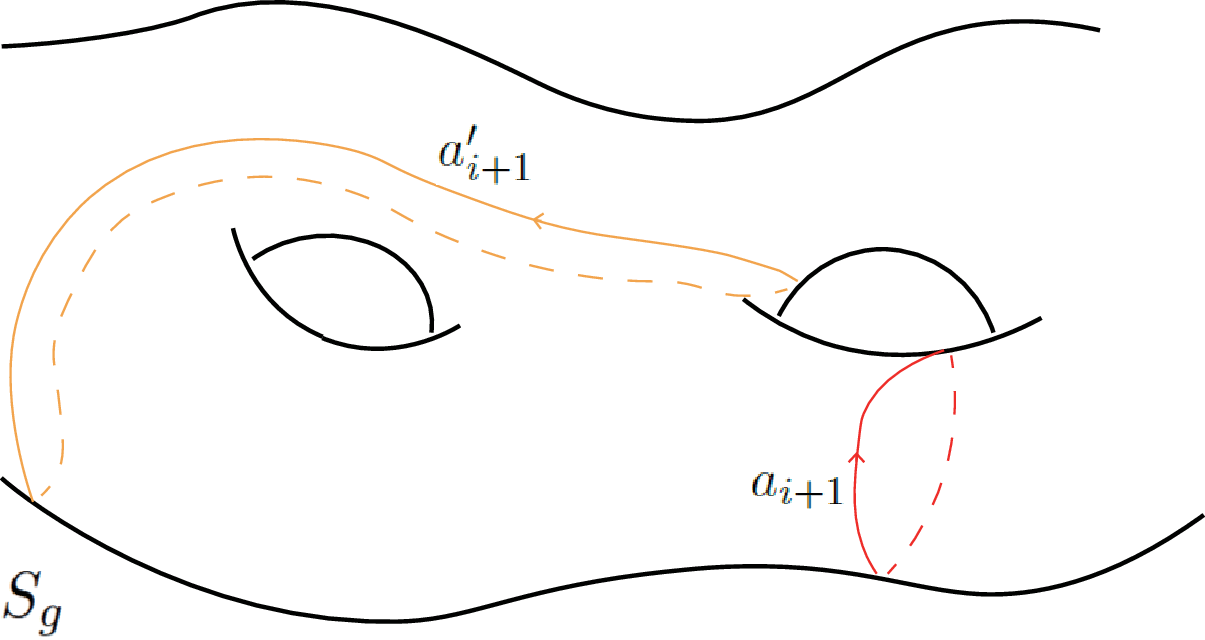}
    \caption{The curves curves $d,d'$ are homotopic to the curves $a_{i+1},a'_{i+1}$}
    \label{f.isotopydi}
\end{figure}

The curves $a_{i+1}$ and $a_{i+1}'$ bound a subsurface (a twice punctured torus), hence in $\Sp(\rm{H}_1(S_g, \ZZ))$
$$\bar{\tau}_{a_{i+1}} = \bar{\tau}_{a'_{i+1}}.$$
As a consequence,
$$\left( \bar{\tau}_{c_i}^2 \bar{\tau}_{b_i} \bar{\tau}_{a_i} \right)^3 = \bar{\tau}_{a_{i+1}}^2,$$
or equivalently,
$$
\bar{\tau}_{a_{i+1}}^2\left(\bar{\tau}_{a_i}^{-1}\bar{\tau}_{b_i}^{-1}\bar{\tau}_{c_i}^{-2}\right)^2\bar{\tau}_{a_i}^{-1}\bar{\tau}_{b_i}^{-1}
 = \bar{\tau}_{c_i}^{2}
$$
which implies that $\bar{\tau}_{c_i}^{2} \in \Gamma$.

So we have proved that $\bar{\tau}_{c_i}^{2} \in \Gamma$ for $i=1,\dots,g-1$. It follows that $\Gamma$ is actually the subgroup of $\Sp(\rm{H}_1(S_g, \ZZ))$ generated by $\bar{\tau}_{a_1}^{\pm 1}, \dots, \bar{\tau}_{a_g}^{\pm 1}, \bar{\tau}_{b_1}^{\pm 1}, \dots, \bar{\tau}_{b_g}^{\pm 1}, \bar{\tau}_{c_1}^{\pm 2}, \dots, \bar{\tau}_{c_{g-1}}^{\pm 2}.$
\end{proof}

\subsection{$\Gamma$ has finite index}

Recall that, for $q\geq 2$, the principal congruence subgroup of level $q$ of $\Sp(\rm{H}_1(S_g, \ZZ))$, which will be denoted by $\Gamma(q)$, is by definition the kernel of the reduction modulo $p$ homomorphism $\Sp(\rm{H}_1(S_g, \ZZ)) \to \Sp(\rm{H}_1(S_g, \ZZ/q\ZZ))$. We will prove the following 

\begin{proposition}\label{proposition:finiteindex}
    $\Gamma$ contains the principal congruence subgroup of level $2^{2g-2}$ of $\Sp(\rm{H}_1(S_g,\ZZ))$. In particular, $\Gamma$ has  finite index in $\Sp(\rm{H}_1(S_g,\ZZ))$.
\end{proposition}

A straightforward counting argument shows that the index of the principal congruence subgroup $\Gamma(q)$ of $\Sp(\rm{H}_1(S_g, \ZZ))$ satisfies 
$$[\Sp(\rm{H}_1(S_g, \ZZ) : \Gamma(q) ] = |\Sp(\rm{H}_1(S_g, \ZZ/q\ZZ))| = |\Sp(2g, \ZZ/q\ZZ))| \leq |M_{2g}(\ZZ/q\ZZ)| = q^{4g^2}.$$
Together with Proposition~\ref{proposition:finiteindex}, this formula yields the following upper bound for the index of our subgroup $\Gamma$
$$[\Sp(\rm{H}_1(S_g, \ZZ) : \Gamma ] \leq 2^{8g^3}.$$
Note that this bound is quite large even when $g$ is very small (about ten billion for $g=2$). More careful estimates of the index of the principal congruence subgroup $\Gamma(2^{2g-2})$ yield a better bound $2^{7g^2}$, but in any case, the actual value of the index of $\Gamma$ is probably much smaller than any bound given by this type of argument.

\begin{proof}[Proof of Proposition~\ref{proposition:finiteindex}]
We will denote by $(e_1,\dots,e_{2g})$ the canonical basis of $\ZZ^{2g}$, and denote by $E_{i,j}$ the matrix in $M_{2g,2g}(\ZZ)$ with a 1 at position $(i,j)$ and 0s elsewhere. We choose $([a_1], \ldots, [a_g], [b_1], \ldots, [b_g])$ as a basis of $\rm{H}_1(S_g, \ZZ)$ (the order of the elements in the basis matters). This basis provides an identification of $\rm{H}_1(S_g, \ZZ)$ with $\ZZ^{2g}$. Observe that: 
\begin{itemize}
    \item  the homology class $[a_i]$ is identified with the vector $e_i$ and the homology class $[b_i]$ is identified with the vector $e_{g+i}$; 
    \item the intersection form on $\rm{H}_1(S_g, \ZZ)$ is identified with the symplectic form $$\sum_{i=1}^g [e_i]^{*} \wedge [e_{g+i}]^{*}$$ on $\ZZ^{2g}$. The matrix of this symplectic form is $$J = \begin{pmatrix} 0 & I_g \\ -I_g & 0 \end{pmatrix};$$
    \item the group $\Sp(\rm{H}_1(S_g, \ZZ))$ is identified  with the group  
    $$\Sp(2g, \ZZ) = \{ M \in \operatorname{GL}(2g, \ZZ) \mid M^T J M = J \};$$
    \item for any $q\geq 2$, the principal congruence subgroup of level $q$ of $\Sp(\rm{H}_1(S_g, \ZZ))$ is identified with the subgroup of $\Sp(2g, \ZZ)$
    $$\{ M \in \Sp(2g, \ZZ) \mid M \equiv I_{2g} \pmod{q} \}$$
    \item the matrices of the Dehn twists $\bar{\tau}_{a_i}$, $\bar{\tau}_{b_i}$, $\bar{\tau}_{c_i}$ with respect to the basis $([a_1], \ldots, [a_g], [b_1], \ldots, [b_g])$ are respectively
\begin{eqnarray*}
A_i & := & I_{2g} + E_{i,g+i}, \\
B_i &  := & I_{2g} - E_{g+i,i}, \\
C_{i} & := & I_{2g} - E_{i,g+i} - E_{i+1,g+i+1} + E_{i+1,g+i} + E_{i,g+i+1}.
\end{eqnarray*}
Since $\Gamma$ is generated by the Dehn twists $\bar{\tau}_{a_1}^{\pm 1}, \dots, \bar{\tau}_{a_g}^{\pm 1}$, $\bar{\tau}_{b_1}^{\pm 1}, \dots, \bar{\tau}_{b_g}^{\pm 1}$ and $\bar{\tau}_{c_1}^{\pm 2}, \dots, \bar{\tau}_{c_{g-1}}^{\pm 2}$, it follows that $\Gamma$ is identified with the subgroup of $\Sp(2g, \ZZ)$ generated by the matrices 
$$A_1^{\pm 1},\dots,A_g^{\pm 1},\; B_1^{\pm 1},\dots,B_g^{\pm 1}\mbox{ and } C_1^{\pm 2},\dots,C_{g-1}^{\pm 2}$$ (notice the power $\pm 2$ on the $C_i$'s ; it is the source of all our troubles).
\end{itemize}

\medskip

We will use a particular case of a general result proved by Tits in \cite{Tits}, relying on some previous work by Bass, Milnor and Serre (\cite{BassMilnorSerre}). We will first quote faithfully Tits' general statement, before explaining how it can be particularized to apply to our situation. Readers unfamiliar with the language of algebraic group schemes need only understand the particularization to our setting (Corollary~\ref{corollary:Tits}), without necessarily fully understanding the statement of Proposition~\ref{proposition:Tits}. 

For every ring $B$, every ideal $\mathfrak{p}$ of $B$, and every group scheme $X$ on $B$, one denotes by $X^{(\mathfrak{p})}$ the ``congruence subgroup" obtained as the kernel of the reduction modulo $\mathfrak{p}$ morphism $X(B)\to X(B/\mathfrak{p})$. 

\begin{proposition}[Proposition 4 of~\cite{Tits}]
\label{proposition:Tits}
Let $A$ be a Dedekind domain of arithmetic type and $\mathfrak{q}$ be a non-zero ideal of $A$. Let $G$ be an almost simple Chevalley scheme of rank $\geq 2$ and $\Phi$ be a root system of $G$. For $\alpha\in\Phi$, denote by $U_\alpha$ the one-parameter subgroup of $G(A)$ associated with $\alpha$, and recall that $U_\alpha^{(\mathfrak{q})}$ is the kernel of the reduction modulo $\mathfrak{q}$ of $U_\alpha$. Denote by $F(\mathfrak{q})$ the subgroup of $G(A)$ generated by the $U_\alpha^{(\mathfrak{q})}$'s when $\alpha$ ranges over $\Phi$.   

Then $F(\mathfrak{q})$ has finite index in $G(A)$. Moreover, if $A$ is not the ring of integers of a totally imaginary field and $G$ is simply connected, then $F(\mathfrak{q})$ is the subgroup of $G(A)$ made of the elements whose reduction modulo $\mathfrak{q}^2$ belongs to the (direct) product, for $\alpha\in\Phi$, of the subgroups $U_\alpha^{(\mathfrak{q}/\mathfrak{q}^2)}$ of $G(A/\mathfrak{q}^2)$.
\end{proposition}

Note that the last sentence implies in particular that $F(\mathfrak{q})$ contains the congruence subgroup $G^{(\mathfrak{q}^2)}$. We apply Proposition~\ref{proposition:Tits} in the following particular setting.
\begin{itemize}
\item The ring $A$ will be $\mathbb{Z}$, which is indeed a Dedekind domain of arithmetic type (see \emph{e.g}~\cite[page 83]{BassMilnorSerre}), and is not the ring of integers of  totally imaginary.
\item The ideal $\mathfrak{q}$ will be $2^{g-1}\ZZ$. 
\item The group scheme $G$ will be $\mathrm{Sp}(2g)$, which is indeed an almost simple, simply connected, Chevalley scheme of rank $g$ (see \emph{e.g.} \cite[chapitre 22]{Chevalley} or \cite[paragraph 18.25]{Milne})  
\item Hence the group $G(A)$ will be $\mathrm{Sp}(2g,\ZZ)$ and the congruence subgroup $G^{(\mathfrak{q}^2)}$ will be the principal congruence subgroup $\Gamma(2^{2g-2})$ of level $2^{2g-2}$ of $\mathrm{Sp}(2g,\ZZ)$.
\item The one-parameter subgroups of $\mathrm{Sp}(2g,\ZZ)$ associated to the elements of the classical root system $\mathrm{Sp}(2g)$ are (see \emph{e.g.} \cite[chapitre 22]{Chevalley} or \cite[paragraph 18.25]{Milne} or \cite[Page 110]{Rossmann}):
\begin{eqnarray*}
&& \left\{V_i^t= I_{2g}+tE_{i,g+i}\right\}_{t\in\ZZ},  \text{ for } 1 \leq i \leq g,  \\
&& \left\{W_i^t=I_{2g}+tE_{g+i,i})\right\}_{t\in\ZZ}   \text{ for } 1 \leq i \leq g, \\
&& \left\{X_{j,k}^t= I_{2g}+t(E_{j,k} - E_{g+k,g+j})\right\}_{t\in\ZZ}   \text{ for } 1 \leq j \neq k \leq g, \\
&& \left\{Y_{j,k}^{t} = I_{2g} + t (E_{g+j,k} + E_{g+k,j})\right\}_{t\in\ZZ}   \text{ for } 1 \leq j < k \leq g, \\
&& \left\{Z_{j,k}^{t} = I_{2g} + t (E_{j,g+k} + E_{k,g+j})\right\}_{t\in\ZZ}  \quad \text{ for } 1 \leq j < k \leq g.
\end{eqnarray*}
Hence, the subgroup $F(\mathfrak{q})$ considered in Proposition~\ref{proposition:Tits} will be, in our setting, the subgroup of $\mathrm{Sp}(2g,\ZZ)$ generated by the matrices 
\begin{eqnarray*}
&& V_i^{2^{g-1}} = I_{2g}+2^{g-1}E_{i,g+i}  \text{ for } 1 \leq i \leq g,  \\
&& W_i^{2^{g-1}} = I_{2g}+2^{g-1}E_{g+i,i}   \text{ for } 1 \leq i \leq g, \\
&& X_{j,k}^{2^{g-1}}= I_{2g}+2^{g-1}(E_{j,k} - E_{g+k,g+j}) \text{ for } 1 \leq j \neq k \leq g, \\
&& Y_{j,k}^{2^{g-1}} = I_{2g} + 2^{g-1} (E_{g+j,k} + E_{g+k,j})   \text{ for } 1 \leq j < k \leq g, \\
&& Z_{j,k}^{2^{g-1}} =  I_{2g} + 2^{g-1} (E_{j,g+k} + E_{k,g+j})  \quad \text{ for } 1 \leq j < k \leq g.
\end{eqnarray*} 
\end{itemize}
In this specific setting, Proposition~\ref{proposition:Tits} particularizes as follows:

\begin{corollary}
\label{corollary:Tits} 
The subgroup of $\mathrm{Sp}(2g,\ZZ)$ generated by the matrices $V_i^{2^{g-1}}$ and $W_i^{2^{g-1}}$ for \hbox{$1\leq i\leq g$,} the matrices $X_{j,k}^{2^{g-1}}$ for $1\leq j\neq k\leq g$ and the matrices $Y_{j,k}^{2^{g-1}}$ and $Z_{j,k}^{2^{g-1}}$ for $1\leq j< k\leq g$  contains the principal congruence subgroup $\Gamma(2^{2g-2})$. 
\end{corollary}

In view of Corollary~\ref{corollary:Tits}, we are left to prove that:

\begin{lemma}
\label{lemma:product-of-generators}
The matrices $V_i^{2^{g-1}}$ and $W_i^{2^{g-1}}$ for \hbox{$1\leq i\leq g$,} the matrices $X_{j,k}^{2^{g-1}}$ for $1\leq j\neq k\leq g$ and the matrices $Y_{j,k}^{2^{g-1}}$ and $Z_{j,k}^{2^{g-1}}$ for $1\leq j< k\leq g$ belong to our subgroup $\Gamma$, \emph{i.e.} that all these matrices can be written as products of the generators $A_1^{\pm 1},\dots,A_g^{\pm 1}$, $B_1^{\pm 1},\dots,B_g^{\pm 1}$, $C_1^{\pm 2},\dots,C_{g-1}^{\pm 2}$ of $\Gamma$.
\end{lemma}

\begin{proof}[Proof of Lemma~\ref{lemma:product-of-generators}]
The case of the matrices $V^{2^{g-1}}_i$ and $W^{2^{g-1}}_i$ is immediate, as pointed by the following observation:

\medskip

\noindent \textbf{Observation A.} \textit{For $i=1,\dots,g$, we have $V_i = A_i$ and $W_i = B^{-1}_i$.}

\medskip

In order to deal with the case of the matrix $X^{2^{g-1}}_{j,k}$, we introduce, for $i=1,\dots,g-1$, the matrix 
$$D_i := A^2_i B^2_{i+1} (A_{i+1} B_{i+1} A_{i+1}) C_i^2 (A_{i+1} B_{i+1} A_{i+1})^{-1}.$$ 
We first prove 

\medskip

\noindent \textbf{Claim B.} For $i=1,\dots, g-1$, we have 
$$D_i = I_{2g}+ 2 (\Emat{i}{i+1} - \Emat{g+i+1}{g+i}) = X_{i,i+1}^{2}.$$ 

\medskip

\begin{proof}
The equality $D_i = I_{2g}+ 2 (\Emat{i}{i+1} - \Emat{g+i+1}{g+i})$ can be obtained by brute force, considering the definition of the matrices $A_i$, $A_{i+1}$, $B_i$, $B_{i+1}$ and $C_i$, and making the product to obtain $D_i$. Yet we shall try to provide a slightly less obscure explanation below. 
\begin{itemize}
\item Let us first notice that $D_i$ preserves each element of our basis except maybe $e_i, e_{i+1}, e_{g+i}$ and $e_{g+i+1}$, since this is the case for $A_i$, $A_{i+1}^{\pm 1}$, $B_{i+1}^{\pm 1}$, $C_i$.  
\item A crucial observation is that $A_{i+1} B_{i+1} A_{i+1}$ is the rotation of angle $\pi/2$ in the plane $\langle e_i, e_{g+i} \rangle$. Hence $A_{i+1} B_{i+1} A_{i+1}$ maps $e_{i+1}$ to $-e_{g+i+1}$, maps $e_{g+i+1}$ to $e_{i+1}$, and preserves $e_i$ and $e_{g+i}$.
\item $C_i^2$ maps $e_{g+i}$ to $e_{g+i} - 2e_i + 2e_{i+1}$, maps $e_{g+i+1}$ to $e_{g+i+1} + 2e_{i} - 2e_{i+1}$, and preserves $e_i$ and $e_{i+1}$.
\item $B^{2}_{i+1}$ maps $e_{i+1}$ to $e_{i+1} - 2e_{g+i+1}$.  
\item $A^2_i$ maps $e_{g+i}$ to $e_{g+i} + 2e_i$.  
\end{itemize}
Hence  
\begin{eqnarray*}
&& e_i \xmapsto{(A_{i+1}B_{i+1}A_{i+1})^{-1}} e_i \xmapsto{C_i^2} e_i\xmapsto{A_{i+1}B_{i+1}A_{i+1}} e_i \xmapsto{B^2_{i+1}} e_i \xmapsto{A^2_{i}} e_i, \\
&& e_{g+i}  \xmapsto{(A_{i+1}B_{i+1}A_{i+1})^{-1}} e_{g+i} \xmapsto{C_i^2} e_{g+i} -2e_i +2e_{i+1} \xmapsto{A_{i+1}B_{i+1}A_{i+1}}\dots\\
&& \dots e_{g+i}-2e_i - 2e_{g+i+1} \xmapsto{B^2_{i+1}} e_{g+i}-2e_{i}-2e_{g+i+1} \xmapsto{A^2_{i}} e_{g+i}-2e_{g+i+1},\\
&& e_{i+1}\xmapsto{(A_{i+1}B_{i+1}A_{i+1})^{-1}} e_{g+i+1} \xmapsto{C_i^2} e_{g+i+1} +2e_i-2e_{i+1} \xmapsto{A_{i+1}B_{i+1}A_{i+1}} \dots\\
&& \dots e_{i+1}+2e_{i}+2e_{g+i+1}  \xmapsto{B^2_{i+1}} e_{i+1}+2e_{i} \xmapsto{A^2_{i}} e_{i+1}+2e_{i},\\
&& e_{g+i+1} \xmapsto{(A_{i+1}B_{i+1}A_{i+1})^{-1}} -e_{i+1} \xmapsto{C_i^2} -e_{i+1}\xmapsto{A_{i+1}B_{i+1}A_{i+1}} e_{g+i+1} \xmapsto{B^2_{i+1}}\dots\\
&&\dots  e_{g+i+1} \xmapsto{A^2_{i}} e_{g+i+1}.
\end{eqnarray*}
So $D_i$ maps $e_{i+1}$ to $e_{i+1} + 2e_i$, maps $e_{g+i}$ to $ e_{g+i} - 2e_{g+i+1}$, and preserves other elements of the basis. It follows that $D_i = I_{2g}+ 2 (\Emat{i}{i+1} - \Emat{g+i+1}{g+i}) = X_{i,i+1}^{2}$.
\end{proof}

Now we prove

\medskip

\noindent \textbf{Claim C.} For $1\leq j < j+\ell\leq g-1$, we have 
$$\left [ X_{j,j+\ell}^{2^\ell} , X_{j+\ell,j+\ell+1}^2 \right ] = X_{j,j+\ell+1}^{2^{\ell+1}}.$$

\medskip

\begin{proof}
Recall that 
$$X_{j,j+\ell}^{\pm 2^\ell} = I_{2g} \pm 2^\ell \Emat{j}{j+\ell} \mp 2^\ell \Emat{g+j+\ell}{g+j}$$ 
and
$$X_{j+\ell,j+\ell+1}^{\pm 2} = I_{2g} \pm 2 \Emat{j+\ell}{j+\ell+1} \mp 2 \Emat{g+j+\ell+1}{g+j+\ell}.$$
Hence 
\begin{eqnarray*}
X_{j,j+\ell}^{2^\ell} X_{j+\ell,j+\ell+1}^{2} 
& = & I_{2g} +  2^\ell \Emat{j}{j+\ell} -2^\ell \Emat{g+j+\ell}{g+j} \\
& & + 2 \Emat{j+\ell}{j+\ell+1}  - 2 \Emat{g+j+\ell+1}{g+j+\ell} + 2^{\ell+1} \Emat{j}{j+\ell+1}
\end{eqnarray*}
hence 
\begin{eqnarray*}
X_{j,j+\ell}^{2^\ell} X_{j+\ell,j+\ell+1}^{2} X_{j,j+\ell}^{-2^\ell} & = &  I_{2g} + 2 \Emat{j+\ell}{j+\ell+1}   - 2 \Emat{g+j+\ell+1}{g+j+\ell} \\
 & & + 2^{\ell+1} \Emat{j}{j+\ell+1} - 2^{\ell+1} \Emat{g+j+\ell+1}{g+j}
\end{eqnarray*}
and finally 
\begin{eqnarray*}
\left [ X_{j,j+\ell}^{2^\ell} , X_{j+\ell,j+\ell+1}^2 \right ]  & = & X_{j,j+\ell}^{2^\ell} X_{j+\ell,j+\ell+1}^{2} X_{j,j+\ell}^{-2^\ell} X_{j+\ell,j+\ell+1}^{-2}\\
& = & I_{2g} - 2^{\ell+1} \Emat{g+j+\ell+1}{g+j} + 2^{\ell+1} \Emat{j}{j+\ell+1} =  X_{j,j+\ell+1}^{2^{\ell+1}}
\end{eqnarray*}
as announced.
\end{proof}

Claims B and C, the definition of the matrix $D_i$ and a trivial induction imply that, for $1 \leq j < k \leq g$, the matrix $X_{j,k}^{2^{k-j}}$ (hence the matrix $X_{j,k}^{2^{g-1}}$) can be written a product of the matrices $A_1^{\pm 1},\dots,A_g^{\pm 1}$, $B_1^{\pm 1},\dots,B_g^{\pm 1}$, $C_1^{\pm 2},\dots,C_{g-1}^{\pm 2}$. The case of $1 \leq k < j \leq g$ can be treated by similar arguments using the following variations on Claims B and C:

\medskip

\noindent \textbf{Claim B'.} For $i=2,\dots, g$, if we denote 
$$D_i' := (A_{i-1} B_{i-1} A_{i-1})C_{i-1}^{-2}(A_{i-1} B_{i-1} A_{i-1})^{-1}B^{-2}_{i-1}A^{-2}_i$$
then 
$$D_i' = I_{2g} - 2 (\Emat{i}{i-1} - \Emat{g+i-1}{g+i}) = X_{i,i-1}^{-2}$$ 

\medskip

\noindent \textbf{Claim C'.} For $2\leq j-\ell < j\leq g$, we have 
$$\left [ X_{j,j-\ell}^{2^\ell} , X_{j-\ell,j-(\ell+1)}^2 \right ] = X_{j,j-(\ell+1)}^{2^{\ell+1}}.$$

\medskip

Now we turn to the case of the matrix $Z_{j,k}^{2^{g-1}}$. The following claim deals with the case where $j=k-1$.

\medskip

\noindent \textbf{Claim D.} For $2\leq k \leq g$, we have $$Z_{k-1,k}^{2} = \left[ V_k , D_{k-1}^{-1}\right] V_{k-1}^{4}.$$ 

\medskip

\begin{proof}
Recall that $$V_k^\pm = I_{2g} \pm E_{k,g+k}\quad\mbox{and}\quad D_{k-1}^\pm = X_{k-1,k}^{\pm 2} = I_{2g}\pm 2 E_{k-1,k}\mp 2 E_{g+k,g+k-1}.$$
Hence 
$$V_kD_{k-1}^{-1} = I_{2g} + E_{k,g+k} - 2 E_{k-1,k} + 2 E_{g+k,g+k-1} + 2 E_{k,g+k-1},$$
hence
$$V_kD_{k-1}^{-1}V_k^{-1} = I_{2g} - 2 E_{k-1,k} + 2 E_{g+k,g+k-1} + 2 E_{k,g+k-1} + 2 E_{k-1,g+k},$$
hence 
$$\left[ V_k , D_{k-1}^{-1}\right] = V_kD_{k-1}^{-1}V_k^{-1}D_{k-1} =  I_{2g} + 2 E_{k,g+k-1} + 2 E_{k-1,g+k} - 4 E_{k-1,g+k-1},$$
and finally 
$$\left[ V_k , D_{k-1}^{-1}\right] V_{k-1}^{4} = I_{2g} + 2 E_{k,g+k-1} + 2 E_{k-1,g+k} = Z_{k-1,k}^{2}.$$
\end{proof}

The general case (\emph{i.e.} the case of $Z_{j,k}^{2^{g-1}}$ for $1\leq j<k\leq g$) will follow from Claim~D and from Claim E below: 

\medskip

\noindent \textbf{Claim E.} For $1\leq k-(\ell-1) < k \leq g$, we have $$Z_{k-\ell,k}^{2^\ell} = \left[ Z_{k-(\ell-1),k}^{2^{\ell-1}} , D_{k-\ell}^{-1}\right].$$ 

\medskip

\begin{proof}
Recall that 
$$Z_{k-(\ell-1),k}^{\pm 2^{\ell-1}} = I_{2g} \pm 2^{\ell-1} E_{k-(\ell-1),g+k} \pm 2^{\ell-1} E_{k,g+k+(\ell-1)}$$ 
and 
$$D_{k-\ell}^{\pm 1} = X_{k-\ell,k-(\ell-1)}^{\pm 2} = I_{2g}\pm 2 E_{k-\ell,k-(\ell-1)}\mp 2 E_{g+k-(\ell-1),g+k-\ell}.$$
Hence 
\begin{eqnarray*}
Z_{k-(\ell-1),k}^{2^{\ell-1}}D_{k-\ell}^{-1}
& = & I_{2g} + 2^{\ell-1} E_{k-(\ell-1),g+k} + 2^{\ell-1} E_{k,g+k+(\ell-1)} \\
& & - 2 E_{k-\ell,k-(\ell-1)} + 2 E_{g+k-(\ell-1),g+k-\ell} + 2^\ell E_{k,g+k-\ell} ,
\end{eqnarray*}
hence 
\begin{eqnarray*}
&& Z_{k-(\ell-1),k}^{2^{\ell-1}}D_{k-\ell}^{-1}Z_{k-(\ell-1),k}^{-2^{\ell-1}} \\
& = & I_{2g} - 2 E_{k-\ell,k-(\ell-1)} + 2 E_{g+k-(\ell-1),g+k-\ell} + 2^\ell E_{k,g+k-\ell} + 2^\ell E_{k-\ell,g+k},
\end{eqnarray*}
hence 
\begin{eqnarray*}
\left[ Z_{k-(\ell-1),k} , D_{k-\ell}^{-1}\right] & = & Z_{k-(\ell-1),k}^{2^{\ell-1}}D_{k-\ell}^{-1}Z_{k-(\ell-1),k}^{-2^{\ell-1}}D_{k-\ell} \\
& = & I_{2g} + 2^\ell E_{k,g+k-\ell} + 2^\ell E_{k-\ell,g+k} = Z_{k-\ell,k}^{2^\ell}.
\end{eqnarray*}
\end{proof}

Claims A, D and E, the definition of the matrix $D_{k-\ell}$ and a straighforward induction over $\ell$ imply that, for $1 \leq j < k \leq g$, the matrix $Z_{j,k}^{2^{k-j}}$ (hence the matrix $Z_{j,k}^{2^{g-1}}$ which is a power of $Z_{j,k}^{2^{k-j}}$) can be written a product of the matrices $A_1^{\pm 1},\dots,A_g^{\pm 1}$, $B_1^{\pm 1},\dots,B_g^{\pm 1}$ and $C_1^{\pm 2},\dots,C_{g-1}^{\pm 2}$. 

Finally, the case of the matrix $Y_{j,k}^{2^{g-1}}$ will be obtained using the following observation: 

\medskip

\noindent \textbf{Observation F.} For $1\leq j < k \leq g$, the matrix $Y_{j,k}$ is the conjugate of the matrix $Z_{j,k}^{-1}$ by the matrix $\prod_{i=1}^g A_i B_i A_i$.

\medskip

\begin{proof}
We have already observed that, for $i=1,\dots,g$, the matrix $A_i B_i A_i$ maps $e_{i}$ to $-e_{g+i}$, maps $e_{g+i}$ to $e_{i}$, and preserves the other elements of the basis. As a consequence, the matrix $\prod_{i=1}^g A_i B_i A_i$ maps $e_{i}$ to $-e_{g+i}$ and  maps $e_{g+i}$ to $e_{i}$ for every $i\in \{1,\dots,g\}$\footnote{This means that $\prod_{i=1}^g A_i B_i A_i$ is actually equal to the matrix $J$ of the symplectic form.}. As a consequence, the conjugate of the matrix $Z_{j,k}^{-1} =  I_{2g} - E_{j,g+k} - E_{k,g+j}$ under the matrix $\prod_{i=1}^g A_i B_i A_i$ is the matrix $I_{2g} + E_{g+j,k} + E_{g+k,j}=Y_{j,k}$.
\end{proof}

Since we already know that, for $1 \leq j < k \leq g$, the matrix $Z_{j,k}^{2^{g-1}}$ can be written a product of the generators $A_1^{\pm 1},\dots,A_g^{\pm 1}$, $B_1^{\pm 1},\dots,B_g^{\pm 1}$ and $C_1^{\pm 2},\dots,C_{g-1}^{\pm 2}$, we deduce from Observation~F that this is also the case for the matrix $Y_{j,k}^{2^{g-1}}$. This completes the proof of lemma~\ref{lemma:product-of-generators}.
\end{proof}

Lemma~\ref{lemma:product-of-generators} and  Corollary~\ref{corollary:Tits} imply that our subgroup $\Gamma$ contains the principal congruence subgroup of level $2^{2g-2}$ of $\mathrm{Sp}(H_1(S_g,\ZZ))$ concluding the proof of Proposition~\ref{proposition:finiteindex}.
\end{proof}

Of course, Proposition~\ref{proposition:subgroup-finiteindex} follows by combining Propositions~\ref{proposition:subgroup} and~\ref{proposition:finiteindex}.

\newpage
\bibliographystyle{plain}
\bibliography{bibli}

\newpage 

\noindent François Béguin 

\noindent{\small Université Sorbonne Paris Nord} 

\noindent{\small Laboratoire Analyse Géométrie et Applications, CNRS, UMR 7539} 

\noindent{\small F‐93430, Villetaneuse, France}

\noindent{\footnotesize{E-mail: beguin@math.univ-paris13.fr }}\\

\noindent Christian Bonatti

\noindent{\small Université Bourgogne Europe}

\noindent {\small CNRS, IMB (UMR 5584)}

\noindent {\small 21000 Dijon, France}

\noindent{\footnotesize{E-mail: bonatti@u-bourgogne.fr }}\\

\noindent Biao Ma

\noindent {\small School of Mathematical Sciences}

\noindent {\small Key Laboratory of Intelligent Computing and
Applications (Tongji University), Ministry of Education}

\noindent{\small Tongji University, Shanghai 200092, China}

\noindent{\footnotesize{E-mail: 24024@tongji.edu.cn }}\\

\noindent Bin Yu

\noindent {\small School of Mathematical Sciences}

\noindent {\small Key Laboratory of Intelligent Computing and
Applications (Tongji University), Ministry of Education}

\noindent{\small Tongji University, Shanghai 200092, China}

\noindent{\footnotesize{E-mail: binyu1980@gmail.com }}

\end{document}